\documentclass{amsart}

\usepackage[T1]{fontenc}
\usepackage{lmodern}

\usepackage{amscd, amssymb, amsmath, amsthm}
\usepackage{stmaryrd} % for double brackets 
\usepackage{enumerate, latexsym, mathrsfs}
\usepackage{xypic}
\usepackage[pdftex]{geometry}
\usepackage[
		bookmarks=true, bookmarksopen=true,%
    bookmarksdepth=3,bookmarksopenlevel=2,%
    colorlinks=true,%
    linkcolor=blue,%
    citecolor=blue]{hyperref}
\newtheorem{theorem}{Theorem}[section]
\newtheorem{lemma}[theorem]{Lemma}
\newtheorem{proposition}[theorem]{Proposition}
\newtheorem{corollary}[theorem]{Corollary}

\theoremstyle{definition}
\newtheorem{definition}[theorem]{Definition}
\newtheorem{convention}[theorem]{Convention}

\newtheorem{remark}[theorem]{Remark}

\numberwithin{equation}{section}

\newcommand{\Natural}{{\mathbb N}}
\newcommand{\Real}{{\mathbb R}}
\newcommand{\Rational}{{\mathbb Q}}
\newcommand{\Complex}{{\mathbb C}}
\newcommand{\Integral}{{\mathbb Z}}
\newcommand{\Field}{{\mathbb F}}

\newcommand{\Sph}{{\mathbb S}}

\newcommand{{\FPS}}{{\mathrm{Fix}}}
\newcommand{{\PPS}}{{\mathrm{Per}}}
\newcommand{{\FPC}}{{\mathscr{F}\mathrm{ix}}}
\newcommand{{\PPC}}{{\mathscr{P}\mathrm{er}}}
\newcommand{{\POC}}{{\mathscr{O}\mathrm{rb}}}

\newcommand{{\periodic}}{{\mathtt{per}}}
\newcommand{{\pA}}{{\mathtt{pA}}}
\newcommand{{\reduction}}{{\mathtt{fDt}}}
\newcommand{{\NT}}{\mathtt{NT}}
\newcommand{{\gd}}{{\mathtt{gd}}}

\title[{hyperbolic 3-manifolds almost determined by finite quotients}]{
Finite-volume hyperbolic 3-manifolds\\ are almost determined\\ by their finite quotient groups}
     
\author[Yi Liu]{%
        Yi Liu} 
\address{%
        Beijing International Center for Mathematical Research, Peking University\\
				Beijing 100871, China P.R.} 
\email{% 
    liuyi@bicmr.pku.edu.cn}

\thanks{Partially supported by NSFC grant 11925101, and National Key R\&D Program of China 2020YFA0712800}

\date{% 
 \today} 

\begin{document}

\begin{abstract}
	For any orientable finite-volume hyperbolic 3-manifold,
	this paper proves that
	the profinite isomorphism type of the fundamental group
	uniquely determines
	the isomorphism type of the first integral cohomology,
	as marked with the Thurston norm and the fibered classes;
	moreover, up to finite ambiguity,
	the profinite isomorphism type
	determines the isomorphism type of the fundamental group,
	among the class of finitely generated 3-manifold groups.
\end{abstract}

\maketitle

\section{Introduction}\label{Sec-introduction}
A pair of finitely generated groups $G_A$ and $G_B$ are said to be \emph{profinitely isomorphic},
if every finite quotient group of $G_A$ is isomorphic to a finite quotient group of $G_B$,
and vice versa.
%Recall that the \emph{profinite completion} of a group $G$ is the group $\widehat{G}$
%obtained as the inverse limit of all the finite quotient groups of $G$.
%It comes together with a natural homomorphism $G\to \widehat{G}$, 
%which is injective if $G$ is residually finite.
It is known that being profinitely isomorphic is equivalent to 
having isomorphic profinite completions.
For any class of finitely generated residually finite groups $\mathscr{C}$, 
%such as finitely generated residually finite groups,
%or fundamental groups of compact $3$--manifolds,
%a \emph{profinite property} among $\mathscr{C}$ refers to a property on groups of $\mathscr{C}$,
%such that a group of $\mathscr{C}$ satisfies the property if and only if 
%any other profinitely isomorphic group of $\mathscr{C}$ also satisfies the property.
a group $G$ in $\mathscr{C}$ is said to be \emph{profinitely rigid} among $\mathscr{C}$
if any group in $\mathscr{C}$ that is profinitely isomorphic to $G$ is isomorphic to $G$.
We say that $G$ is \emph{profinitely almost rigid} among $\mathscr{C}$,
if there exist finitely many groups in $\mathscr{C}$,
such that any group in $\mathscr{C}$ that is profinitely isomorphic to $G$
is isomorphic to one of those groups.

For example, among all finitely generated residually finite groups,
it is easy to see that finite groups and abelian groups are profinitely rigid;
nilpotent groups are profinitely almost rigid,
but not always profinitely rigid \cite{Pickel_nilpotent,Pickel_nilpotent-by-finite,Remeslennikov};
and there exist metabelian groups which are not profinitely almost rigid \cite{Pickel_metabelian}.
More recently, M.~R.~Bridson, D.~B.~McReynolds, A.~W.~Reid, and R.~Spitler
have proved the profinite rigidity of
several arithmetic lattices of small covolume in $\mathrm{PSL}(2,\Complex)$, or in $\mathrm{O}^+(3,1)$,
\cite{BMRS}.
%including the Bianchi group $\mathrm{PSL}(2,\Integral[\omega])$ where $\omega$ is the cubic root of unity,
%the non-uniform lattice of the smallest covolume, and 
%the uniform torsion-free lattice of the smallest covolume
%(namely, the Weeks manifold group), \cite{BMRS}.
See \cite{Aka_Kazhdan,Aka_finiteness,BCR_Fuchsian} for more examples
of profinite rigidity or profinite almost rigidity
among various classes of lattices in Lie groups.
In the literature, 
the profinite isomorphism class is also called the \emph{genus} of the group,
which reflects the original analogy of this notion with 
the genus of an integral quadratic form.
We refer the reader to Reid's beautiful survey \cite{Reid_survey} 
for an overview of this topic.

Among finitely generated $3$--manifold groups,
%(or equivalently, fundamental groups of connected compact $3$--manifolds,)
the fundamental group of any once-punctured torus bundle
over a circle is profinitely rigid \cite{BRW};
see also \cite{Bridson--Reid} for the special case with the figure-eight knot group.
The Weeks manifold provides a closed example with profinite rigidity \cite{BMRS}.
However, there exist Anosov torus bundles and periodic closed surface bundles
with non-isomorphic but profinitely isomorphic fundamental groups
\cite{Funar_torus_bundles,Hempel_quotient,Stebe_integer_matrices}.
In a series of works,
G.~Wilkes shows that orientable closed Seifert fiber spaces and graph manifolds
are profinitely almost rigid, and in fact,
any profinitely isomorphic pair of those manifolds are commensurable
\cite{Wilkes_sf,Wilkes_graph,Wilkes_graph_II}.

For any orientable connected closed $3$--manifold,
H.~Wilton and P.~A.~Zalesskii show that 
the profinite isomorphism type of the fundamental group
determines the profinite isomorphism types
of the prime factors and their geometric pieces,
and also determines 
the geometric decomposition graphs together with the kinds of vertex geometries
\cite{WZ_graph_manifolds,WZ_geometry,WZ_decomposition}.
Therefore,
in view of the geometric decomposition,
it remains the most interesting to understand 
orientable finite-volume hyperbolic $3$--manifolds,
or more specifically,
how much the profinite completion tells about their groups.

The goal of this paper is to prove the following main result.
To emphasize,
throught this paper, we talk about finite hyperbolic $3$--manifolds
referring only to those \emph{complete} ones.

\begin{theorem}\label{main_profinite_almost_rigidity}
	%Among the fundamental groups of connected compact $3$--manifolds,
	%%Among $3$--manifold groups,
	%the fundamental group of any finite-volume hyperbolic $3$--manifold
	%is profinitely almost rigid.
	Among finitely generated $3$--manifold groups,
	every finite-volume hyperbolic $3$--manifold group is profinitely almost rigid.
\end{theorem}

%\begin{corollary}\label{main_corollary_lattices}
	%Among all lattices in $\mathrm{PSL}(2,\Complex)$,
	%every lattice is profinitely almost rigid.
%\end{corollary}

An important step toward Theorem \ref{main_profinite_almost_rigidity} is 
to show, for finite-volume hyperbolic 3-manifolds,
that profinite group isomorphisms
are quite restrictive on the profinite free abelianization level.
To be precise, let $\pi_A$ and $\pi_B$ be 
a pair of finitely generated residually finite groups,
and denote by $H_A$ and $H_B$ their abelianizations modulo torsion, respectively.
Denote by $\widehat{\Integral}$ the commutative ring of profinite integers,
and $\widehat{\Integral}^\times$ its multiplicative group of units.
We say that an isomorphism $\Psi\colon\widehat{\pi_A}\to \widehat{\pi_B}$
between the profinite completions
is \emph{$\widehat{\Integral}^\times$--regular}, 
if the induced $\widehat{\Integral}$--linear isomorphism
$\Psi_*\colon \widehat{H_A}\to \widehat{H_B}$ is 
the profinite completion of some isomorphism $H_A\to H_B$ 
composed with the scalar multiplication by some unit $\mu\in\widehat{\Integral}^\times$.
When $\mu$ happens to be $\pm1$, 
being $\widehat{\Integral}^\times$--regular 
is the equivalent to being \emph{regular} modulo torsion,
where the latter condition was introduced 
by M.~Boileau and S.~Friedl in \cite{Boileau--Friedl_fiberedness}.

For orientable aspherical compact $3$--manifolds with empty or tori boundary,
Boileau and Friedl show that regular profinite group isomorphisms canonically determine
linear isomorphisms on the first integral cohomology,
preserving the Thurston norm and the fibered classes
\cite{Boileau--Friedl_fiberedness}.
In general, 
isomorphisms between profinite $3$--manifold groups do not have to be regular.
Irregular examples include those in Hempel's pairs \cite{Hempel_quotient},
and almost all automorphisms of the profinite $3$--torus group.
For finite-volume hyperbolic $3$--manifold groups,
one might suspect that all profinite isomorphisms are regular,
but that conjecture appears hard to approach.
What we are able to show in the present paper
is the $\widehat{\Integral}^\times$--regularity in the finite-volume hyperbolic case,
which still leads to a profinite invariance result
regarding the Thurston-norm unit balls and the fibered faces.

\begin{theorem}\label{main_xregular}
	For any pair of finite-volume hyperbolic $3$--manifolds,
	every isomorphism between their profinite group completions is $\widehat{\Integral}^\times$--regular.
\end{theorem}

\begin{theorem}\label{main_Thurston_norm}
	For any pair of orientable finite-volume hyperbolic $3$--manifolds,
	every isomorphism between their profinite group completions 
	determines a linear isomorphism on the first integral cohomology,
	canonically up to the central symmetric involution,
	such that the linear isomorphism and its inverse 
	both preserve the Thurston norm and the fibered classes.
\end{theorem}

In particular, Theorem \ref{main_Thurston_norm} addresses 
Reid's Question 7.3 in \cite{Reid_survey} affirmatively.
Theorem \ref{main_profinite_almost_rigidity} provides positive evidence to \cite[Question 9]{Reid_discrete},
as it immediately implies that any profinite isomorphism class of lattices in $\mathrm{PSL}(2,\Complex)$ 
is the disjoint union of at most finitely many conjugacy classes.
It seems that Theorems \ref{main_profinite_almost_rigidity} and \ref{main_xregular}
can be combined with former results of Wilkes and Wilton--Zalesskii,
yielding profinite almost rigidity for more general classes of $3$--manifolds,
such as aspherical compact $3$--manifolds of zero Euler characteristic.
However, the techniques of this paper have limitation in telling
if profinitely isomorphic finite-volume hyperbolic $3$--manifolds
should have a common finite cover, or equal volume, or same chirality.

The proof of Theorems \ref{main_xregular} and \ref{main_Thurston_norm}
is concluded in Section \ref{Sec-TN_unit_ball}.
The proof of Theorem \ref{main_profinite_almost_rigidity}
is concluded in Section \ref{Sec-profinite_almost_rigidity}.

\subsection{Ingredients}
The general idea of studying profinite $3$--manifold groups
is to read off topology from the profinite completion.
For instance, we wish to establish 
either the profinite invariance of certain characteristic properties,
or a profinite correspondence of certain characteristic objects.
From that perspective,
our profinite almost rigidity for finite-volume hyperbolic $3$--manifolds 
is proved with three major steps, as follows.
Technical ingredients of those steps are explained next with more detail.
\begin{enumerate}
\item We establish a profinite correspondence of the Thurston-norm cones (Theorem \ref{profinite_isomorphism_npc}),
based on a weaker correspondence of the fibered cones (Theorem \ref{fibered_cone}).
This step actually works for 
nonpositively curved, virtually fibered $3$--manifolds.
\item We establish the $\widehat{\Integral}^\times$--regularity of profinite group isomorphisms (Theorem \ref{main_xregular}),
using more virtual properties that are available for finite-volume hyperbolic $3$--manifolds.
This quickly leads to the profinite invariance of the Thurston norm (Theorem \ref{main_Thurston_norm}),
and also to a useful profinite correspondence of the fiber subgroups (Corollary \ref{fiber_surface_correspondence}).
\item We establish the invariance of the indexed periodic Nielsen numbers,
for profinitely corresponding fibers and their pseudo-Anosov monodromy (Theorem \ref{profinite_invariance_nu}).
This implies the invariance of the pseudo-Anosov stretch factor,
and the profinite almost rigidity (Theorem \ref{main_profinite_almost_rigidity})
follows from classical finiteness results in mapping class group theory.
\end{enumerate}

\subsubsection{Profinite correspondence of Thurston-norm cones}
The key ingredient of this step is inspired from a recent result of A.~Jaikin-Zapirain,
which shows that topologically fibering (as a surface bundle over a circle)
is a profinite property among finitely generated $3$--manifold groups \cite{JZ}.
Taking a closer look, 
we realize that Jaikin-Zapirain's proof
already sets up certain partial correspondence between the fibered cones, 
given a profinite isomorphism between $3$--manifold groups,
(Theorem \ref{fibered_cone} and Corollary \ref{both_fibered}).
For $3$--manifolds satisfying Agol's RFRS criterion for virtual fibering,
we promote the partial correspondence to a bijective correspondence between the Thurston-norm cones,
preserving fibered cones, (Theorem \ref{profinite_isomorphism_npc} and Corollary \ref{fibered_cone_correspondence}).
This is done by applying the partial correspondence to some regular finite cover,
where pullback non-fibered classes all lie on faces of fibered cones.

To motivate the partial correspondence, 
we revisit Jaikin--Zapirain's result as follows:
Given an isomorphism $\Psi\colon\widehat{\pi_A}\to \widehat{\pi_B}$
between profinite completions of $3$--manifold groups,
the induced isomorphism $\Psi_*\colon\widehat{H_A}\to \widehat{H_B}$ 
on the profinite free abelianization level
can be written as a linear combination 
$z_1\Phi_1+\cdots+z_r\Phi_r$, where $\Phi_1,\cdots,\Phi_r\in \mathrm{Hom}_\Integral(H_A,H_B)$ 
are nonzero and where $z_1,\cdots,z_r\in\widehat{\Integral}$ are linearly independent over $\Integral$.
Then, for generic $r$--tuples of integers $c=(c_1,\cdots,c_r)$ in $\Integral^r$,
the homomorphism $\Psi_*^c\colon H_A\to H_B$ defined as $\Psi_*^c=c_1\Phi_1+\cdots+c_r\Phi_r$
is nondegenerate, and it induces 
a dual homomorphism $\Psi^*_c\colon H^1(M_B;\Integral)\to H^1(M_A;\Integral)$
on the first integral cohomology of the $3$--manifolds.
The argument of Jaikin--Zapirain actually implies that 
the $\Psi^*_c$--preimage of any fibered cone for $M_A$
is contained in a fibered cone for $M_B$.
(See Theorem \ref{fibered_cone} with a slightly more general statement,
where $\Psi$ is allowed to be any epimorphism.)
This observation motivates us to introduce a $\Integral$--submodule
$\mathrm{MC}(\Psi_*;H_A,H_B)$ of $\widehat{\Integral}$,
which is contained in the $\Integral$--span of $z_1,\cdots,z_r$,
torsion-free of rank $r$, and uniquely defined.
We call $\mathrm{MC}(\Psi_*;H_A,H_B)$
the \emph{matrix coefficient module} for $\Psi_*$ with respect to $H_A$ and $H_B$,
thinking of $\Psi_*$ abstractly as living in 
$\mathrm{Hom}_\Integral(H_A,H_B)\otimes_\Integral \mathrm{MC}(\Psi_*;H_A,H_B)$,
and thinking of $\Psi_*^c$ as the specialization of $\Psi_*$
under some $\Integral$--homomorphism $\mathrm{MC}(\Psi_*;H_A,H_B)\to \Integral$.
We study more generally specializations $\Psi_*^\varepsilon\colon H_1(M_A;\Real)\to H_1(M_B;\Real)$
and their duals $\Psi^*_\varepsilon\colon H^1(M_B;\Real)\to H^1(M_A;\Real)$,
where $\varepsilon\in \mathrm{Hom}_\Integral(\mathrm{MC}(\Psi_*;H_A,H_B),\Real)$.
See Section \ref{Sec-MC} for the precise definition.
The matrix coefficient module plays a crucial role throughout this paper.
For example, the $\widehat{\Integral}^\times$--regularity condition on $\Psi$
is equivalent to that $\mathrm{MC}(\Psi_*;H_A,H_B)$ has rank one.

\subsubsection{Profinite correspondence of the Thurston-norm unit ball}
For any finite-volume hyperbolic $3$--manifold $M$,
any codimension--$0$ Thurston-norm cone $\mathcal{C}$ in $H^1(M;\Real)$
has a compact cross-section, 
which is an affine polytope and whose projective structure depends only on $\mathcal{C}$.
If $\mathcal{C}$ is a fibered cone, its linear dual cone in $H_1(M;\Real)$
can be characterized in terms of Fried's homology directions.
This relates up fibered cones with the dynamics of 
their associated unparametrized pseudo-Anosov flows on $M$,
(see Section \ref{Subsec-polytopes_and_cones}).
For any isomorphism $\Psi\colon\widehat{\pi_A}\to \widehat{\pi_B}$
between profinite completions of finite-volume hyperbolic $3$--manifolds
$M_A$ and $M_B$,
and any $\Psi^*_\varepsilon$--corresponding pair of 
codimension--$0$ Thurston-norm cones $\mathcal{C}_A$ and $\mathcal{C}_B$,
the dual linear isomorphism $\Psi_*^\varepsilon$ maps
the extreme rays dual to the codimension--$1$ faces of $\mathcal{C}_A$
bijectively onto those similar rays with respect to $\mathcal{C}_B$.
As one may observe from the above description,
those extreme rays already puts many constraints to the structure of $\Psi_*$,
since their correspondence cannot vary under small perturbation of $\varepsilon$.

To prove the $\widehat{\Integral}^\times$--regularity of $\Psi$,
we offer a criterion, 
saying that $\mathrm{MC}(\Psi_*;H_A,H_B)$ has rank one
if $\Psi_*^\varepsilon$ witnesses an extra correspondence between
some pair of rays, in the interior of the linear duals of $\mathcal{C}_A$ and $\mathcal{C}_B$,
and if they stay invariant under small perturbation of $\varepsilon$.
In \cite{Liu_vhsr},
the author develops relevant techniques 
for constructing such interior rays as desired.
In fact,
the rays can be constructed as the projection image of some virtual extreme rays, 
associated to some finite covers of the $3$--manifolds and their fibered cones.
The construction invokes fillings of quasiconvex subgroups 
in virtually special word hyperbolic groups \cite{Wise_book,Wise_notes,AGM-MSQT},
and its currently available version works for closed hyperbolic $3$--manifolds.
We derive the cusped case from the closed case,
using the profinite correspondence of cusp subgroups
due to Wilton and Zalesski \cite{WZ_geometry,WZ_decomposition}.

\subsubsection{Profinite invariance of Nielsen numbers}
In \cite{Liu_procongruent_conjugacy}, the author proves a finiteness result
about mapping classes, which is analogous to the profinite almost rigidity.
In particular, the invariance of the indexed periodic Nielsen numbers 
is proved in that case,
for pairs of mapping classes $f_A,f_B\in\mathrm{Mod}(S)$ of an orientable connected closed surface $S$,
such that the induced profinite group outer automorphisms $[f_A],[f_B]\in\mathrm{Out}(\widehat{\pi_1(S)})$
are conjugate. 
When the mapping classes $f_A$ and $f_B$ are pseudo-Anosov,
one may actually see (with Theorem \ref{main_xregular}) that their mapping torus groups
are regularly profinitely isomorphic.
With the profinite correspondence of fiber subgroups (Corollary \ref{fiber_surface_correspondence}),
the strategy of this step is similar with \cite{Liu_procongruent_conjugacy}:
We first establish the profinite invariance of 
certain twisted Reidemeister torsions (Theorem \ref{profinite_invariance_TRT}),
and then prove the invariance of the indexed periodic Nielsen numbers
using twisted Lefshetz numbers.

Some essential modification from the treatment in \cite{Liu_procongruent_conjugacy} is required,
because we only have $\widehat{\Integral}^\times$--regularity
instead of regularity.
The extra ingredient comes from a recent work of J.~Ueki,
which shows that profinitely isomorphic knot groups have 
identical Alexander polynomials \cite{Ueki}.
Ueki's proof relies on an early result of D.~Fried
about cyclic resultants of reciprocal polynomials \cite{Fried_resultant}.
It also relies on a special fact about knot groups,
that is, the elementary ideal of the Alexander module is principal in $\Integral[t^{\pm1}]$,
 (generated by the Alexander polynomial).
Using some results from \cite{Ueki},
we prove the profinite invariance of twisted Reidemeister torsions
with respect to fibered classes
and $\Rational$--linear representations with finite image.
Note that the $\Rational$--irreducible characters of a finite group
only distinguish all the conjugacy classes of cyclic subgroups.
(By contrast, the $\Complex$--irreducible characters 
distinguish all the conjugacy classes of elements.)
For this reason,
we also must strengthen the argument about Nielsen numbers
in \cite{Liu_procongruent_conjugacy},
using the conjugacy separability of finitely generated $3$--manifold groups 
\cite{HWZ_conjugacy_separability}
and representation theory of finite groups over $\Rational$.

\subsection{Organization}
The organization of this paper is as follows.

In Section \ref{Sec-preliminary}, 
we review profinite completions of finitely generated $3$--manifold groups,
and polytopes and cones associated to the Thurston norm of $3$--manifolds,
and twisted Reidemeister torsions of $3$--manifolds.

In Section \ref{Sec-MC}, we introduce matrix coefficient modules and their specializations.
In Section \ref{Sec-fibered_cones}, we prove a generalization
Jaikin-Zapirain's result about fibered $3$--manifolds, allowing profinite epimorphisms.
In Section \ref{Sec-TN_cones},
we establish the profinite correspondence of the Thurston-norm cones
for orientable, nonpositively curved, virtually fibered $3$--manifolds.

In Section \ref{Sec-TN_unit_ball},
we establish the profinite invariance of the Thurston-norm unit ball,
for orientable finite-volume hyperbolic $3$--manifolds, 
proving Theorems \ref{main_xregular} and \ref{main_Thurston_norm}.

In Section \ref{Sec-TRT},
we establish the profinite invariance of certain twisted Reidemeister torsions over $\Rational$,
for profinitely corresponding fibered classes of orientable finite-volume $3$--manifolds.
In Section \ref{Sec-Nielsen},
we establish the profinite invariance of indexed periodic Nielsen numbers,
for profinitely corresponding fibered classes of orientable closed hyperbolic $3$--manifolds.
In Section \ref{Sec-profinite_almost_rigidity},
we establish the profinite almost rigidity among finitely generated $3$--manifold groups,
for finite-volume hyperbolic $3$--manifolds,
proving Theorem \ref{main_profinite_almost_rigidity}.

\subsection*{Acknowledgement} 
The author is grateful to the anonymous referee for careful proofreading,
and for correcting an inaccurate point
in a preliminary version of the proof of Lemma \ref{p_a_r_cover}.

\section{Preliminary}\label{Sec-preliminary}
In this section, we review background materials that are necessary
for our subsequent treatment.
We use \cite{AFW_book_group} as our general reference book for $3$--manifold groups,
and \cite{Brown_book,Ribes--Zalesskii_book} 
for cohomology and homology theory of abstract or profinite groups,
and \cite{Turaev_book_torsion} for combinatorial torsions,
and \cite{FLP_book} for surface automorphisms.
Other topic surveys include \cite{Friedl--Vidussi_survey} on twisted Alexander polynomials,
and \cite{Reid_survey} on profinite properties of $3$--manifolds.

\subsection{Profinite completions of groups}\label{Subsec-profinite_completion}
	For any group $G$,
	the \emph{profinite completion} $\widehat{G}$ refers to the inverse limit
	of (the inverse system of)
	all the quotient groups $G/N$, where $N$ ranges over all the finite-index normal subgroups of $G$.
	The natural homomorphism $G\to\widehat{G}$ is injective if and only if $G$ is residually finite.	
	
	The \emph{profinite topology} on $\widehat{G}$ can be characterized as
	the unique, coarsest topology, 
	such that every homomorphism of $\widehat{G}$ to any discrete finite group is continuous.
	Furnished with the profinite topology, $\widehat{G}$ becomes a compact, Hausdorff, totally disconnected 
	topological group.
	A subgroup of $\widehat{G}$ is open if and only if it is closed of finite index.
	When $G$ is finitely generated, 
	every finite-index subgroup of $\widehat{G}$ is open,
	which is a deep result due to N.~Nikolov and D.~Segal \cite{Nikolov--Segal}.
	It follows that 
	the abstract group structure of $\widehat{G}$ determines its profinite topology.
	Moreover, in this case,
	the isomorphism type of $\widehat{G}$
	is determined by all the finite quotients of $G$,
	or more precisely, by the set of their isomorphism types without duplication,
	no matter in the abstract or the topological sense,
	(see \cite[Theorem 2.2]{Reid_survey}).
		
	%Denote by $[\Integral G]$ the group algebra of $G$ over $\Integral$.
	The \emph{completion group algebra} $\llbracket\widehat{\Integral}\widehat{G}\rrbracket$
	%of $\widehat{G}$ over $\widehat{\Integral}$
	is defined as the inverse limit of all the quotient group algebras 
	$[(\Integral/l\Integral)(G/N)]$ of the group algebra $[\Integral G]$, where $l\Integral$ 
	ranges over all the finite-index ideals of $\Integral$, 
	and where $N$ ranges over all the finite-index normal subgroups of $G$.
	There is also a natural profinite topology on $\llbracket\widehat{\Integral}\widehat{G}\rrbracket$,
	which turns $\llbracket\widehat{\Integral}\widehat{G}\rrbracket$ into
	a compact, Hausdorff, totally disconnected topological ring.
	In general,  
	for any discrete left $\llbracket \widehat{\Integral}\widehat{G}\rrbracket$--module $V$,
	the $n$--th profinite group cohomology $H^n(\widehat{G};V)$
	with coefficients in $V$  ($n\in\Integral)$ is 
	a functorially defined, discrete $\widehat{\Integral}$--module,
	(which is abstractly a torsion abelian group);
	for any profinite right $\llbracket \widehat{\Integral}\widehat{G}\rrbracket$--module $W$,
	the $n$--th profinite group homology $H_n(\widehat{G};W)$ with coefficients in $W$ 
	($n\in\Integral$) is 
	a functorially defined, profinite $\widehat{\Integral}$--module,
	(which is abstractly a profinite abelian group).
	We refer the reader to \cite[Chapter 6]{Ribes--Zalesskii_book}
	for detailed account of profinite group cohomology and homology;
	compare the parallel theory with abstract groups in \cite[Chapter III]{Brown_book}.
			
	A group $G$ is said to be \emph{cohomologically good},
	if the natural homomorphism
	$H^n(\widehat{G};V)\to H^n(G;V)$ is an isomorphism of abelian groups,
	for all $n$, 
	and for all finite left $[\Integral G]$--modules $V$,
	(which are automatically discrete left $\llbracket\widehat{\Integral}\widehat{G}\rrbracket$--modules).
	This concept is introduced by J.~P.~Serre in \cite[Chapter I, \S 2, Exercise 1]{Serre_book_Galois}.
	When $G$ is of type $\mathrm{FP}_\infty$,	
	for each $n$ the group homology $H^n(G;V)$ is finitely generated as an abelian group
	if $V$ is finitely generated as an abelian group,
	so $H^n(G;V)$ is finite if $V$ is finite
	\cite[Chapter VIII, Section 5, cf.~Exercise 1]{Brown_book}.
	In this case, being cohomologically good 
	is equivalent to the condition that 
	$H_n(G;W)\to H_n(\widehat{G};W)$
	is an isomorphism
	for all $n$ and for all finite right $[\Integral G]$--modules $W$,
	(which are automatically profinite right $\llbracket\widehat{\Integral}\widehat{G}\rrbracket$--modules).
	In fact, this follows immediately from 
	the Pontrjagin duality between group homology and cohomology \cite[Chapter VI, Proposition 7.1]{Brown_book}, 
	and its profinite group version \cite[Chapter 6, Proposition 6.3.6]{Ribes--Zalesskii_book}.
	As pointed out in \cite[Proposition 3.1]{JZ},
	a group of type $\mathrm{FP}_\infty$ is cohomologically good
	if and only if 
	$H_n(G;\llbracket \widehat{\Integral}\widehat{G}\rrbracket)$ 
	vanishes except $H_0(G;\llbracket \widehat{\Integral}\widehat{G}\rrbracket)\cong\widehat{\Integral}$,
	or in other words,
	$H_*(G;\llbracket \widehat{\Integral}\widehat{G}\rrbracket)
	\cong H_*(\widehat{G};\llbracket \widehat{\Integral}\widehat{G}\rrbracket)$.
	
	In this paper, a \emph{$3$--manifold group} refers to any group
	that is isomorphic to the fundamental group of a connected $3$--manifold
	(possibly bounded, possibly non-compact).
	Any finitely generated $3$--manifold group can be realized 
	with a connected compact $3$--manifold.
	In particular,
	finitely generated $3$--manifold groups
	are all finitely presented and residually finite \cite[Section 3.2, (C.5) and (C.29)]{AFW_book_group}.
	Every finitely generated $3$--manifold group can be decomposed
	as a free product of finitely many freely indecomposable factors.
	The infinite factors are all realizable with aspherical compact $3$--manifolds,
	which provide finite CW complex models of their classifying spaces;
	the finite factors can be realized with spherical geometric $3$--manifolds,
	but their classifying spaces
	(and more generally, the classifying space of any finite group) 
	also admit CW complex models with finitely many cells of each dimension,
	by standard construction
	\cite[Chapter I, Section 5, cf.~Exercise 3]{Brown_book}.
	It follows that finitely generated $3$--manifold groups
	are all of type $\mathrm{FP}_\infty$.
	As a consequence of the geometric decomposition
	and the virtual fibering of (complete) finite-volume hyperbolic $3$--manifolds,
	it is known that finitely generated $3$--manifold groups are all
	cohomologically good \cite[Theorem 7.3]{Reid_discrete}.

\subsection{Polytopes and cones}\label{Subsec-polytopes_and_cones}
	For any orientable connected compact $3$-manifold $M$,
	the \emph{Thurston norm} for $M$ is a seminorm on the real linear space $H^1(M;\Real)$,
	introduced by W.~P.~Thurston in \cite{Thurston-norm}.
	It is symmetric and takes integral values on the integral lattice $H^1(M;\Integral)$.
	If $M$ supports a complete hyperbolic	structure of finite volume in its interior,
	the Thurston norm is actually a norm.
	In this paper, we denote the Thurston norm for $M$ as
	\begin{equation}\label{Thurston_norm_def}
	\|\cdot\|_{\mathtt{Th}}\colon H^1(M;\Real)\to [0,+\infty).
	\end{equation}
	
	We recall the definition (\ref{Thurston_norm_def})
	from \cite[Section 1]{Thurston-norm} as follows.
	For any orientable compact surface $\Sigma$ whose disconnected components are
	$\Sigma_1,\cdots,\Sigma_r$, the \emph{complexity} of $\Sigma$ refers to 
	the sum of the nonnegative integers $\mathrm{max}\{-\chi(\Sigma_i),0\}$,
	ranging over all the components $\Sigma_i$.
	Choose an orientation of $M$.
	For any cohomology class $\phi\in H^1(M;\Integral)$, 
	the Thurston norm $\|\phi\|_{\mathtt{Th}}$ is defined as 
	the minimum complexity of $\Sigma$,
	where $\Sigma$ ranges over all the properly embedded, oriented compact subsurfaces of $M$,
	such that $\phi$ is the Poincar\'e--Lefschetz dual of $[\Sigma]\in H_2(M,\partial M;\Integral)$.
	Thurston shows 
	$\|\phi+\psi\|_{\mathtt{Th}}\leq \|\phi\|_{\mathtt{Th}}+\|\psi\|_{\mathtt{Th}}$
	for all $\phi,\psi\in H^1(M;\Integral)$,
	and 
	$\|m\cdot\phi\|_{\mathtt{Th}}=|m|\cdot\|\phi\|_{\mathtt{Th}}$
	for all $\phi\in H^1(M;\Integral)$ and $m\in\Integral$.
	It follows that $\|\cdot\|_{\mathtt{Th}}$ can be extended linearly over $H^1(M;\Rational)$,
	and then continuously over $H^1(M;\Real)$.
	The result is the Thurston norm (\ref{Thurston_norm_def}),
	and it is obviously independent of the auxiliary orientation of $M$.
		
	In the above definition,
	one may use properly immersed subsurfaces instead of properly embedded ones,
	and the resulting seminorm does not change.
	This nontrivial fact is proved by D.~Gabai in \cite{Gabai-taut}, 
	confirming a former conjecture of Thurston \cite[Section 5]{Thurston-norm}.
	In particular, the Thurston norm for any finite cover 
	is proportional to the covering degree
	\cite[Corollary 6.13]{Gabai-taut}:
	
	\begin{theorem}\label{TN_cover}
	Let $M$ be any orientable connected compact $3$--manifold,
	and $M'\to M$ be a finite cover.
	Then the following formula holds
	for any cohomology class $\phi\in H^1(M;\Real)$ and its pullback $\phi'\in H^1(M';\Real)$:
	$$\|\phi'\|_{\mathtt{Th}}=[M':M]\cdot\|\phi\|_{\mathtt{Th}}.$$ 
	\end{theorem}
	
	The unit ball $\mathcal{B}_{\mathtt{Th}}(M)$ of the Thurston norm for $M$
	is a (possibly noncompact) codimension--$0$ convex polyhedron in $H^1(M;\Real)$.
	It is the linear dual of a polytope in $H_1(M;\Real)$,
	which is symmetric about the origin and 
	whose vertices all lie in the integral lattice $H_1(M;\Integral)_{\mathtt{free}}$,
	\cite[Section 2]{Thurston-norm}.
	The radial rays in $H^1(M;\Real)$
	through any closed face of $\mathcal{B}_{\mathtt{Th}}(M)$
	form a convex polyhedral cone, on which the Thurston norm is linear.
	Such cones are called \emph{Thurston-norm cones} for $M$.
	Therefore, $H^1(M;\Real)$ is the union of the codimension--$0$ Thurston-norm cones.
	For any finite cover $M'\to M$,
	the intersection of $\mathcal{B}_{\mathtt{Th}}(M')$
	with the image of induced embedding $H^1(M;\Real)\to H^1(M';\Real)$ is 
	the image of $\mathcal{B}_{\mathtt{Th}}(M)$ dilated by the factor $[M':M]$.
	This is the geometric interpretation of Theorem \ref{TN_cover}.

	\begin{remark}\label{polyhedron_remark}\
	\begin{enumerate}
	\item
	In real affine linear spaces, we speak of convexity 
	and the point-set topology;
	(a \emph{real affine linear space} 
	is a homogeneous space which is isomorphic to $\Real^n$ 
	furnished with its affine linear transformations).
	A (closed) \emph{convex polyhedron}	in a real affine linear space 
	refers to a closed convex subset,
	such that every point on the boundary lies in
	only finitely many maximal convex subsets of the boundary.
	A compact convex polyhedron is a \emph{polytope}.
	We speak of \emph{closed} or \emph{open faces}
	of various dimension or codimension in convex polyhedra, as usual.
	For clarity, we avoid saying `open convex polyhedra',
	which could be confusing in the positive codimension case.
	Instead, we speak of the interior of convex polyhedra.
	\item
	In real linear spaces,
	a (radial) \emph{cone} refers generally to a subset that is invariant under positive rescaling.
	However, we only encounter convex polyhedral cones in this paper.
	\end{enumerate}
	\end{remark}

	%\subsubsection{Fibered cones}
	A cohomology class $\phi\in H^1(M;\Integral)$ is said to be \emph{fibered}
	if $\phi$ is represented by the free homotopy class of
	a continuous map $M\to \Sph^1$ which is a fiber-bundle projection
	onto the standard circle.
	The fiber is connected precisely when $\phi$ is primitive.
	Thurston shows that every fibered class $\phi$ lies in 
	the interior of a unique codimension--$0$ Thurston-norm cone
	$\mathcal{C}_{\mathtt{Th}}(M,\phi)$, and moreover,
	every cohomology class $\psi\in H^1(M;\Integral)$
	in the interior of $\mathcal{C}_{\mathtt{Th}}(M,\phi)$
	is a fibered class; see \cite[Theorem 5]{Thurston-norm}.
	Any $\mathcal{C}_{\mathtt{Th}}(M,\phi)$ arising this way
	is called a \emph{fibered cone} for $M$.
	
	\begin{theorem}\label{quasi-fibered}
	Let $M$ be an orientable connected compact $3$--manifold.
	Suppose that $M$ has empty or tori boundary, and 
	admits a complete Riemannian metric of nonpositive sectional curvature in the interior.
	Then, there exists a finite regular cover $M'\to M$,
	such that for every nontrivial cohomology class $\phi\in H^1(M;\Real)$, 
	the pull-back cohomology class $\phi'\in H^1(M';\Real)$ 
	lies in the closure of a fibered cone for $M'$.
	\end{theorem}
	
	Theorem \ref{quasi-fibered} follows from 
	Agol's RFRS criterion for virtual fibering \cite{Agol_RFRS},
	and the virtual specialization of nonpositively curved $3$--manifold groups
	\cite{Agol_VHC,Liu_npc,Przytycki--Wise_graph,Przytycki--Wise_mixed}.
	Note that $\phi'$ must lie on the boundary of a fibered cone if $\phi$ is not fibered,
	and such $\phi'$ is sometimes called \emph{quasi-fibered} in the literature.
	See \cite[Sections 4.7, 4.8, and 5.4.3]{AFW_book_group}.
	
	%{\color{blue} other structures including the unparametrized flow and the homology direction}
	
	The cones in $H_1(M;\Real)$ dual to the fibered cones 
	are intimately related to flow structures on $M$ transverse to the surface fibers.
	We describe the picture below,
	focusing on the case of fibered closed hyperbolic $3$--manifolds,
	and following Fried's expository article \cite[{Expos\'e 14}]{FLP_book}.	
	
	Suppose that $M$ is an orientable closed $3$--manifold which admits a hyperbolic metric.
	If $\phi\in H^1(M;\Integral)$ is a primitive fibered class,
	then $M$ can be identified homeomorphically with the mapping torus
	of a pseudo-Anosov automorphism, such that $\phi$ is identified with
	the distinguished cohomology class dual to the surface, (see Remark \ref{pA_remark}).
	D. Fried gives a characterization of 
	the linear dual	in $H_1(M;\Real)$ of the fibered cone $\mathcal{C}_{\mathtt{Th}}(M,\phi)$,
	which we denote as
	$$\mathcal{C}^{\mathtt{Fr}}(M,\phi)=
	\left\{x\in H_1(M;\Real)\colon \psi(x)\geq 0\mbox{ for all }\psi\in \mathcal{C}_{\mathtt{Th}}(M,\phi)\right\}.$$
	Considering all the homology classes represented by the forward periodic trajectories
	of the suspension flow, 
	Fried shows that the radial rays in $H_1(M;\Real)$ that pass through those homology classes
	form a dense subset of $\mathcal{C}^{\mathtt{Fr}}(M,\phi)$.
	Therefore, 	
	$\mathcal{C}^{\mathtt{Fr}}(M,\phi)$
	is the smallest convex polyhedral cone in $H_1(M;\Real)$ 
	that contains those homology classes,
	and it has codimension zero in $H_1(M;\Real)$.
	Up to isotopy,
	the fibered cone $\mathcal{C}_{\mathtt{Th}}(M,\phi)$ actually determines
	the unparametrized flow,
	by which we mean
	the oriented $1$--foliation of $M$ given by the flow trajectories.
	Different primitive cohomology classes in $\mathcal{C}_{\mathtt{Th}}(M,\phi)$ are 
	just the Poincar\'e dual to different cross-sections of the unparametrized flow.
		
	\begin{remark}\label{pA_remark}\
	\begin{enumerate}
	\item
	For any orientable connected compact surface $S$, possibly with boundary,
	and for	any orientation-preserving self-homeomorphism
	$f\colon S\to S$,
	the \emph{mapping torus} of $(S,f)$ refers to the $3$--manifold
	%$$M_f=\frac{S\times\Real}{(x,r+1)\sim(f(x),r)}.$$
	$M_f={S\times\Real}/\sim$, 
	such that ${(x,r+1)\sim(f(x),r)}$ for all $(x,r)\in S\times\Real$.
	The \emph{suspension flow} on $M_f$ refers the continuous family of homeomorphisms
	$\theta_t\colon M_f\to M_f$,
	parametrized by $t\in\Real$ and determined by $(x,r)\mapsto (x,r+t)$.	
	The distinuished cohomology class of $M_f$ refers to 
	the cohomology class $\phi_f\in H^1(M_f;\Integral)$ 
	represented by the distinguished projection $M_f\to \Sph^1$,
	where $\Sph^1$ is identified with $\Real/\Integral$ 
	and the projection is determined by $(x,r)\mapsto r$.
	\item
	Moreover, when $S$ is closed,
	$f$ is said to be a \emph{pseudo-Anosov} automorphism
	if there exist a constant $\lambda>1$ and a pair of measured foliations 
	$(\mathscr{F}^{\mathtt{s}},\mu^{\mathtt{s}})$ and $(\mathscr{F}^{\mathtt{u}},\mu^{\mathtt{u}})$
	of $S$,
	such that
	$f\cdot(\mathscr{F}^{\mathtt{u}},\mu^{\mathtt{u}})=(\mathscr{F}^{\mathtt{u}},\lambda\mu^{\mathtt{u}})$
	and
	$f\cdot(\mathscr{F}^{\mathtt{s}},\mu^{\mathtt{s}})=(\mathscr{F}^{\mathtt{s}},\lambda^{-1}\mu^{\mathtt{s}})$.
	The constant $\lambda$ is called the \emph{stretch factor} of 
	the pseudo-Anosov automorphism $f$.
	The measured foliations $(\mathscr{F}^{\mathtt{s}},\mu^{\mathtt{s}})$ and $(\mathscr{F}^{\mathtt{u}},\mu^{\mathtt{u}})$
	are called the \emph{stable} and the \emph{unstable} measured foliations of $f$, respectively.
	The unmeasured invariant foliations are transverse to each other,
	except at finitely many common singular points.
	At each singular point, 
	both of the invariant foliations have a $k$--prong singularity,
	for some positive integer $k\geq3$.
	\end{enumerate}
	\end{remark}

\subsection{Twisted invariants}\label{Subsec-twisted_invariants}
	Let $M$ be a connected compact $3$--manifold.
	Denote by $\pi_1(M)$ the fundamental group of $M$
	with respect to an implicitly fixed basepoint.
	Let $R$ be a (commutative) Noetherian unique factorization domain.
	%Denote by $R^\times$ the multiplicative group of units in $R$,
	Denote by $R[t^{\pm1}]$ the ring of Laurent polynomials over $R$ in an indeterminate $t$.
	We recall the following twisted invariants from 
	Friedl and Vidussi's survey \cite{Friedl--Vidussi_survey}.

	Let	$\rho\colon\pi_1(M)\to \mathrm{GL}(k,R)$ 
	be a linear representation of $\pi_1(M)$ 
	on the free $R$--module $R^k$ of finite rank $k$.
	Let $\phi\in H^1(M;\Integral)$ be a cohomology class,
	regarded meanwhile as a homomorphism $\phi\colon\pi_1(M)\to \Integral$.
	For any $n\in\Integral$,
	the $n$--th \emph{twisted homology} $H_n^{\rho,\phi}(M;R[t^{\pm1}]^k)$
	%of $M$ with respect to $(\rho,\phi)$ 
	refers to the twisted (singular) homology of $M$ with respect to the representation
	$t^\phi\rho\colon \pi_1(M)\to \mathrm{GL}(k,R[t^{\pm1}])$,
	such that $g\mapsto t^{\phi(g)}\cdot\rho(g)$.
	%If one fixes a (pointed) universal covering space $\widetilde{M}$ of $M$
	%and identify $\pi_1(M)$ as the deck transformation group,
	%$H_n^{\rho,\phi}(M;R[t^{\pm1}])$ can be constructed explicitly
	%as the $n$--th homology of the twisted (singular) chain complex 
	%$C_\bullet(\widetilde{M})\otimes_{\Integral\pi_1(M)}R[t^{\pm1}]^k$
	%with boundary operators $\partial_\bullet\otimes\mathbf{1}$.
	This is a finitely generated $R[t^{\pm1}]$--module,
	vanishing unless $n=0,1,2,3$.
	In this paper,
	we sometimes simplify the notation as 
	$H^{\rho,\phi}_n$,	if $M$ and $R[t^{\pm1}]^k$ are apparent from the context.
			
	The $n$--th \emph{twisted Alexander polynomial}
	$\Delta_n^{\rho,\phi}(M;R[t^{\pm1}]^k)$	of $M$ with respect to $(\rho,\phi)$,
	or simply $\Delta_n^{\rho,\phi}(t)$,
	is defined as the order of the finitely generated $R[t^{\pm1}]$--module
	$H_n^{\rho,\phi}(M;R[t^{\pm1}]^k)$,
	(see Remark \ref{order_remark}).
	We treat $\Delta_n^{\rho,\phi}(t)$ as a Laurent polynomial in $t$ over $R$,
	%up to units of $R[t^{\pm1}]$.
	uniquely defined up to monomial factors and units of $R$.
	
	If the rank of $H_n^{\rho,\phi}(M;R[t^{\pm1}]^k)$ 
	over $R[t^{\pm1}]$ is zero for all $n$, 
	the \emph{twisted Reidemeister torsion} $\tau^{\rho,\phi}(M;R[t^{\pm1}]^k)$
	of $M$ with respect to $(\rho,\phi)$, or simply $\tau^{\rho,\phi}(t)$,
	can be characterized as the alternating product
	\begin{equation}\label{TRT_def}
	\tau^{\rho,\phi}(t)
	\doteq 
	\frac{\prod_{n\mbox{ odd}}\Delta^{\rho,\phi}_n(t)}{
	\prod_{n\mbox{ even}}\Delta^{\rho,\phi}_n(t)},
	\end{equation}
	where the dotted equality symbol
	means being equal up to monomial factors with nonzero $R$--fractional coefficients,
	in the field of rational functions in $t$ over $R$,
	(see Remark \ref{order_remark}).
	Note that the products in (\ref{TRT_def}) 
	essentially only involves	the factors with $n=0,1,2,3$.
	If some $H_n^{\rho,\phi}(M;R[t^{\pm1}]^k)$ has $R[t^{\pm1}]$--rank nonzero,
	it is customary to define $\tau^{\rho,\phi}\left(M;R[t^{\pm1}]^k\right)$
	to be zero.

	\begin{remark}\label{order_remark}
	For any (commutative) domain $S$ and and any finitely generated $S$--module $V$,
	the \emph{rank} of $V$ refers to	the dimension of $V\otimes_S\mathrm{Frac}(S)$ 
	over the field of fractions $\mathrm{Frac}(S)$.
	When $S$ is a Noetherian unique factorization domain,
	%$V$ is finitely presented, 
	%so it is isomorphic to the cokernel	of some matrix over $S$.
	the \emph{order} of $V$ refers to 
	%any generator of the minimal principal ideal of $S$
	%that contains the $0$--th elementary ideal of $V$.
	the greatest common divisor of all the largest-size minors
	in a finite-presentation matrix of $V$,
	(that is, a matrix over $S$ 
	whose cokernel is isomorphic to $V$).
	The order of $V$ is well-defined up to a unit of $S$,
	and it depends only on the isomorphism type of $V$.
	Therefore, the order of $V$ equals zero
	if and only if the rank of $V$ is nonzero.
	See \cite[Chapter I, Section 4]{Turaev_book_torsion}.
	\end{remark}
	
	%\begin{theorem}\label{TAP_properties}
		%Let $M$ be a connected compact orientable $3$--manifold.
		%Let $R$ be a Noetherian UFD, and $k$ be a natural number.
		%Suppose that $M$ has empty or tori boundary.
		%Suppose that $\phi\in H^1(M;\Integral)$ is nontrivial,
		%and $\rho\colon \pi_1(M)\to \mathrm{GL}(k,R)$ is a representation.
		%Then the following statements all hold true.
		%\begin{enumerate}
		%\item $\Delta_n^{\rho\otimes\phi}$		
		%\end{enumerate}	
	%\end{theorem}
	
	We collect the following facts 
	from \cite{Friedl--Vidussi_survey},
	see Section 3.3.1, Proposition 2 (1), (4), (7) and Section 3.3.4, Propositions 4 and 5
	therein.
	
	\begin{theorem}\label{TAP_properties}
	Let $M$ be an orientable connected compact $3$--manifold with empty or tori boundary.
	Let $R$ be a Noetherian UFD and $k$ be a natural number.
	Suppose that $\phi\in H^1(M;\Integral)$ is a nontrivial cohomology class
	and $\rho\colon \pi_1(M)\to\mathrm{GL}(k,R)$ is a representation.
	If $\Delta^{\rho,\phi}_1(t)\neq0$,
	then $\Delta^{\rho,\phi}_n(t)\neq0$ for all $n$.
	In this case,	the following statements all hold true.
	\begin{enumerate}
	\item $\Delta^{\rho,\phi}_n(t)\doteq 1$ for all $n$ other than $0,1,2$. 
	\item  
	If $\partial M\neq\emptyset$, 
	then 
	$\Delta^{\rho,\phi}_0(t)\doteq \Delta^{\bar{\rho},\phi}_0(t^{-1})$, and
	$\Delta^{\rho,\phi}_1(t)\doteq \Delta^{\bar{\rho},\phi}_1(t^{-1})$, and
	$\Delta^{\rho,\phi}_2(t)\doteq1$.
	\item If $\partial M=\emptyset$, then
	$\Delta^{\rho,\phi}_0(t)\doteq \Delta^{\bar{\rho},\phi}_2(t^{-1})$, and
	$\Delta^{\rho,\phi}_1(t)\doteq \Delta^{\bar{\rho},\phi}_1(t^{-1})$, and
	$\Delta^{\rho,\phi}_2(t)\doteq\Delta^{\bar{\rho},\phi}_0(t^{-1})$.	
	%\item $\tau^{\rho,\phi}(t)\doteq\tau^{\bar{\rho}}(t^{-1})$.
	\end{enumerate}
	Here, $\bar{\rho}$ stands for the representation $\pi_1(M)\to \mathrm{GL}(k,R)$, such that 
	$\bar{\rho}(g)$ is the transpose of $\rho(g)^{-1}$ for all $g\in\pi_1(M)$;
	the dotted equality symbol means an equality up to monomial factors and units of $R$.
	\end{theorem}
	
	With respect to a fibered class, it is possible to interprete
	the twisted Reidemeister torsion (\ref{TRT_def}) 
	through the dynamics of a fiber-transverse flow.
	Again we describe in the case of pseudo-Anosov mapping tori.
	%See also the review sections \cite[Sections 7 and 10]{Liu_procongruent_conjugacy}.
	
	Suppose that $M$ is an orientable closed $3$--manifold which admits a hyperbolic metric.
	If $\phi\in H^1(M;\Integral)$ is a primitive fibered class,
	we adopt the notations of Remark \ref{pA_remark}
	and identify $(M,\phi)$ with $(M_f,\phi_f)$,
	where $f\colon S\to S$ is 
	a pseudo-Anosov automorphism of a connected closed orientable surface.
	In this setting,
	periodic orbits of $f$ correspond to periodic trajectories of the suspension flow,
	and they have well-defined periodic indices.
		
	For any natural number $m\in\Natural$,
	an $m$--periodic point of $f$ refers to a fixed point $p\in S$ of $f^m$,
	and it gives rise to an $m$--periodic trajectory of the suspension flow on $M_f$,
	namely, the loop $\Real/m\Integral\to M_f$ determined by 
	the map $\Real\to S\times \Real\colon r\mapsto (p,r)$.
	The free-homotopy class of an $m$--periodic trajectory
	depends only on the $f$--iteration orbit of an $m$--periodic point.
	We denote by $\mathrm{Orb}_m(f)$ the set of $m$--periodic orbits of $f$,
	and for any $m$--periodic orbit $\mathbf{O}\in\mathrm{Orb}_m(f)$,
	denote by $\ell_m(f;\mathbf{O})$
	the free-homotopy class of the $m$--periodic trajectory
	determined by $\mathbf{O}$. 
	Since the homotopy set $[\Real/m\Integral,M_f]$ is naturally identified 
	with the set of conjugacy classes of $\pi_1(M_f)$,
	denoted as $\mathrm{Orb}(\pi_1(M_f))$,
	we write
	\begin{equation}\label{ell_m}
	\ell_m(f;\mathbf{O})\in\mathrm{Orb}(\pi_1(M_f)),
	\end{equation}
	referring to it as an (essential) \emph{$m$--periodic trajectory class}.
	
	For any $m$--periodic point $p\in S$ of $f$,
	the $m$--periodic index of $f$ at $p$ refers to the fixed point index of $f^m$ at $p$.
	Explicitly, when $p$ is a $k$--prong singularity of 
	the stable (or unstable) invariant foliation for some $k\geq3$,
	the $m$--periodic index of $f$ at $p$ equals $1-k$ if $f^m$ 
	preserves every prong	of the foliation at $p$,
	or it equals $1$ otherwise;
	when $p$ is a regular point,
	the $m$--periodic index of $f$ at $p$ equals $-1$ or $1$,
	according to a similar rule as if $k=2$.
	Note that
	the $m$--periodic index of $f$ depends only	on the $f$--iteration orbit of $m$--periodic points.
	For any $m$--periodic orbit $\mathbf{O}\in\mathrm{Orb}_m(f)$,
	we define the \emph{$m$--periodic index} of $f$ at $\mathbf{O}$ using any point $p\in\mathbf{O}$,
	and denote it as
	\begin{equation}\label{ind_m}
	\mathrm{ind}_m(f;\mathbf{O})\in\Integral\setminus\{0\}.
	\end{equation}
	
	For any representation $\rho\colon \pi_1(M_f)\to \mathrm{GL}(k,R)$,
	the character $\chi_\rho\colon\pi_1(M_f)\to R$, 
	defined as $g\mapsto \mathrm{tr}_R(\rho(g))$,
	is constant on conjugacy classes of $\pi_1(M_f)$.
	So we treat $\chi_\rho$ as a function $\mathrm{Orb}(\pi_1(M_f))\to R$.	
	For any $m\in\Natural$,
	we define the $m$--th \emph{twisted Lefschetz number} with respect to $\chi_\rho$ and $f$
	as the following value in $R$:
	\begin{equation}\label{L_m_rho}
	L_m(f;\chi_\rho)=\sum_{\mathbf{O}\in\mathrm{Orb}_m(f)} \chi_\rho\left(\ell_m(f;\mathbf{O})\right)\cdot\mathrm{ind}_m(f;\mathbf{O}).
	\end{equation}
	
	The following theorem reformulates a result due to B.~Jiang \cite[Theorem 1.2]{Jiang_periodic},
	(see \cite[Lemma 8.2]{Liu_procongruent_conjugacy}):
	
	\begin{theorem}\label{tau_zeta}
	Let $M$ be an orientable connected closed $3$--manifold
	which admits a hyperbolic metric.
	Let $\Field$ be a (commutative) field of characteristic $0$.
	Suppose that $\phi\in H^1(M;\Integral)$ is a primitive fibered class.
	Adopt the notations of Remark \ref{pA_remark},
	and identify $(M,\phi)$ with $(M_f,\phi_f)$ for some pseudo-Anosov automorphism $f\colon S\to S$.	
	Suppose that $\rho\colon \pi_1(M_f)\to \mathrm{GL}(k,\Field)$ 
	is a finite-dimensional linear representation over $\Field$.
	Then, the following formula holds true:
	$$\tau^{\rho,\phi}\left(M;\Field[t^{\pm1}]^k\right)\doteq
	\exp\left(\sum_{m\in\Natural} \frac{L_m(f;\chi_\rho)}m\cdot t^m\right),$$
	where $\exp(z)$ stands for the formal power series $\sum_{m=0}^\infty \frac{z^m}{m!}$.
	
	More precisely,
	the right-hand side is a formal power series in $t$
	over $\Field$ with constant term $1$;
	it is the formal power series expansion
	of a unique rational function in $t$ over $\Field$;
	and it equals the left-hand side, 
	up to monomial factors with coefficients in $\Field^\times$,
	as a rational function in $t$ over $\Field$.
	\end{theorem}
	
	\begin{remark}\label{tau_zeta_remark}
		The formula in Theorem \ref{tau_zeta} works in general 
		for mapping classes	of orientable closed (or compact) surfaces,
		in terms of essential periodic orbit classes and their indices.
		For pseudo-Anosov automorphisms of closed surfaces,
		any essential periodic orbit class consists of a unique periodic orbit.
		See Jiang's book \cite{Jiang_book} for a comprehensive introduction
		to the Nielsen fixed point theory; 
		see also \cite[Sections 7 and 10]{Liu_procongruent_conjugacy}
		for a quick summary.
	\end{remark}

\section{Matrix coefficients of profinite linear transformations}\label{Sec-MC}
In this section, we introduce matrix coefficient modules and their specializations.
For any finitely generated free $\Integral$--module $H$ of finite rank, 
denote by $H^*$ the dual $\Integral$--module $\mathrm{Hom}_\Integral(H,\Integral)$.
We identify the abelian profinite group completion $\widehat{H}$ of $H$ 
naturally as the (abstract) $\widehat{\Integral}$--module $H\otimes_\Integral \widehat{\Integral}$.
For any pair of finitely generated free $\Integral$--modules $H_A$ and $H_B$,
we identify the abelian profinite group $\mathrm{Hom}(\widehat{H_A},\widehat{H_B})$
of continuous homomorphisms $\widehat{H_A}\to \widehat{H_B}$
naturally as the $\widehat{\Integral}$--module
$\mathrm{Hom}_{\Integral}(H_A,H_B)\otimes_\Integral\widehat{\Integral}$.

\begin{definition}\label{MC_def}
Let $H_A$ and $H_B$ be a pair of finitely generated free $\Integral$--modules.
Let $\Phi\colon \widehat{H_A}\to \widehat{H_B}$ be a continuous homomorphism of the profinite completions.
We define the \emph{matrix coefficent module} for $\Phi$ with respect to $H_A$ and $H_B$
to be the smallest $\Integral$--submodule $L$ of $\widehat{\Integral}$,
such that $\Phi(H_A)$ lies in the submodule $H_B\otimes_\Integral L$ 
of $\widehat{H_B}=H_B\otimes_\Integral\widehat{\Integral}$.
We denote the matrix coefficient module as 
$$\mathrm{MC}(\Phi;H_A,H_B)\subset\widehat{\Integral},$$
or simply $\mathrm{MC}(\Phi)$ when there is no danger of confusion.
We denote by 
$$\Phi^{\mathrm{MC}}\colon H_A\to H_B\otimes_\Integral \mathrm{MC}(\Phi)$$
the homomorphism uniquely determined by the restriction of $\Phi$ to $H_A$.
\end{definition}

%\begin{example}\label{MC_def_example}
	%Let $\pi_A$ and $\pi_B$ be a pair of finitely generated, residually finite groups $\pi_A$ and $\pi_B$,
	%and $\Phi\colon \widehat{\pi_A}\to\widehat{\pi_B}$
	%be a continuous homomorphism between their profinite completions.
	%We take $H_A=H_1(\pi_A;\Integral)_{\mathtt{free}}$, and $H_B=H_1(\pi_B;\Integral)_{\mathtt{free}}$,
	%and $\Phi\colon \widehat{H_A}\to \widehat{H_B}$ the continuous homomorphism induced by $\Phi$.
%\end{example}

\begin{proposition}\label{MC_property}
Let $H_A$ and $H_B$ be a pair of finitely generated free $\Integral$--modules.
Let $\Phi\colon\widehat{H_A}\to\widehat{H_B}$ be a continuous homomorphism.
\begin{enumerate}
\item 
With respect to any bases $u_1,\cdots,u_a$ of $H_A$ and $v_1,\cdots,v_b$ of $H_B$,
represent $\Phi$ uniquely as a matrix $(\varphi_{ij})_{b\times a}$ over $\widehat{\Integral}$.
Then $\mathrm{MC}(\Phi;H_A,H_B)$ is $\Integral$--linearly spanned by all the entries 
$\varphi_{ij}$ in $\widehat{\Integral}$.
In particular, 
$\mathrm{MC}(\Phi;H_A,H_B)$ 
is a free $\Integral$--submodule of $\widehat{\Integral}$ of finite rank.
%$$\mathrm{rank}_\Integral(\mathrm{MC}(\Phi;H_A,H_B))\leq 
%\mathrm{rank}_\Integral(H_A)\times \mathrm{rank}_\Integral(H_B).$$
\item
If $\mathrm{rank}_\Integral(\mathrm{MC}(\Phi;H_A,H_B))=1$,
then for each generator $z\in\widehat{\Integral}$ of $\mathrm{MC}(\Phi;H_A,H_B)$,
there is a unique homomorphism $F_z\colon H_A\to H_B$,
such that $\Phi=z\widehat{F_z}$ holds in 
$\mathrm{Hom}(\widehat{H_A},\widehat{H_B})$.
Hence, the other generator $-z$ gives rise to $F_{-z}=-F_z$.
%Moreover, $z\in\widehat{\Integral}^\times$ and $T\in\mathrm{Isom}_\Integral(H_A,H_B)$
%hold if $\Phi$ is a profinite isomorphism.
%\item
%For any finite-index $\Integral$--submodules $L_A$ of $H_A$ and $L_B$ of $H_B$,
%such that $\Phi(\widehat{L_A})\subset\widehat{L_B}$, 
%restrict $\Phi$ to obtain a continuous homomorphism $\widehat{L_A}\to\widehat{L_B}$,
%still denoted as $\Phi$.
%Then $\mathrm{MC}(\Phi;L_A,L_B)$ and $\mathrm{MC}(\Phi;H_A,H_B)$ 
%are commensurable in $\widehat{\Integral}$,
%namely, their intersection has finite index in both.
%In particular, 
%$$\mathrm{rank}_\Integral(\mathrm{MC}(\Phi;H_A,H_B))=\mathrm{rank}_\Integral(\mathrm{MC}(\Phi;L_A,L_B)).$$
\item
For any $\Integral$--submodules $L_A$ of $H_A$ and $L_B$ of $H_B$,
such that $\Phi(\widehat{L_A})$ is contained in $\widehat{L_B}$, 
restrict $\Phi$ to obtain a continuous homomorphism $\widehat{L_A}\to\widehat{L_B}$,
still denoted as $\Phi$.
Then $\mathrm{MC}(\Phi;L_A,L_B)$ is commensurable with
a $\Integral$--submodule of $\mathrm{MC}(\Phi;H_A,H_B)$ in $\widehat{\Integral}$.
In particular, 
$$\mathrm{rank}_\Integral(\mathrm{MC}(\Phi;L_A,L_B))\leq\mathrm{rank}_\Integral(\mathrm{MC}(\Phi;H_A,H_B)).$$
Moreover, if $L_A$ and $L_B$ both have finite index,
then $\mathrm{MC}(\Phi;L_A,L_B)$ is commensurable with $\mathrm{MC}(\Phi;H_A,H_B)$ in $\widehat{\Integral}$,
and the equality is achieved.
\end{enumerate}
\end{proposition}

\begin{proof}
	For any $u\in H_A$ and $v^*\in H^*_B=\mathrm{Hom}_\Integral(H_B,\Integral)$,
	we obtain an element $\widehat{v^*}(\Phi(u))$ in $\mathrm{MC}(\Phi)$,
	where $\widehat{v^*}\colon \widehat{H_B}\to\widehat{\Integral}$ is the completion of $v^*$.
	Moreover, all such elements form a $\Integral$--submodule of $\widehat{\Integral}$,
	so it is by definition $\mathrm{MC}(\Phi;H_A,H_B)$.
	For any given basis $u_1,\cdots,u_a$ of $H_A$,
	and any basis $v^*_1,\cdots,v^*_b$ of $H_B^*$
	dual to a given basis of $H_B$,
	the elements $u_i\otimes v_j^*$ form a basis
	of $H_A\otimes_\Integral H_B^*$,
	so the matrix entries $\varphi_{ij}=\widehat{v^*_j}(\Phi(u_i))$ 
	span $\mathrm{MC}(\Phi;H_A,H_B)$ over $\Integral$.
	This proves the assertion (1).
	
	If the rank of $\mathrm{MC}(\Phi;H_A,H_B)$ equals $1$, 
	every entry $\varphi_{ij}$ is a unique $\Integral$--multiple $t_{ij}$ of 
	any fixed generator $z\in\widehat{\Integral}$
	of $\mathrm{MC}(\Phi;H_A,H_B)$. So the matrix of $\Phi$ takes the form
	$zT$ where $T$ is the matrix with entries $t_{ij}$.
	This proves the assertion (2).
	
	It remains to prove the assertion (3).	
	In fact, it suffices to prove the finite-index case,
	then the general case follows from the assertion (1),
	by considering the smallest direct-summand modules $L'_A$ of $H_A$ that contains $L_A$,
	and $L'_B$ of $H_B$ that contains $L_B$.
	Suppose below that $L_A$ and $L_B$ have finite index in $H_A$ and $H_B$, respectively.
	Take bases $u_1,\cdots,u_a$ of $H_A$, and $v_1,\cdots,v_b$ of $H_B$,
	and represent $\Phi\colon \widehat{H_A}\to \widehat{H_B}$
	as a $b\times a$--matrix $T$ over $\widehat{\Integral}$,
	namely, $(\Phi(u_1),\cdots,\Phi(u_a))=(v_1,\cdots,v_b)T$.
	Similarly, take bases $u'_1,\cdots,u'_a$ of $L_A$, and $v'_1,\cdots,v'_b$ of $H_B$,
	and represent $\Phi\colon \widehat{L_A}\to \widehat{L_B}$ as a matrix $T'$.
	Let $S_A$ and $S_B$ be the square matrices over $\Integral$
	such that $(u'_1,\cdots,u'_a)=(u_1,\cdots,u_a)S_A$ and $(v'_1,\cdots,v'_b)=(v_1,\cdots,v_b)S_B$.
	We obtain a matrix equation $S_BT'=TS_A$ over $\widehat{\Integral}$.
	Then $S_BT'S^*_A=\mathrm{det}(S_A)\cdot T$, where $S^*_A$ is the adjoint matrix of $S_A$,
	and $\mathrm{det}(S_A)$ equals the index $[H_A:L_A]$ of $L_A$ in $H_A$.
	Applying the assertion (1), we see that $[H_A:L_A]\cdot \mathrm{MC}(\Phi;H_A,H_B)$
	is contained in $\mathrm{MC}(\Phi;L_A,L_B)$.
	Similarly, $[H_B:L_B]\cdot \mathrm{MC}(\Phi;L_A,L_B)$
	is contained in $\mathrm{MC}(\Phi;H_A,H_B)$.
	Therefore, the assertion (3) follows.	
\end{proof}

\begin{definition}\label{specialization_def}
Adopt the notations of Definition \ref{MC_def}.
For any homomorphism $\varepsilon\in\mathrm{Hom}_\Integral(\mathrm{MC}(\Phi),\Real)$,
the \emph{$\varepsilon$--specialization} of $\Phi$
refers to the unique $\Real$--linear extension of the composite homomorphism 
$$\xymatrix{H_A \ar[r]^-{\Phi^{\mathrm{MC}}} & H_B\otimes_\Integral \mathrm{MC}(\Phi)
\ar[r]^-{1\otimes\varepsilon} & H_B\otimes_\Integral\Real,}$$
denoted as 
\begin{equation*}\label{specialization}
\Phi^\varepsilon\colon H_A\otimes_\Integral\Real\to H_B\otimes_\Integral\Real;
\end{equation*}
the \emph{dual $\varepsilon$--specialization} of $\Phi$
refers to the $\Real$--linear homomorphism dual to the $\varepsilon$--specialization,
denoted as
\begin{equation*}\label{dual_specialization}
\Phi^{\mathtt{dual}}_\varepsilon\colon \mathrm{Hom}_\Integral(H_B;\Real)\to\mathrm{Hom}_\Integral(H_A;\Real).
\end{equation*}
\end{definition}

%\begin{lemma}
%Let $H_A$ and $H_B$ be a pair of finitely generated free $\Integral$--modules.
%Let $\Phi\colon\widehat{H_A}\to\widehat{H_B}$ be a continuous homomorphism.
%Let $L_A$ and $L_B$ be $\Integral$--submodules of $H_A$ and $H_B$, respectively.
%Then the following conditions are equivalent:
%\begin{enumerate}
%\item $\Phi(\widehat{L_A})$ is contained in $\widehat{L_B}$.
%\item For all $\varepsilon\in\mathrm{Hom}_\Integral(\mathrm{MC}(\Phi),\Real)$,
%$\Phi^{\varepsilon}(L_A\otimes_\Integral\Real)$ is contained in $L_B\otimes_\Integral\Real$.
%\item For some (point-set topologically) nonempty open subset $\mathrm{Hom}_\Integral(\mathrm{MC}(\Phi),\Real)$,
%and for all $\varepsilon\in\mathcal{U}$,
 %$\Phi^{\varepsilon}(L_A\otimes_\Integral\Real)$ is contained in $L_B\otimes_\Integral\Real$.
%\end{enumerate}
%\end{lemma}

%\begin{proposition}\label{nondegenerate_specialization}
	%Let $H_A$ and $H_B$ be a pair of finitely generated free $\Integral$--modules.
	%Let $\Phi\colon\widehat{H_A}\to\widehat{H_B}$ be an epimorphism,
	%which is necessarily continuous.
	%Denote by $\mathcal{W}$ 
	%the subset of $\mathrm{Hom}_\Integral(\mathrm{MC}(\Phi),\Real)$
	%consisting of all $\varepsilon$ such that 
	%$\Psi^\epsilon\colon H_A\otimes_\Integral\Real\to H_B\otimes_\Integral\Real$
	%is also an epimorphism.
	%
	%Then, $\mathcal{W}$ is an open dense subset of $\mathrm{Hom}_\Integral(\mathrm{MC}(\Phi),\Real)$
	%with finitely many path-connected components.
%\end{proposition}
%
%\begin{proof}
%
%\end{proof}
%

\section{Detecting fibered cones with a profinite epimorphism}\label{Sec-fibered_cones}
In this section, we revisit Jaikin-Zapirain's work on profinite invariance of fiberedness,
with some slight generalization. We reformulate the result in terms of matrix coefficient modules,
(see Theorem \ref{fibered_cone} and Remark \ref{fibered_cone_remark}).
We introduce the following general setting,
which is frequently invoked in subsequent sections.

\begin{convention}\label{profinite_morphism_setting}
	By declaring $(M_A,M_B,\Psi)$ as a \emph{profinite morphism setting of a 3-manifold pair},
	we agree to adopt the following hypotheses and notations.
	\begin{enumerate}
	\item
	Suppose that $M_A$ and $M_B$ are 
	a pair of orientable connected compact $3$--manifolds.	
	Suppose that $\Psi\colon\widehat{\pi_A}\to\widehat{\pi_B}$
	is a continuous homomorphism 
	between the profinite completions of the fundamental groups 
	$\pi_A=\pi_1(M_A)$ for $M_A$, and similarly $\pi_B$ for $M_B$,
	with respect to implicitly fixed basepoints.
	\item
	Denote by $\Psi_*\colon \widehat{H_A}\to \widehat{H_B}$ the induced continuous homomorphism
	between the profinite completions of the first integral homology modulo torsion
	$H_A=H_1(M_A;\Integral)_{\mathtt{free}}$, and similarly $H_B$ for $M_B$.
	Denote by $\mathrm{MC}(\Psi_*)$ the matrix coefficient module for $\Psi_*$
	with respect to $H_A$ and $H_B$, (see Definition \ref{MC_def}).
	\item
	For any homomorphism $\varepsilon\in\mathrm{Hom}_\Integral(\mathrm{MC}(\Psi_*),\Real)$,
	denote by 
	$\Psi_*^\varepsilon\colon H_1(M_A;\Real)\to H_1(M_B;\Real)$
	the $\varepsilon$--specialization of $\Psi_*$,
	naturally identifying 
	$H_1(M_A;\Real)\cong H_A\otimes_\Integral\Real$, and similarly for $M_B$;
	and denote by
	$\Psi^*_\varepsilon\colon H^1(M_B;\Real)\to H^1(M_A;\Real)$
	the dual $\varepsilon$--specialization of $\Psi_*$,
	naturally identifying $H^1(M_A;\Real)\cong \mathrm{Hom}_\Integral(H_A,\Real)$ for $M_A$,
	and similarly for $M_B$,
	(see Definition \ref{specialization_def}).
	\end{enumerate}
\end{convention}

For most of the time, we only discuss profinite isomorphism settings,
and we often put extra assumptions on the $3$--manifolds,
such as nonpositively curved, or hyperbolic. 
This section is an exception, where we consider a profinite epimorphism setting.

\begin{theorem}\label{fibered_cone}
Let $(M_A,M_B,\Psi)$ be a profinite morphism setting of a $3$--manifold pair
(Convention \ref{profinite_morphism_setting}).
Suppose that $\Psi\colon\widehat{\pi_A}\to\widehat{\pi_B}$ is an epimorphism.
Suppose $\varepsilon\in\mathrm{Hom}_\Integral(\mathrm{MC}(\Psi_*),\Real)$.

Then, 
with respect to $\Psi^*_{\varepsilon}\colon H^1(M_B;\Real)\to H^1(M_A;\Real)$,
the preimage of any fibered cone for $M_A$
is either empty or contained in a unique fibered cone for $M_B$.
In other words,
a cohomology class $\phi\in H^1(M_B;\Real)$ lies in a fibered cone for $M_B$
if $\Psi^*_\varepsilon(\phi)$ lies in a fibered cone for $M_A$.
\end{theorem}

\begin{corollary}\label{both_fibered}
Under the assumptions of Theorem \ref{fibered_cone},
if $b_1(M_B)$ equals $b_1(M_A)$,
and if $M_A$ fibers over a circle, then $M_B$ also fibers over a circle.
\end{corollary}

\begin{remark}\label{fibered_cone_remark}\
	\begin{enumerate}
	\item 
	Any epimorphism between 
	profinite completions of finitely generated groups 
	is continuous, thanks to Nikolov--Segal \cite{Nikolov--Segal},
	(see Section \ref{Subsec-profinite_completion}).
	Our assumption in Theorem \ref{fibered_cone} is partly motivated
	by some well-known facts on knot group epimorphisms.
	Given an epimorphism between knot groups,
	the target knot is fibered if the source knot is fibered.
	This is simply because 
	knot groups all have infinite cyclic abelianization,
	and fibered knot groups are precisely those with finitely generated commutator subgroups.
	See also \cite{KSW_epimorphism,Silver--Whitten_epimorphism}.
	\item
	Corollary \ref{both_fibered} implies \cite[Corollary 1.2]{JZ} as a special case.
	Strictly speaking,
	the use of rank in \cite[Section 2]{JZ} relies on a nontrivial fact, that is,
	the commutative pro--$p$ group ring
	$\llbracket\mathbb{F}_p t^{\Integral_p}\rrbracket$ is a domain,
	(see \cite[Theorem B]{CKL_zero_divisor}). 
	For any prime $p$, 
	the notations $\mathbb{F}_p$ and $\Integral_p$
	here stand for the finite field of $p$ elements
	and the additive group of $p$--adic integers, respectively.
	For proving Theorem \ref{fibered_cone},
	we need to work with $\llbracket\mathbb{F} t^{\widehat{\Integral}}\rrbracket$,
	essentially because we need $\mathrm{MC}(\Psi_*)$	to be defined in $\widehat{\Integral}$,
	in order for some subsequent application, 
	(especially Theorem \ref{profinite_invariance_TRT}).
	On the other hand, the ground finite field $\mathbb{F}$ does not quite matter.
	Not to worry about zero divisors in $\llbracket\mathbb{F} t^{\widehat{\Integral}}\rrbracket$,
	we only speak of rank for finitely generated modules over 
	the dense subdomain $[\mathbb{F}t^{\widehat{\Integral}}]$,
	(namely, the group algebra over $\mathbb{F}$ 
	of the abstract torsion-free abelian group $\widehat{\Integral}$).
	This technical difference from \cite{JZ} is reflected
	on Proposition \ref{annihilator_behavior} including its proof.
	\end{enumerate}
\end{remark}

The rest of this section is devoted to the proof of Theorem \ref{fibered_cone}.
We invoke the following Theorem \ref{fibered_TH},
which translates the topological fibering condition
into an algebraic one, in terms of twisted homology over finite fields.
The `only if' direction is due to Friedl and Kim \cite[Theorem 1.2 and Lemma 2.2]{Friedl--Kim_fibered},
and the `if' direction is due to Friedl and Vidussi \cite{Friedl--Vidussi_vanishing}.

\begin{theorem}\label{fibered_TH}
Let $M$ be an orientable connected compact $3$--manifold and $\mathbb{F}$ be a finite field.
A nonzero cohomology class $\phi\in H^1(M;\Integral)$ is fibered if and only if
for every finite dimensional linear representation 
$\rho\colon \pi\to \mathrm{GL}(k,\mathbb{F})$ over $\mathbb{F}$,
$H^{\rho,\phi}_1(M;[\mathbb{F} t^{\Integral}]^k)$ has finite dimension over $\mathbb{F}$.
\end{theorem}

The next proposition prepares all the algebraic gadgets needed for proving Theorem \ref{fibered_cone}.

\begin{proposition}\label{annihilator_behavior}
	Let $\pi$ be a group which is cohomologically good and which is of type $\mathrm{FP}_\infty$.
	Let $\mathbb{F}$ be a (commutative) finite field.
	Let $\rho\colon \pi\to \mathrm{GL}(k,\mathbb{F})$ be a linear representation of $\pi$ 
	over $\mathbb{F}$ of finite dimension $k$.
	Denote by $\widehat{\rho}\colon \widehat{\pi}\to \mathrm{GL}(k,\mathbb{F})$ 
	the completion of $\rho$.
	\begin{enumerate}
	\item
	Let $\phi\colon \pi\to\widehat{\Integral}$ be a group homomorphism.
	If the $n$-th twisted group homology $H^{\rho,\phi}_n(\pi;[\mathbb{F}t^{\widehat{\Integral}}]^k)$
	has a nonzero annihilator in the group algebra $[\mathbb{F}t^{\widehat{\Integral}}]$,
	then the $n$-th twisted profinite group homology
	$H^{\widehat{\rho},\widehat{\phi}}_n(\widehat{\pi};\llbracket \mathbb{F}t^{\widehat\Integral}\rrbracket^k)$
	has a nonzero annihilator in 
	the completion group algebra $\llbracket \mathbb{F}t^{\widehat{\Integral}}\rrbracket$.
	\item
	Let $\phi,\psi\colon \pi\to\widehat{\Integral}$ be group homomorphisms.
	Suppose that $\mathrm{Ker}(\psi)$ contains $\mathrm{Ker}(\phi)$.
	If $H^{\rho,\psi}_n(\pi;[\mathbb{F}t^{\widehat{\Integral}}]^k)$
	has a nonzero annihilator in $[\mathbb{F}t^{\widehat{\Integral}}]$,
	then $H^{\rho,\phi}_n(\pi;[\mathbb{F}t^{\widehat{\Integral}}]^k)$
	has a nonzero annihilator in $[\mathbb{F}t^{\widehat{\Integral}}]$.
	\item 
	Let $\Pi'$ be a profinite group, 
	and $\Psi\colon\Pi'\to\widehat{\pi}$ be a continuous epimorphism.
	Let $\psi\colon \pi\to \widehat{\Integral}$ be a group homomorphism.
	If 
	$H_1^{\widehat{\rho}',\widehat{\psi}'}(\Pi';\llbracket\mathbb{F}t^{\widehat{\Integral}}\rrbracket^k)$
	has a nonzero annihilator in $\llbracket\mathbb{F}t^{\widehat{\Integral}}\rrbracket$,
	then
	$H_1^{\widehat{\rho},\widehat{\psi}}(\widehat{\pi};\llbracket\mathbb{F}t^{\widehat{\Integral}}\rrbracket^k)$
	also has a nonzero annihilator in $\llbracket\mathbb{F}t^{\widehat{\Integral}}\rrbracket$.
	Here, $\widehat{\rho}'$ and $\widehat{\psi}'$ denote the pullbacks
	$\widehat{\rho}\circ\Psi$ and $\widehat{\psi}\circ\Psi$, respectively.
	\item 
	Let $\phi\colon \pi\to\Integral$ be a group homomorphism.
	Then
	$H^{\widehat{\rho},\widehat{\phi}}_n(\widehat{\pi};\llbracket \mathbb{F}t^{\widehat\Integral}\rrbracket^k)$
	has a nonzero annihilator in $\llbracket \mathbb{F}t^{\widehat{\Integral}}\rrbracket$
	if and only if 
	$H^{\rho,\phi}_n(\pi;[\mathbb{F}t^{\Integral}]^k)$ has finite dimension over $\mathbb{F}$.
	\end{enumerate}
\end{proposition}

\begin{proof}
Since $\pi$ is of type $\mathrm{FP}_\infty$,
we can take a projective resolution of $\Integral$ by finitely generated free left $[\Integral\pi]$--modules:
$$\xymatrix{
\cdots \ar[r] & C_2 \ar[r]^{\partial_2} & C_1 \ar[r]^{\partial_1} & C_0 \ar[r]^{\partial_0} &\Integral}.$$
Since $\pi$ is cohomologically good,
the completion operation yields a projective resolution of $\widehat{\Integral}$ 
by left profinite $\llbracket\widehat{\Integral}\widehat\pi\rrbracket$--modules:
$$\xymatrix{
\cdots \ar[r] & \widehat{C}_2 \ar[r]^{\widehat{\partial}_2} & \widehat{C}_1 \ar[r]^{\widehat{\partial}_1} 
& \widehat{C}_0 \ar[r]^{\widehat{\partial}_0} &\widehat{\Integral}},$$
(see \cite[Proposition 3.1]{JZ}).

For any group homomorphism $\phi\colon \pi\to \widehat{\Integral}$,
$H^{\rho,\phi}_\bullet(\pi;[\mathbb{F}t^{\widehat{\Integral}}]^k)$ and 
$H^{\widehat{\rho},\widehat{\phi}}_\bullet(\widehat{\pi};\llbracket \mathbb{F}t^{\widehat{\Integral}}\rrbracket^k)$
can be computed respectively as the homology of the chain complexes
$$\left(C^{\rho,\phi}_\bullet,\,\partial^{\rho,\phi}_\bullet\right)
=\left([\mathbb{F} t^{\widehat{\Integral}}]^k\otimes_{\rho,\phi}C_\bullet,\,\mathbf{1}_{k\times k}\otimes\partial_\bullet\right)$$
and
$$\left(C^{\widehat{\rho},\widehat{\phi}}_\bullet,\,\partial^{\widehat{\rho},\widehat{\phi}}_\bullet\right)
=\left(\llbracket \mathbb{F}t^{\widehat{\Integral}}\rrbracket^k\otimes_{\widehat{\rho},\widehat{\phi}}\widehat{C}_\bullet,\,
\mathbf{1}_{k\times k}\otimes\widehat{\partial}_\bullet\right).$$
We naturally identify $C_n^{\rho,\phi}$ as a dense $[\mathbb{F}t^{\widehat{\Integral}}]$--submodule of 
$C_n^{\widehat{\rho},\widehat{\phi}}$.
To be more explicit,
we fix a basis $c_n^1,\cdots,c_n^{r_n}$ for each $C_n$ and identify $C_n$ as $[\Integral\pi]^{r_n}$.
Let $\varepsilon^1,\cdots,\varepsilon^k$ be the standard basis of $\mathbb{F}^k$.
We furnish $C_n^{\rho,\phi}$ with the basis
$c_n^1\otimes\varepsilon^1,\cdots,c_n^1\otimes\varepsilon^n,\cdots,c_n^{r_n}\otimes\varepsilon^1,\cdots,c_n^{r_n}\otimes\varepsilon^k$.
Then $C_n^{\rho,\phi}$ is represented as $[\mathbb{F}t^{\widehat{\Integral}}]^{kr_n}$,
regarded as the space of column vectors of size $kr_n$ with entries in $[\mathbb{F}t^{\widehat{\Integral}}]$,
and $\partial_n^{\rho,\phi}$ is represented as a matrix $P_n$ of size $kr_{n-1}\times kr_n$ 
with entries in $[\mathbb{F}t^{\widehat{\Integral}}]$.
Similarly we represent $C_n^{\widehat{\rho},\widehat{\phi}}$ and 
$\partial_n^{\widehat{\rho},\widehat{\phi}}$ over $\llbracket\mathbb{F}t^{\widehat{\Integral}}\rrbracket$.
Plainly speaking,
the matrix of $\partial_n^{\widehat{\rho},\widehat{\phi}}$ is just the same as $P_n$,
while vectors in $C_n^{\widehat{\rho},\widehat{\phi}}$ 
allow more values of entries than in $C_n^{{\rho},{\phi}}$.

To prove the assertion (1),
%we observe that $[\mathbb{F}t^{\widehat{\Integral}}]$ is a domain.
%Then $H^{\rho\otimes\phi}_n(\pi;[\mathbb{F}t^{\widehat{\Integral}}]^k)$
%has a nonzero annihilator in $[\mathbb{F}t^{\widehat{\Integral}}]$ if and only if
%its rank over $[\mathbb{F}t^{\widehat{\Integral}}]$ equals zero.
%(The rank of a finitely generated module over a domain refers
%the dimension of its extension of scalars
%over the field of fractions.)
we claim that there exists some nonzero $a_n\in [\mathbb{F}t^{\widehat{\Integral}}]$, such that
$a_n$ annihilates the topological cokernel 
(namely, the quotient of the range by the closure of the image)
of the natural $[\mathbb{F}t^{\widehat{\Integral}}]$--linear homomorphism 
\begin{equation}\label{H-homomorphism}
H^{\rho,\phi}_n\left(\pi;[\mathbb{F}t^{\widehat{\Integral}}]^k\right)\to
H^{\widehat{\rho},\widehat{\phi}}_n\left(\widehat{\pi};\llbracket \mathbb{F}t^{\widehat{\Integral}}\rrbracket^k\right).
\end{equation}
This is equivalent to saying
\begin{equation}\label{a_nZ}
a_n\, \mathrm{Ker}\left(\partial_n^{\widehat{\rho},\widehat{\phi}}\right)
\subset \mathrm{clos}\left(\mathrm{Ker}\left(\partial_n^{{\rho},{\phi}}\right)\right),
\end{equation}
in the $\llbracket \mathbb{F}t^{\widehat{\Integral}}\rrbracket$--module 
$C_n^{\widehat{\rho},\widehat{\phi}}$.
Assuming the claim for the moment, we can prove the assertion (1) as follows:
If there exists some nonzero $\Delta_n\in [\mathbb{F}t^{\widehat{\Integral}}]$
which annihilates $H^{\rho,\phi}_n$,
it also annihilates the closure of its image in $H^{\widehat{\rho},\widehat{\phi}}_n$.
Then the claim implies 
that the nonzero element $a_n\Delta_n\in [\mathbb{F}t^{\widehat{\Integral}}]$ annihilates
$H^{\widehat{\rho},\widehat{\phi}}_n$,
as asserted.

The proof of the claim reduces to elementary matrix algebra.
To this end, 
we represent $\partial^{\rho,\phi}_n$ and $\partial^{\widehat{\rho},\widehat{\phi}}_n$
as a $kr_{n-1}\times kr_n$--matrix $P_n$ over $[\mathbb{F}t^{\widehat{\Integral}}]$,
with respect to the fixed bases as explained above.
There exist some $kr_n\times kr_{n-1}$--matrix $Q_n$ over $[\mathbb{F}t^{\widehat{\Integral}}]$
and some nonzero $a_n$ in $[\mathbb{F}t^{\widehat{\Integral}}]$, such that 
the matrix equation $P_nQ_nP_n=a_nP_n$ holds:
For example, one may first work with fractional coefficients, 
constructing a matrix $a^{-1}_nQ_n$ as a composite linear transformation,
which first projects the range of $P_n$
onto the image of $P_n$, and then lifts the image of $P_n$ 
to the domain of $P_n$.
Then, take $a_n$ as a common denominator of the entries, 
and take $Q_n$ as the matrix of the nominators.
Back to the claim,
if a vector $X\in \llbracket \mathbb{F}t^{\widehat{\Integral}}\rrbracket^{kr_n}$ 
satisfies the linear equation $P_nX=0$ in $\llbracket \mathbb{F}t^{\widehat{\Integral}}\rrbracket^{kr_n}$,
we construct a sequence of vectors $\{X^{(j)}\}_{j\in\Natural}$
in $[\mathbb{F}t^{\widehat{\Integral}}]^k$ 
by taking $X^{(j)}=a_nY^{(j)}-Q_nP_nY^{(j)}$,
where $\{Y^{(j)}\}_{j\in\Natural}$ 
is a sequence of vectors in $[\mathbb{F}t^{\widehat{\Integral}}]^{kr_n}$
converging to $X$. 
It follows that $P_nX^{(j)}=0$ holds in $[\mathbb{F}t^{\widehat{\Integral}}]^{kr_n}$
and $\{X^{(j)}\}_{j\in\Natural}$ converges to $a_nX$ in 
$\llbracket \mathbb{F}t^{\widehat{\Integral}}\rrbracket^{kr_n}$.
This verifies the claimed relation (\ref{a_nZ}).
So we are done with the assertion (1).

To prove the assertion (2), we observe that 
$\partial^{\rho,\psi}_n$ and $\partial^{\rho,\phi}_n$ can be represented 
as $kr_{n-1}\times kr_n$--matrices with entries 
in the subrings $[\mathbb{F}t^{\mathrm{Im}(\phi)}]$ and $[\mathbb{F}t^{\mathrm{Im}(\phi)}]$
of the domain $[\mathbb{F}t^{\widehat{\Integral}}]$, respectively,
using the fixed bases of $C^{\rho,\psi}_n$ and $C^{\rho,\phi}_n$.
Denote by $\phi^\sharp\colon \pi\to \mathrm{Im}(\phi)$ 
the epimorphism
which yields $\phi$ under the inclusion $\mathrm{Im}(\phi)\to \widehat{\Integral}$,
and similarly $\psi^\sharp\colon \pi\to \mathrm{Im}(\psi)$.
Since $\mathrm{Ker}(\psi)$ contains $\mathrm{Ker}(\phi)$,
$\phi^\sharp$ factorizes as 
the composition of $\psi^{\sharp}$ with
a unique surjective homomorphism $\mathrm{Im}(\phi)\to\mathrm{Im}(\psi)$.
The latter induces a surjective $\mathbb{F}$--algebra homomorphism
$[\mathbb{F}t^{\mathrm{Im}(\phi)}]\to[\mathbb{F}t^{\mathrm{Im}(\psi)}]$,
which makes the matrix of $\partial^{\rho,\psi^\sharp}_n$ abstractly 
a specialization of the matrix of $\partial^{\rho,\phi^\sharp}_n$,
for every dimension $n$.
Recall that the rank of a module over a (commutative) domain 
refers to the maximum cardinality for all subsets of linearly independent elements.
Specialization of a homomorphism 
does not increase rank of the image, and does not decrease rank of the kernel.
We obtain
\begin{eqnarray*}
& &
\mathrm{rank}_{[\mathbb{F}t^{\widehat{\Integral}}]}\left(\mathrm{Ker}\left(\partial^{\rho,\phi}_n\right)\right)
-
\mathrm{rank}_{[\mathbb{F}t^{\widehat{\Integral}}]}\left(\mathrm{Im}\left(\partial^{\rho,\phi}_n\right)\right)
\\
&=&
\mathrm{rank}_{[\mathbb{F}t^{\mathrm{Im}(\phi)}]}\left(\mathrm{Ker}\left(\partial^{\rho,\phi^\sharp}_n\right)\right)
-
\mathrm{rank}_{[\mathbb{F}t^{\mathrm{Im}(\phi)}]}\left(\mathrm{Im}\left(\partial^{\rho,\phi^\sharp}_{n+1}\right)\right) \\
&\leq&
\mathrm{rank}_{[\mathbb{F}t^{\mathrm{Im}(\psi)}]}\left(\mathrm{Ker}\left(\partial^{\rho,\psi^\sharp}_n\right)\right)
-
\mathrm{rank}_{[\mathbb{F}t^{\mathrm{Im}(\psi)}]}\left(\mathrm{Im}\left(\partial^{\rho,\psi^\sharp}_{n+1}\right)\right) \\
&=&
\mathrm{rank}_{[\mathbb{F}t^{\widehat{\Integral}}]}\left(\mathrm{Ker}\left(\partial^{\rho,\psi}_n\right)\right)
-
\mathrm{rank}_{[\mathbb{F}t^{\widehat{\Integral}}]}\left(\mathrm{Im}\left(\partial^{\rho,\psi}_{n+1}\right)\right),
\end{eqnarray*}
or equivalently,
$$
\mathrm{rank}_{[\mathbb{F}t^{\widehat{\Integral}}]}\left(
H^{\rho,\phi}_n\left(\pi;[\mathbb{F}t^{\widehat{\Integral}}]^k\right)\right)
\leq
\mathrm{rank}_{[\mathbb{F}t^{\widehat{\Integral}}]} \left(H^{\rho,\psi}_n\left(\pi;[\mathbb{F}t^{\widehat{\Integral}}]^k\right)\right).
$$
If $H^{\rho,\psi}_n$ has a nonzero annihilator in $[\mathbb{F}t^{\widehat{\Integral}}]$,
the rank of $H^{\rho,\psi}_n$ over $[\mathbb{F}t^{\widehat{\Integral}}]$ is zero.
Then the above inequality shows that 
the rank of $H^{\rho,\phi}_n$ over $[\mathbb{F}t^{\widehat{\Integral}}]$ is also zero.
For any finite generating set of $H^{\rho,\phi}_n$,
each generator is annihilated by some nonzero element of $[\mathbb{F}t^{\widehat{\Integral}}]$,
and their product is a nonzero annihilator of $H^{\rho,\phi}_n$.
This proves the assertion (2).

To prove the assertion (3), 
we consider the continuous $\llbracket\mathbb{F}t^{\widehat{\Integral}}\rrbracket$--module homomorphism
\begin{equation}\label{Psi_induced_H}
H_1^{\widehat{\rho}',\widehat{\psi}'}\left(\Pi';\llbracket\mathbb{F}t^{\widehat{\Integral}}\rrbracket^k\right)
\to
H_1^{\widehat{\rho},\widehat{\psi}}\left(\widehat{\pi};\llbracket\mathbb{F}t^{\widehat{\Integral}}\rrbracket^k\right),
\end{equation}
which is naturally induced by the continuous homomorphism $\Psi\colon \Pi'\to\widehat{\pi}$.
The assumption that $\Psi$ is surjective guarantees that (\ref{Psi_induced_H}) is also surjective.
In fact, this is a consequence of 
the Lyndon--Hochschild--Serre spectral sequence for profinite groups
\cite[Section 7.2]{Ribes--Zalesskii_book},
or more directly, 
an application of \cite[Corollary 7.2.6]{Ribes--Zalesskii_book}:
With notations thereof, we take the profinite group $G=\Pi'$ and its normal closed subgroup 
$K=\mathrm{Ker}(\Psi)$,
and take the projective right $\llbracket \widehat{\Integral}G\rrbracket$--module 
$B=\llbracket\mathbb{F}t^{\widehat{\Integral}}\rrbracket^k$
with the right $G$--action $\widehat{\psi}'t^{\widehat{\rho}'}$;
the topologically cofixed module $B_K=B/\mathrm{clos}(\langle bg-b\colon b\in B,g\in K\rangle)$ 
is the same as $B$ here, since $\widehat{\phi}'t^{\widehat{\rho}'}$ is the pullback of $t^\phi\rho$,
(see \cite[Lemma 6.3.3]{Ribes--Zalesskii_book} for the notation $B_K$). 
Therefore, any nonzero annihilator of 
$H^{\widehat{\rho}',\widehat{\psi}'}_1$ 
also annihilates 
$H^{\widehat{\rho},\widehat{\psi}}_1$.
This proves the assertion (3).

To prove the assertion (4), we notice that $[\mathbb{F}t^\Integral]$ is a principal ideal domain (PID).
Any finitely generated module $M$ over $[\mathbb{F}t^\Integral]$ is the direct sum
of a unique $[\mathbb{F}t^\Integral]$--torsion submodule $M_{\mathtt{tors}}$
and some submodule which is free over $[\mathbb{F}t^\Integral]$ of rank the same as the rank of $M$.
In particular, $M$ has finite dimension over $\mathbb{F}$ if and only if $M$ equals $M_{\mathtt{tors}}$.
For any $m\in\Natural$ we denote by $\phi_m\colon\pi\to \Integral/m\Integral$ 
the residue of $\phi$ modulo $m$.
We obtain a split short exact sequence of $[\mathbb{F}t^\Integral]$--modules
by the universal coefficient theorem:
\begin{equation}\label{UCT_sequence}
\xymatrix{
0 \ar[r] & 
H^{\rho,\phi}_n
\otimes_{[\mathbb{F}t^\Integral]}[\mathbb{F}t^{\Integral/m\Integral}]
\ar[r] &
H^{\rho,\phi_m}_n
\ar[r] &
\mathrm{Tor}^{[\mathbb{F}t^\Integral]}_1\left(
H^{\rho,\phi}_{n-1},\,
[\mathbb{F}t^{\Integral/m\Integral}]\right)
\ar[r] & 0,
}
\end{equation}
where
$H^{\rho,\phi}_n=H^{\rho,\phi}_n(\pi;[\mathbb{F}t^\Integral]^k)$
and
$H^{\rho,\phi_m}_n=H^{\rho,\phi_m}_n(\pi;[\mathbb{F}t^{\Integral/m\Integral}]^k)$.
Note that the splitting is natural in the twisting coefficient modules, 
although unnatural in the untwisted chain complexes,
(see \cite[Chapter V, Theorem 2.5]{Hilton--Stammbach_book}).
The profinite $\llbracket\mathbb{F}t^{\widehat{\Integral}}\rrbracket$--module 
$H^{\widehat{\rho},\widehat{\phi}}$ is the inverse limit of $H^{\rho,\phi_m}$.
If $H^{\rho,\phi}_n$ has finite dimension over $\mathbb{F}$,
it follows that $H^{\widehat{\rho},\widehat{\phi}}$
is isomorphic to $H^{\rho,\phi_m}$ 
for sufficiently divisible $m$,
for example, if $m$ is divisible by both the $\mathbb{F}$--dimension of $H^{\rho,\phi}_n$ and 
the $\mathbb{F}$--dimension of the $[\mathbb{F}t^\Integral]$--torsion submodule of $H^{\rho,\phi}_{n-1}$.
Then $H^{\widehat{\rho},\widehat{\phi}}$ is a finitely generated $[\mathbb{F}t^\Integral]$--torsion module,
and can be annihilated by its order in $[\mathbb{F}t^\Integral]$.
Conversely, if $H^{\rho,\phi}_n$ has infinite dimension over $\mathbb{F}$,
the $\llbracket \mathbb{F}t^{\widehat{\Integral}}\rrbracket$--direct summand 
$H^{\rho,\phi}_n\otimes_{[\mathbb{F}t^\Integral]}\llbracket \mathbb{F}t^{\widehat{\Integral}}\rrbracket$
in $H^{\widehat{\rho},\widehat{\phi}}_n$ 
contains a $\llbracket \mathbb{F}t^{\widehat{\Integral}}\rrbracket$--direct summand
which is nontrivial and free.
This means $H^{\widehat{\rho},\widehat{\phi}}$ cannot be annihilated by any nonzero element of  
$\llbracket \mathbb{F}t^\Integral\rrbracket$.
This proves the assertion (4).
\end{proof}

We proceed to prove Theorem \ref{fibered_cone}.
Let $(M_A,M_B,\Psi)$ be a profinite morphism setting of a $3$--manifold pair (Convention \ref{profinite_morphism_setting}).
Suppose that $\Psi\colon\widehat{\pi_A}\to\widehat{\pi_B}$ is an epimorphism.

Observe that $\Psi^*_\varepsilon(\phi)$ varies continuously in $H^1(M_A;\Real)$
if $(\phi,\varepsilon)$ varies continuously in 
$H^1(M_B;\Real)\times\mathrm{Hom}_\Integral(\mathrm{MC}(\Psi_*);\Real)$.
Since fibered cones are open in the real first cohomology,
it suffices to prove Thereom \ref{fibered_cone} for any pair $(\phi,\varepsilon)$
in the dense subset
$H^1(M_B;\Rational)\times \mathrm{Hom}_{\Integral}(\mathrm{MC}(\Psi_*),\Rational)$.
Moreover, since $\Psi^*_{m\varepsilon}(n\phi)=mn\cdot\Psi^*_{\varepsilon}(\phi)$,
it suffices to prove Theorem \ref{fibered_cone} for any positive integral multiples of $\phi$ and $\varepsilon$.

Therefore, 
we assume without loss of generality that 
$\phi\in H^1(M_B;\Integral)$ and $\varepsilon\in\mathrm{Hom}_\Integral(\mathrm{MC}(\Psi_*);\Integral)$.
In this case, $\Psi^*_\varepsilon\colon H^1(M_B;\Real)\to H^1(M_A;\Real)$ 
takes integral cohomology classes to integral ones,
so we also observe $\Psi^*_\varepsilon(\phi)\in H^1(M_A;\Integral)$.

For convenience,
we rewrite $\phi\in H^1(M_B;\Integral)$ as 
$\phi_B\colon \pi_B\to \Integral$,
and rewrite $\Psi^*_\varepsilon(\phi)\in H^1(M_A;\Integral)$ as 
$\psi_A\colon \pi_A\to \Integral$.
Besides, we denote by 
$\phi_A\colon \pi_A\to\widehat{\Integral}$
the restriction of the completion pullback of $\phi_B$,
namely, the composite homomorphism
$$\xymatrix{
\pi_A \ar[r]^-{\mathrm{incl}} & \widehat{\pi_A} \ar[r]^-{\Psi} & \widehat{\pi_B} \ar[r]^-{\widehat{\phi_B}} & \widehat{\Integral}.
}$$
Observe that the kernel of $\psi_A$ contains the kernel of $\phi_A$.
In fact, we can factorize $\phi_A$ alternatively as 
$$\xymatrix{
\pi_A \ar[r] &
H_A \ar[r]^-{\Psi_*^{\mathrm{MC}}} & 
H_B\otimes_\Integral\mathrm{MC}(\Psi_*) \ar[r]^-{\phi_B\otimes{1}} & 
\Integral\otimes_\Integral\mathrm{MC}(\Psi_*) \ar[r]^-{=} & 
\mathrm{MC}(\Psi_*) \ar[r]^-{\mathrm{incl}} &
\widehat{\Integral}.
}$$
On the other hand, $\psi_A$ factorizes as
$$\xymatrix{
\pi_A \ar[r] &
H_A \ar[r]^-{\Psi_*^{\mathrm{MC}}} & 
H_B\otimes_\Integral\mathrm{MC}(\Psi_*) \ar[r]^-{\phi_B\otimes1} & 
\Integral\otimes_\Integral\mathrm{MC}(\Psi_*) \ar[r]^-{=} & 
\mathrm{MC}(\Psi_*) \ar[r]^-{\varepsilon} &
\Integral
}$$
since we have assumed $\varepsilon$ having image in $\Integral$.
So $\psi_A$ kills at least whatever $\phi_A$ kills.

Suppose that $\psi_A$ is fibered for $M_A$.
For any finite-dimensional representation $\rho_B\colon\pi_B\to \mathrm{GL}(k,\mathbb{F})$
of $\pi_B$ over a finite field $\mathbb{F}$,
denote by $\rho_A\colon\pi_A\to  \mathrm{GL}(k,\mathbb{F})$
the restriction of the completion pullback of $\rho_B$,
namely, the composite homomorphism
$$\xymatrix{
\pi_A \ar[r]^-{\mathrm{incl}} & \widehat{\pi_A} \ar[r]^-{\Psi} & \widehat{\pi_B} \ar[r]^-{\widehat{\rho_B}} & \mathrm{GL}(k,\mathbb{F}).
}$$
Theorem \ref{fibered_TH} applies to $(M_A,\phi_A)$, implying that
$H^{\rho_A,\psi_A}_1(\pi_A;[\mathbb{F}t^\Integral]^k)
\cong H^{\rho_A,\psi_A}_1(M_A;[\mathbb{F}t^\Integral]^k)$ has finite dimension over $\mathbb{F}$.
In particular, $H^{\rho_A,\psi_A}_1(\pi_A;[\mathbb{F}t^{\widehat{\Integral}}]^k)$
has a nonzero annihilator in $[\mathbb{F}t^{\widehat{\Integral}}]$.
Since the kernel of $\psi_A$ contains the kernel of $\phi_A$,
it follows that 
$H^{\rho_A,\phi_A}_1(\pi_A;[\mathbb{F}t^{\widehat{\Integral}}]^k)$
has a nonzero annihilator in $[\mathbb{F}t^{\widehat{\Integral}}]$,
by Proposition \ref{annihilator_behavior} (2).
Then,
$H^{\widehat{\rho_A},\widehat{\phi_A}}_1(\widehat{\pi_A};\llbracket\mathbb{F}t^{\widehat{\Integral}}\rrbracket^k)$
has a nonzero annihilator in $\llbracket\mathbb{F}t^{\widehat{\Integral}}\rrbracket$,
by Proposition \ref{annihilator_behavior} (1).
Since $\Psi$ is an epimorphism,
$H^{\widehat{\rho_B},\widehat{\phi_B}}_1(\widehat{\pi_B};\llbracket\mathbb{F}t^{\widehat{\Integral}}\rrbracket^k)$
also has a nonzero annihilator in $\llbracket\mathbb{F}t^{\widehat{\Integral}}\rrbracket$,
by Proposition \ref{annihilator_behavior} (3).
Therefore,
$H^{\rho_B,\phi_B}_1(\pi_B;[\mathbb{F}t^{\widehat{\Integral}}]^k)
\cong H^{\rho_B,\phi_B}_1(M_B;[\mathbb{F}t^{\widehat{\Integral}}]^k)$
is finite dimensional over $\mathbb{F}$,
by Proposition \ref{annihilator_behavior} (4).
Since $\rho_B$ is arbitrary,
we apply Theorem \ref{fibered_TH} again to $(M_B,\phi_B)$, and conclude that $\phi_B$ is fibered.

This completes the proof of Theorem \ref{fibered_cone}.

\section{Profinite correspondence for Thurston-norm cones}\label{Sec-TN_cones}
In this section, we show that profinitely isomorphic pairs
of nonpositively curved $3$--manifolds with empty or tori boundary
have bijectively corresponding Thurston-norm cones and fibered cones
(Theorem \ref{profinite_isomorphism_npc} and Corollary \ref{fibered_cone_correspondence}).

\begin{theorem}\label{profinite_isomorphism_npc}
Let $(M_A,M_B,\Psi)$ be a profinite morphism setting of a $3$--manifold pair (Convention \ref{profinite_morphism_setting}).
Suppose that $\Psi\colon\widehat{\pi_A}\to\widehat{\pi_B}$ is an isomorphism.
Suppose that $M_A$ and $M_B$ are both irreducible with empty or tori boundary, 
and admit complete metrics of nonpositive curvature in the interior.
Suppose $\varepsilon\in\mathrm{Hom}_\Integral(\mathrm{MC}(\Psi_*),\Real)$.

If $\Psi^*_\varepsilon\colon H^1(M_B;\Real)\to H^1(M_A;\Real)$ is nondegenerate,
then $\Psi^*_\varepsilon$ witnesses 
a dimension-preserving bijective correspondence
between the Thurston-norm cones for $M_A$ and those for $M_B$.
Namely,
under the linear isomorphism $\Psi^*_\varepsilon$,
every Thurston-norm cone for $M_B$ 
projects onto a distinct and unique Thurston-norm cone for $M_A$.
\end{theorem}

\begin{remark}\label{profinite_isomorphism_npc_remarks}
	Let $M$ be an orientable connected compact $3$--manifold 
	which is irreducible with empty or tori boundary.
	Then, 
	$M$	admits a complete (Riemannian) metric of nonpositive sectional curvature
	in the interior
	if and only if $\pi_1(M)$ virtually embeds into a right-angled Artin group,
	(see \cite[Corollary 4.3.3 and Theorem 4.7.4]{AFW_book_group}).
	This condition is satisfied 
	if $M$ contains a hyperbolic piece in its geometric decomposition,
	or if $M$ is a graph manifold with nonempty tori boundary,
	otherwise  
	$M$ is a class of closed graph manifolds which can be characterized topologically.
	This condition turns out to depend only on the isomorphism type of $\widehat{\pi_1(M)}$.
	In fact, the first two cases follows
	because $\widehat{\pi_1(M)}$ 
	determines the geometric decomposition graph and the vertex geometries
	\cite[Theorem B and Theorem C]{WZ_decomposition}.
	The closed graph manifold case is implied by 
	the fact that $\widehat{\pi_1(M)}$ 
	determines the commensurability class of $M$ \cite[Theorem A]{Wilkes_graph_II}.	
\end{remark}

\begin{corollary}\label{fibered_cone_correspondence}
	Under the assumptions of Theorem \ref{profinite_isomorphism_npc},
	any nondegenerate $\Psi^*_\varepsilon$ also witnesses a bijective correspondence
	between the fibered cones for $M_A$ and those for $M_B$.
\end{corollary}

\begin{proof}
	By Theorem \ref{profinite_isomorphism_npc} and Theorem \ref{fibered_cone}, 
	every fibered cone for $M_A$ is projected isomorphically 
	from a fibered cone for $M_B$ under $\Psi^*_\varepsilon$.
	In particular, $M_B$ has at least as many fibered cones as $M_A$ does.
	Since $\Psi$ is an isomorphism, 
	it also follows by symmetry that 
	$M_A$ has no fewer fibered cones than $M_B$ does.
	Therefore, the fibered cones for $M_A$ must correspond exactly to 
	the fibered cones for $M_B$ under $\Psi^*_\varepsilon$.
\end{proof}

Note that the cone correspondence is only locally constant in $\varepsilon$.
For example, $-\varepsilon$ instead of $\varepsilon$ clearly yields a different cone correspondence.
However, we can derive some useful objects 
whose correspondence depends only on $\Psi$.
For any orientable connected compact $3$-manifold $M$,
we say that a linear subspace $V$ of $H^1(M;\Real)$ is a \emph{cone carrier},
if there is some Thurston-norm cone $\mathcal{C}$ 
and if $V$ is the smallest linear subspace which contains $\mathcal{C}$.
Every cone carrier $V$ determines a direct-summand $\Integral$--submodule of
$H_1(M;\Integral)_{\mathtt{free}}$, namely,
\begin{equation}\label{lattice_kernel_def}
K(V)=\{u\in H_1(M;\Integral)_{\mathtt{free}}\colon \phi(u)=0\mbox{ for all }\phi\in V\}.
\end{equation}
We refer to $K(V)$ as the \emph{lattice kernel} of $V$.
Therefore, the dimension of a Thurston-norm cone is the dimension of its cone carrier,
and is the corank of the lattice kernel.

\begin{corollary}\label{cone_carrier_correspondence}
	Under the assumptions of Theorem \ref{profinite_isomorphism_npc},
	$\Psi$ determines a bijective correspondence between the cone carriers
	for $M_A$ and those for $M_B$, as witnessed by any nondegenerate $\Psi^*_\varepsilon$.
	
	Moreover, a cone carrier $V_A$ for $M_A$ corresponds to a cone carrier $V_B$ for $M_B$,
	if and only if $\widehat{K(V_A)}$ projects isomorphically onto $\widehat{K(V_B)}$
	under $\Psi_*\colon \widehat{H_A}\to \widehat{H_B}$.
\end{corollary}

\begin{proof}
	For any cone carriers $V_A$ for $M_A$ and $V_B$ for $M_B$,
	take a direct-sum decomposition $H_A=K_A\oplus L_A$ where $K_A=K(V_A)$, and similarly $H_B=K_B\oplus L_B$.
	By considering the matrix of $\Psi_*$ with respect to any chosen bases of the direct summands,
	one may easily verify that $\Psi_*(\widehat{K_A})\subset\widehat{K_B}$ holds if and only if
	$\Psi_*^\varepsilon(K_A\otimes_\Integral\Real)\subset K_B\otimes_\Integral\Real$ holds
	for all $\varepsilon\in\mathcal{U}$, where $\mathcal{U}$ is some, or any, nonempty open subset of
	$\mathrm{Hom}_\Integral(\mathrm{MC}(\Psi_*;H_A,H_B),\Real)$.
	Note that $\Psi_*^\varepsilon(K_A\otimes_\Integral\Real)\subset K_B\otimes_\Integral\Real$
	is equivalent to $V_B\subset\Psi^*_\varepsilon(V_A)$.	
	Since the correspondence between the Thurston-norm cones 
	is fixed under small deformation of $\varepsilon$	(keeping $\Psi^*_\varepsilon$ nondegenerate),
	the induced correspondence between the cone carriers hence depends only on $\Psi$.
	Moreover, the above equivalence conditions imply the asserted characterization with lattice kernels.
\end{proof}

The rest of this section is devoted to the proof of Theorem \ref{profinite_isomorphism_npc}.
The idea is to pass to a corresponding pair of regular finite covers, using Theorem \ref{quasi-fibered},
and apply \ref{fibered_cone} to that pair.
Therefore, we need to relate the matrix coefficient modules associated to $\Psi$
and to its lift to the cover pair.
This is the job of Lemmas \ref{MC_cover} and \ref{U_construction} below.

\begin{definition}\label{corresponding_cover_def}
Let $(M_A,M_B,\Psi)$ be a profinite morphism setting of a $3$--manifold pair
where $\Psi$ is an isomorphism (Convention \ref{profinite_morphism_setting}).
Given any finite-index subgroup $\pi'_B$ of $\pi_B$, 
obtain a finite-index subgroup of $\pi_A$
as $\pi'_A=\pi_A\cap\Psi^{-1}(\mathrm{clos}(\pi'_B))$.
Denote by $M'_A\to M_A$ the finite cover associated to $\pi'_A$, and similarly 
by $M'_B\to M_B$ the finite cover associated to $\pi'_B$.
We refer to any pair of finite covers $M'_A\to M_A$ and $M'_B\to M_B$ 
obtained this way as a \emph{$\Psi$--corresponding pair}.
In this situation,
we denote by $\Psi'\colon\widehat{\pi'_A}\to\widehat{\pi'_B}$ 
the profinite isomorphism restricting $\Psi$.
We also adopt the notations 
$H'_A=H_1(M'_A;\Integral)_{\mathtt{free}}$ and similarly $H'_B$,
and $\mathrm{MC}(\Psi'_*)=\mathrm{MC}(\Psi'_*;H'_A,H'_B)$,
thinking of $(M'_A,M'_B,\Psi')$ 
as a profinite morphism setting of a $3$--manifold pair.
\end{definition}

\begin{lemma}\label{MC_cover}
	Let $(M_A,M_B,\Psi)$ be a profinite morphism setting of a $3$--manifold pair
	where $\Psi$ is an isomorphism (Convention \ref{profinite_morphism_setting}).
	For any $\Psi$--corresponding pair of finite covers 
	$M'_A\to M_A$ and $M'_B\to M_B$,
	$\mathrm{MC}(\Psi_*)$ is commensurable in $\widehat{\Integral}$
	with a $\Integral$--submodule of $\mathrm{MC}(\Psi'_*)$.
	In particular,
	$$\mathrm{rank}_\Integral(\mathrm{MC}(\Psi_*))\leq\mathrm{rank}_\Integral(\mathrm{MC}(\Psi'_*)).$$	
\end{lemma}

\begin{proof}
	We obtain a $\Integral$--submodule of $H'_A$ as
	\begin{equation}\label{def_K}
		K'_A=\mathrm{Ker}\left(\xymatrix{ H'_A \ar[r]^-{\mathrm{cov}_*} & H_A}\right),
	\end{equation}
	and a $\Integral$--submodule of $H_A$ as
	\begin{equation}\label{def_L}
		L_A=\mathrm{Im}\left(\xymatrix{ H'_A \ar[r]^-{\mathrm{cov}_*} & H_A}\right),
	\end{equation}
	and similarly $K'_B$ and $L_B$.
	Choose a $\Integral$--submodule $L'_A$ of $H'_A$ which projects isomorphically onto $L_A$,
	so we decompose $H'_A$ as the direct sum $K'_A\oplus L'_A$.
	Similary, we decompose $H'_B$ as $K'_B\oplus L'_B$.
	Choose bases of $K'_A$, $L'_A$, $K'_B$, and $L'_B$,
	and also furnish $L_A$ and $L_B$ with the bases projected from $L'_A$ and $L'_B$.
	
	With respect to the above $\Integral$--submodules with chosen bases,
	we represent 
	$\Psi'_*\in\mathrm{Hom}(\widehat{K'_A}\oplus\widehat{L'_A}, \widehat{K'_B}\oplus\widehat{L'_B})$ 
	as a square matrix over $\widehat{\Integral}$.
	The matrix of $\Psi'_*$ 
	can be divided into four blocks representing the homomorphism summands, 
	as indicated by the following arrangement:
	$$
	\mathrm{Hom}\left(\widehat{K'_A}\oplus\widehat{L'_A},\,\widehat{K'_B}\oplus\widehat{L'_B}\right)
	\cong
	\left[\begin{array}{cc}
	\mathrm{Hom}\left(\widehat{K'_A},\widehat{K'_B}\right) & \mathrm{Hom}\left(\widehat{L'_A},\widehat{K'_B}\right)\\
	\mathrm{Hom}\left(\widehat{K'_A},\widehat{L'_B}\right) & \mathrm{Hom}\left(\widehat{L'_A},\widehat{L'_B}\right)
	\end{array}\right].
	$$
	%in $\mathrm{Hom}(\widehat{K'_A},\widehat{K'_B})$, $\mathrm{Hom}(\widehat{L'_A},\widehat{K'_B})$,
	%$\mathrm{Hom}(\widehat{L'_A},\widehat{K'_B})$, and 
	%$\mathrm{Hom}(\widehat{L'_A},\widehat{L'_B})$.
	The $\mathrm{Hom}(\widehat{K'_A},\widehat{L'_B})$--block must be zero because of the commutative diagram
	\begin{equation}\label{Psi_cov}
		\xymatrix{
		\widehat{H'_A} \ar[r]^-{\Psi'_*} \ar[d]_{\mathrm{cov}_*} & \widehat{H'_B} \ar[d]^{\mathrm{cov}_*}\\
		\widehat{H_A} \ar[r]^-{\Psi_*} & \widehat{H_B}.		
	}
	\end{equation}
	The profinite submodule $\mathrm{Hom}(\widehat{L'_A},\widehat{L'_B})$ 
	can be identified with $\mathrm{Hom}(\widehat{L_A},\widehat{L_B})$,
	and the corresponding block represents the matrix of 
	the isomorphism $\widehat{L_A}\to \widehat{L_B}$ restricting $\Psi_*$.
	
	The above block description shows that
	$\mathrm{MC}(\Psi'_*;H'_A,H'_B)$ contains
	$\mathrm{MC}(\Psi'_*;L'_A,L'_B)$,
	which equals $\mathrm{MC}(\Psi_*;L_A,L_B)$,
	by Proposition \ref{MC_property} (1).
	Moreover, $\mathrm{MC}(\Psi_*;L_A,L_B)$
	is commensurable in $\widehat{\Integral}$ with $\mathrm{MC}(\Psi_*;H_A,H_B)$,
	by Proposition \ref{MC_property} (3).
	Therefore, we see that $\mathrm{MC}(\Psi_*;H_A,H_B)$
	is commensurable with a $\Integral$--submodule of $\mathrm{MC}(\Psi'_*;H'_A,H'_B)$, 
	as asserted.
	In particular, we obtain the rank inequality
	$$\mathrm{rank}_\Integral(\mathrm{MC}(\Psi_*;H_A,H_B))=
	\mathrm{rank}_\Integral(\mathrm{MC}(\Psi_*;L_A,L_B))\leq
	\mathrm{rank}_\Integral(\mathrm{MC}(\Psi'_*;H'_A,H'_B)),$$
	as asserted.
\end{proof}

\begin{lemma}\label{U_construction}
	Let $(M_A,M_B,\Psi)$ be a profinite morphism setting of a $3$--manifold pair
	where $\Psi$ is an isomorphism (Convention \ref{profinite_morphism_setting}).
	For any $\Psi$--corresponding pair of finite covers 
	$M'_A\to M_A$ and $M'_B\to M_B$,
	there exists some (point-set topological) 
	open dense subset $\mathcal{U}'$ 
	of $\mathrm{Hom}_\Integral(\mathrm{MC}(\Psi_*),\Real)$
	with the following property:
		For any	$\varepsilon\in\mathcal{U}'$,
	there exists some 
	$\varepsilon'\in\mathrm{Hom}_\Integral(\mathrm{MC}(\Psi'_*),\Real)$,
	such that 
	the $\varepsilon'$--specialization $(\Psi')_*^{\varepsilon'}$ of $\Psi'_*$ is nondegenerate,
	and	such that the following diagram commutes:
	$$\xymatrix{
		H_1(M'_A;\Real) \ar[r]^-{(\Psi')_*^{\varepsilon'}} \ar[d]_{\mathrm{cov}_*} & H_1(M'_B;\Real) 
		\ar[d]^{\mathrm{cov}_*}\\
		H_1(M_A;\Real) \ar[r]^-{\Psi^{\varepsilon}_*} & H_1(M_B;\Real)		
	}$$
	Moreover, one may require $\mathcal{U}'$ to be invariant 
	under nonzero real scalar multiplication, and 
	$\varepsilon'$ to depend continuously on $\varepsilon$.
	%Here, $m'\Psi^{\varepsilon}_*$ denotes the $m'$--scalar multiple of the $\varepsilon$--perturbation 
	%$\Psi^\varepsilon_*$ of $\Psi_*$.
\end{lemma}

%
%\begin{lemma}\label{U_construction}
	%For any $\Psi$--corresponding pair of finite covers 
	%$M'_A\to M_A$ and $M'_B\to M_B$,
	%there exists some positive integer $m'$ and some open dense subset $\mathcal{U}'$ 
	%of the profinite module
	%$\mathrm{Hom}_\Integral(\mathrm{MC}(\Psi_*;H_A,H_B),\widehat{\Integral})$
	%with the following property:
	%
	%For every homomorphism $\varepsilon\colon \mathrm{MC}(\Psi_*;H_A,H_B)\to \Integral$ in $\mathcal{U}'$,
	%there exists some homomorphism $\varepsilon'\colon \mathrm{MC}(\Psi'_{*};H'_A,H'_B)\to \Integral$,
	%such that 
	%the $\varepsilon'$--perturbation $(\Psi')_*^{\varepsilon'}$ of $\Psi'_*$ is injective,
	%and	such that the following diagram commutes:
	%$$\xymatrix{
		%H'_A \ar[r]^-{(\Psi')_*^{\varepsilon'}} \ar[d]_{\mathrm{cov}_*} & H'_B \ar[d]^{\mathrm{cov}_*}\\
		%H_A \ar[r]^-{m'\Psi^{\varepsilon}_*} & H_B.		
	%}$$
	%Here, $m'\Psi^{\varepsilon}_*$ denotes the $m'$--scalar multiple of the $\varepsilon$--perturbation 
	%$\Psi^\varepsilon_*$ of $\Psi_*$.
%\end{lemma}
%

\begin{proof}

	Remember that $\mathrm{MC}(\Psi'_*)=\mathrm{MC}(\Psi'_*;H'_A,H'_B)$
	is a free $\Integral$--module of finite rank (Proposition \ref{MC_property}),
	and $\mathrm{MC}(\Psi_*)=\mathrm{MC}(\Psi_*;H_A,H_B)$ is commensurable 
	with a $\Integral$--submodule of it (Lemma \ref{MC_cover}).
	We denote by $E$ the smallest direct summand of $\mathrm{MC}(\Psi'_*)$
	that is commensurable with $\mathrm{MC}(\Psi_*)$.
	Take a positive integer $m'$ such that $\mathrm{MC}(\Psi_*)$ contains $m'E$.
		
	Choose a basis $x_1,\cdots,x_r,x_{r+1},\cdots,x_{r+s}\in\widehat{\Integral}$ 
	of $\mathrm{MC}(\Psi'_*)$,
	such that $x_1,\cdots,x_r$ form a basis of $E$.
	To avoid confusion,
	we treat $\mathrm{MC}(\Psi'_*)$ 
	formally as the degree--$1$ homogeneous summand of a polynomial ring $\Integral[X_1,\cdots,X_{r+s}]$,
	and together with 
	%a distinguished $\widehat{\Integral}$--specialization
	%%(namely, a ring homomorphism with value in $\widehat{\Integral}$)
	%$\xi_0\colon \Integral[X_1,\cdots,X_{r+s}]\to \widehat{\Integral}$,
	%namely, the ring homomorphism determined by $\xi_0(X_i)=z_i$ for all $i=1,\cdots,r+s$.
	a distinguished homomorphism 
	$\xi_0\in \mathrm{Hom}_{\Integral}(\mathrm{MC}(\Psi'_*),\widehat{\Integral})$,
	such that $\xi_0(X_i)=x_i$ for all $i=1,\cdots,r+s$.
	We also adopt the same $\Integral$--modules 
	$L_A$, $K'_A$, $L'_A$ and $L_B$, $K'_B$, $L'_B$
	with fixed choices of their bases,
	as in the proof of Lemma \ref{MC_cover}.
	This way,
	we treat the matrix of $(\Psi'_*)^{\mathrm{MC}}$ 
	as a square matrix over $\Integral[X_1,\cdots,X_{r+s}]$,
	(while thinking of the matrix of $\Psi'_*$ as over $\widehat{\Integral}$; 
	see Definition \ref{MC_def}).
	The determinant of the matrix of $(\Psi'_*)^{\mathrm{MC}}$ 
	is a homogeneous polynomial 
	$F$ in $\Integral[X_1,\cdots,X_{r+s}]$ of degree $b_1(M'_A)=b_1(M'_B)$.
	%(The polynomial $F$ 
	%depends on the chosen bases for $H'_A$, $H'_B$, and $\mathrm{MC}(\Psi'_*)$.)
	Moreover, 
	since the block form of the matrix is upper triangular,
	there exists a factorization 
	$F=F_{K'}F_{L'}$ in $\Integral[X_1,\cdots,X_{r+s}]$,
	where $F_{L'}$ is a homogeneous polynomial in $\Integral[X_1,\cdots,X_r]$
	of degree $b_1(M_A)=b_1(M_B)$;
	(compare the commutative diagram (\ref{Psi_cov})).
	We obtain a continuous map as follows:
	\begin{eqnarray*}
	\mathrm{Hom}_\Integral\left(\mathrm{MC}(\Psi'_*),\widehat{\Integral}\right) & \longrightarrow &\widehat{\Integral} \\
	\xi &\mapsto & F(\xi(X_1),\cdots,\xi(X_{r+s}))
	\end{eqnarray*}
	Since $\Psi'_*$ is invertible,
	$\xi_0$ is mapped to a nonzero value $F(x_1,\cdots,x_{r+s})$ in $\widehat{\Integral}$.
	Then there are $c_1,\cdots,c_r,c_{r+1},\cdots,c_{r+s}\in\Integral$,
	such that $F(c_1,\cdots,c_{r+s})\in\Integral$ is also nonzero.
	In particular, $c_1,\cdots,c_r$ are not all zero,
	since $F_{L'}(c_1,\cdots,c_r)$ is nonzero.
	It follows that there exist degree--$1$ homogeneous polynomials 
	$C_{r+1},\cdots,C_{r+s}$ in $\Rational[X_1,\cdots,X_r]$, such that 
	$$F^C(X_1,\cdots,X_r)=F(X_1,\cdots,X_r,C_{r+1}(X_1,\cdots,X_r),\cdots,C_{r+s}(X_1,\cdots,X_r))$$
	is a nonzero homogeneous polynomial in $\Rational[X_1,\cdots,X_r]$.
	(For example, assuming $c_1\neq0$, one may take
	$C_i=c_iX_1/c_1$ for $i=r+1,\cdots,r+s$.)
	%In particular, $F(X_1,\cdots,X_r,c_{r+1},\cdots,c_{r+s})$ 
	%is nonzero polynomial in $\Integral[X_1,\cdots,X_r]$.
	We define a continuous map $f$ as follows:
	\begin{eqnarray*}
	\mathrm{Hom}_\Integral\left(\mathrm{MC}(\Psi_*),\Real\right) 
	&\stackrel{f}{\longrightarrow} & \Real \\
	\eta &\mapsto& F^C\left(\frac{\eta(m'X_1)}{m'},\cdots,\frac{\eta(m'X_r)}{m'}\right)
	\end{eqnarray*}
	Note that the degree--$1$ homogeneous summand of the polynomial subring $\Integral[X_1,\cdots,X_r]$	is treated as $E$,
	so $m'X_1,\cdots,m'X_r$ are all contained in $\mathrm{MC}(\Psi_*)$ by our requirement on $m'$.
	Plainly speaking, 
	if one identifies $\mathrm{MC}(\Psi_*)$ as the degree--$1$ homogeneous summand of 
	a polynomial ring $\Integral[Y_1,\cdots,Y_r]$, 
	then $X_1,\cdots,X_r$ can be represented uniquely 
	as some $\Rational$--linear combination of $Y_1,\cdots,Y_r$
	in $\Rational[Y_1,\cdots,Y_r]$ with a common denominator $m'$.
	Then $f$ is explicitly a homogeneous polynomial function of the dual coordinates 
	$\eta=(\eta_1,\cdots,\eta_r)$	in $\Real^r$,	
	and $f$ is not constant zero.
	We take an open dense subset 
	of $\mathrm{Hom}_\Integral(\mathrm{MC}(\Psi_*),\Real)$
	as 
	$$\mathcal{U}'=f^{-1}\left(\Real\setminus\{0\}\right).$$
	Note that $\mathcal{U}'$ is invariant under nonzero scalar multiplication,
	because of the homogeneity of $F_C$.
	
	It remains to check that 
	the $\mathcal{U}'$ satisfy the asserted property of Lemma \ref{U_construction}.
	In fact, for any $\varepsilon$ in $\mathcal{U}'$,
	%in $\mathrm{Hom}_\Integral(\mathrm{MC}(\Psi_*),\Real)\cap\mathcal{U}'$, 
	we take the asserted $\varepsilon'$
	in $\mathrm{Hom}_\Integral(\mathrm{MC}(\Psi'_*),\Real)$ as follows:
	\begin{equation}\label{epsilon_cov}
	\varepsilon'(X_i)=
	\begin{cases}
	{\varepsilon(m'X_i)}/{m'} & i=1,\cdots,r \\
	C_i\left({\varepsilon(m'X_1)}/{m'},\cdots,{\varepsilon(m'X_r)}/{m'}\right) & i=r+1,\cdots,r+s
	\end{cases}
	\end{equation}
	Note that $m'X_1,\cdots,m'X_r$ all lie in $\mathrm{MC}(\Psi'_*)$,
	because of our requirement on the positive integer $m'$.
	We show that $\varepsilon'$ satisfies the asserted properties.
		
	With respect to the fixed basis of $H'_A$ and $H'_B$,
	the matrix of the $\varepsilon'$--specialization 
	$(\Psi'_*)^{\varepsilon'}$
	has determinant 
	$$
	F\left(\varepsilon'(X_1),\cdots,\varepsilon'(X_{r+s})\right)=
	F^C\left(\frac{\varepsilon(m'X_1)}{m'},\cdots,\frac{\varepsilon(m'X_r)}{m'}\right)
	=f(\varepsilon)\neq0,
	$$
	so $(\Psi'_*)^{\varepsilon'}$ is nondegenerate.
	
	To see the asserted commutative diagram,
	we observe the following simplifications:
	It suffices to check for any $\varepsilon$ in
	$\mathrm{Hom}_\Integral(\mathrm{MC}(\Psi_*),\Rational)\cap\mathcal{U}'$,
	since this is a dense subset in $\mathcal{U}'$, 
	while $\varepsilon'$ in (\ref{epsilon_cov}) depends continuously on $\varepsilon$;
	moreover, 
	it suffices to check for any $\varepsilon$ in
	$\mathrm{Hom}_\Integral(\mathrm{MC}(\Psi_*),m'\Integral)\cap\mathcal{U}'$,
	since $\mathcal{U}'$ is rescaling-invariant 
	while $\varepsilon'$ in (\ref{epsilon_cov}) is proportional to $\varepsilon$.
	Below we assume without loss of generality 
	$\varepsilon\in \mathrm{Hom}_{\Integral}(\mathrm{MC}(\Psi_*),m'\Integral)\cap\mathcal{U}'$.
	In this case,	
	we observe that $\varepsilon$ extends uniquely 
	to become a $\Integral$--linear homomorphism $E\to \Integral$,
	since $m'E$ is contained in $\mathrm{MC}(\Psi_*)$.
	We also observe $\varepsilon'(X_i)=\varepsilon(X_i)$
	for all $i=1,\cdots,r$ with the extension of $\varepsilon$,	
	by (\ref{epsilon_cov}).
	It suffices to compare the $\mathrm{Hom}_\Integral(L'_A,L'_B)$--block $P'$
	of the matrix of $(\Psi')_*^{\varepsilon'}$ 
	and	the matrix $P$ of the restricted homomorphism $\Psi_*\colon L_A\to L_B$,
	under the identification $L'_A\cong L_A$ and $L'_B\cong L_B$.
	This is because $L_A$ and $L_B$ are the images of $H'_A$ and $H'_B$ 
	in $H_A$ and $H_B$.
	The matrix coefficient module $\mathrm{MC}(\Psi'_*;L'_A,L'_B)=\mathrm{MC}(\Psi_*;L_A,L_B)$
	is a submodule of $E$, so its elements are all $\Integral$--linear combinations of $X_1,\cdots,X_r$.
	Because $\varepsilon'(X_i)=\varepsilon(X_i)$ holds for all $i=1,\cdots,r$,
	we obtain $P'=P$ from the commutative diagram (\ref{Psi_cov}).
	This shows that 
	$(\Psi')_*^{\varepsilon'}|_{H'_A}\colon H'_A\to H'_B$ 
	descends to $\Psi_*^\varepsilon|_{H_A}\colon H_A\to H_B$,
	and hence,
	$(\Psi')_*^{\varepsilon'}\colon H_1(M'_A;\Real)\to H_1(M'_B;\Real)$ 
	descends to $\Psi_*^\varepsilon\colon H_1(M_A;\Real)\to H_1(M_B;\Real)$.
	In other words, the asserted diagram commutes.
\end{proof}

We proceed to prove Theorem \ref{profinite_isomorphism_npc}.
Let $(M_A,M_B,\Psi)$ be a profinite morphism setting of a $3$--manifold pair (Convention \ref{profinite_morphism_setting}).
Suppose that $\Psi\colon\widehat{\pi_A}\to\widehat{\pi_B}$ is an isomorphism.
Suppose that $M_A$ and hence $M_B$ both have empty or tori boundary,
and admit complete Riemannian metrics of nonpositive sectional curvature
in the interior.

Take a finite cover $M'_A$ of $M_A$ which satisfies the conclusion of Theorem \ref{quasi-fibered}.
We obtain an open dense subset $\mathcal{U}'$ of $\mathrm{Hom}_\Integral(\mathrm{MC}(\Psi_*),\Real)$
as provided by Lemma \ref{U_construction}.
Note that the conclusion in Lemma \ref{U_construction} implies 
the dual commutative diagram
\begin{equation}\label{Psi_dual_cov}
\xymatrix{
	H^1(M_B;\Real) \ar[r]^-{\Psi^*_{\varepsilon}} \ar[d]_{\mathrm{cov}^*} & H^1(M_A;\Real) \ar[d]^{\mathrm{cov}^*}\\
	H^1(M'_B;\Real) \ar[r]^-{(\Psi')^*_{\varepsilon'}}  & H^1(M'_A;\Real),
}
\end{equation}
for any $\varepsilon$ and $\varepsilon'$ as thereof.
In particular, $\Psi^*_\varepsilon$ will necessarily be nondegenerate for any $\varepsilon\in\mathcal{U}'$, 
since $(\Psi')^*_{\varepsilon'}$ is nondegenerate.

We observe the following simplifications for proving Theorem \ref{profinite_isomorphism_npc}:
First, it suffices to prove for any $\varepsilon\in\mathcal{U}'$.
In fact, for the general case where 
$\varepsilon$ is any homomorphism in $\mathrm{Hom}_\Integral(\mathrm{MC}(\Psi_*),\Real)$
with nondegenerate $\Psi^*_\varepsilon$,
we can approach $\varepsilon$ with $\tilde{\varepsilon}\in\mathcal{U}'$,
then $\Psi^*_{\varepsilon}$ will witness the same Thurston-norm cone correspondence
as $\Psi^*_{\tilde{\varepsilon}}$ does when $\tilde{\varepsilon}$ is sufficiently close to $\varepsilon$.
This is because $\Psi^*_{\tilde{\varepsilon}}$ varies continuously 
in $\mathrm{Hom}_\Real(H^1(M_B;\Real),H^1(M_A;\Real))$
as $\tilde{\varepsilon}$ varies in $\mathrm{\mathrm{Hom}_\Integral(\mathrm{MC}(\Psi_*),\Real)}$,
and because being nondegenerate is an open condition about $\tilde{\varepsilon}$.
Moreover, it suffices to show that any codimension--$0$ Thurston-norm cone for $M_B$
projects onto a codimension--$0$ Thurston-norm cone for $M_A$
under $\Psi^*_\varepsilon$.
This is because any lower-dimensional cone can be uniquely specified
using top-dimensional ones by taking suitable intersection of their closures,
and because $\Psi^*_\varepsilon$ is a linear isomorphism under the nondegeneracy assumption.

Without loss of generality, we assume $\varepsilon\in\mathcal{U'}$ below,
and prove the asserted correspondence only for codimension--$0$ Thurston-norm cones.

Under the covering induced embedding $H^1(M_B;\Real)\to H^1(M'_B;\Real)$,
the intersection of the Thurston-norm unit ball for $M'_B$ with the embedded image of $H^1(M_B;\Real)$
coincides with the Thurston-norm unit ball for $M_B$ dilated by the covering degree,
(see Theorem \ref{TN_cover}).
By convexity of Thurston-norm unit balls, 
the intersection of any closed Thurston-norm cone for $M'_B$
with the image of $H^1(M_B;\Real)$ is 
either the origin or the image of a closed Thurston-norm cone for $M_B$.
The same description works for $M'_A\to M_A$ as well.

Given $\varepsilon\in\mathcal{U}'$,
we obtain $\varepsilon'\colon \mathrm{MC}(\Psi'_*)\to\Integral$
as in the conclusion of Lemma \ref{U_construction}. 
For any codimension--$0$ closed Thurston-norm cone $\bar{\mathcal{C}}$ for $M_A$,
there exists some codimension--$0$ closed Thurston-norm cone $\bar{\mathcal{C}}'$ for $M'_A$,
such that the interior of $\bar{\mathcal{C}}'$ is a fibered cone, 
and such that the intersection of $\bar{\mathcal{C}}'$ 
with the image of $H^1(M_A;\Real)$ is the image of $\bar{\mathcal{C}}$,
by the construction of $M'_A$ and Theorem \ref{quasi-fibered}.
Under the inverse of the linear isomorphism $(\Psi')^*_{\varepsilon'}$,
the embedded image of $H^1(M_A;\Real)$ projects isomorphically onto the embedded image of $H^1(M_B;\Real)$,
by the commutative diagram (\ref{Psi_dual_cov}).
The closed cone $\bar{\mathcal{C}}'$ projects into a unique codimension--$0$ closed Thurston-norm cone for $M'_B$,
(which is also the closure of some fibered cone,) by Theorem \ref{fibered_cone}.
It follows that $\bar{\mathcal{C}}$ projects into a unique closed Thurston-norm cone for $M_B$,
which is necessarily of codimension--$0$,
under the inverse of $\Psi^*_{\varepsilon}$.
Since the codimension--$0$ closed Thurston-norm cones for $M_A$
fill $H^1(M_A;\Real)$ altogether, 
the number of codimension--$0$ cones for $M_B$ is at most that number for $M_A$.
It follows by symmetry that $M_A$ has the same number of codimension--$0$ cones 
as $M_B$ does.
Therefore,
every codimension--$0$ closed Thurston-norm cone for $M_A$ must project
onto the codimension--$0$ closed Thurston-norm cone for $M_B$
which contains its image under $\Psi^*_\varepsilon$ inverse.
In other words, every codimension--$0$ Thurston-norm cone for $M_B$
project onto a unique and distinct codimension--$0$ Thurston-norm cone for $M_A$
under $\Psi^*_\varepsilon$, as asserted.

This completes the proof of Theorem \ref{profinite_isomorphism_npc}.

\section{Profinite correspondence of the Thurston-norm unit ball}\label{Sec-TN_unit_ball}
In this section, we prove Theorems \ref{main_xregular} and \ref{main_Thurston_norm},
which are restated in the effective form 
as Theorem \ref{profinite_isomorphism_hyperbolic}
and Corollaries \ref{fibered_cone_correspondence} and \ref{TN_correspondence}.

\begin{theorem}\label{profinite_isomorphism_hyperbolic}
Let $(M_A,M_B,\Psi)$ be a profinite morphism setting of a $3$--manifold pair (Convention \ref{profinite_morphism_setting}).
Suppose that $\Psi\colon\widehat{\pi_A}\to\widehat{\pi_B}$ is an isomorphism.
%Suppose that $M_A$, and hence $M_B$, 
%has empty or tori boundary and 
%admits a complete hyperbolic metric in the interior.
Suppose that $M_A$ and $M_B$
both have positive first Betti number,
and admit complete hyperbolic metrics of finite volume in the interior.

Then, there exists some unit $\mu\in\widehat{\Integral}^\times$, such that 
$\mathrm{MC}(\Psi_*)$ is the $\Integral$--submodule $\mu\Integral$ of $\widehat{\Integral}$.
Note that the scalar multiplication by $\mu^{-1}$ on $\widehat{\Integral}$ 
defines a homomorphism $1/\mu\in\mathrm{Hom}_\Integral(\mathrm{MC}(\Psi_*),\Real)$
with image in $\Integral$.
Therefore, $\Psi^*_{1/\mu}\colon H^1(M_B;\Real)\to H^1(M_A;\Real)$ 
is the $\Real$--linear extension of 
a $\Integral$--module isomorphism $H^1(M_B;\Integral)\to H^1(M_A;\Integral)$.
\end{theorem}

%witnesses a linear isomorphism between the Thurston-norm unit ball for $M_B$ and that for $M_A$.
%In other words, $\Psi^*_{1/\mu}$ is Thurston-norm preserving.

\begin{corollary}\label{fiber_surface_correspondence}
	Under the assumptions of Theorem \ref{profinite_isomorphism_hyperbolic},
	if $S_A$ is a connected fiber surface of $M_A$, with respect to some fibration over a circle, 
	then there exists a homotopically unique connected fiber surface of $M_B$,
	such that the closure of $\pi_1(S_A)$ projects isomorphically onto the closure of $\pi_1(S_B)$
	under $\Psi$. 
\end{corollary}

\begin{proof}
	Let $\phi_B\colon\pi_B\to \Integral$ be the surjective homomorphism 
	induced by the fibration of $M_B$ with a connected fiber surface $S_B$.
	Then $\pi_1(S_B)$ can be identified as the kernel of $\phi_B$,
	and $\phi_B$ represents a primitive integral cohomology class in $H^1(M_B;\Integral)$,
	which is fibered for $M_B$.
	By Theorem \ref{profinite_isomorphism_hyperbolic},
	$\Psi^*_{1/\mu}$ maps $\phi_B$ 
	to a primitive integral cohomology class $\psi_A$ in $H^1(M_A;\Integral)$.
	Moreover, $\psi_A$ is fibered for $M_A$ by Corollary \ref{fibered_cone_correspondence}.
	Denote by $S_A$ a connected fiber surface for the fibration of $M_A$ associated to $\psi_A$,
	so $\pi_1(S_A)$ can be identified with the kernel of
	the induced surjective homomorphism $\psi_A\colon\pi_A\to \Integral$.
	
	%Note that $\phi_B$ factorizes as the composite homomorphism $\pi_B\to H_B\to \Integral$.
	Denote by $\phi_{B*}\colon H_B\to \Integral$ the homomorphism induced by $\phi_B$,
	and similarly $\psi_{A*}\colon H_A\to \Integral$.
	The preimage of the closure of $\mathrm{Ker}(\phi_{B*})$ in $\widehat{H_B}$,
	with respect to $\widehat{\pi_B}\to\widehat{H_B}$,
	is precisely the closure of $\pi_1(S_B)$ in $\widehat{\pi_B}$.
	The similar description holds for the closure of $\pi_1(S_A)$ in $\widehat{\pi_A}$.
	Therefore, it suffices to show that the isomorphism $\Psi_*\colon \widehat{H_A}\to \widehat{H_B}$	
	maps the kernel closure of $\psi_{A*}$ onto the kernel closure of $\phi_{B*}$.
	In fact, this is an easy consequence of the characterization $\mathrm{MC}(\Psi_*)=\mu\Integral$.
	
	To be precise,
	Theorem \ref{profinite_isomorphism_hyperbolic} implies that
	$\Psi_*$ is the completion of an isomorphism $F_\mu\colon H_A\to H_B$ 
	followed by a $\mu$--scalar multiplication on $\widehat{H_B}=H_B\otimes_\Integral\widehat{\Integral}$;
	see also Proposition \ref{MC_property} (2).
	Then $\psi_{A*}$ is by definition the composite homomorphism
	$$\xymatrix{
	H_A	\ar[r]^-{F_\mu\otimes\mu} & H_B\otimes_\Integral\mu\Integral \ar[r]^-{1\otimes\,\mu^{-1}} 
	& H_B\otimes_\Integral\Integral \ar[r]^-{=} & H_B \ar[r]^{\phi_{B*}} & \Integral,
	}$$
	where $\mu\Integral$ is precisely $\mathrm{MC}(\Phi_*)$.
	We obtain the relation
	$\Psi_*(\mathrm{Ker}(\psi_{A*}))=\mu F_\mu(\mathrm{Ker}(\phi_{A*}))=\mu\mathrm{Ker}(\phi_{B*})$
	in $\widehat{H_B}$.
	Since $\mathrm{Ker}(\phi_{B*})$ is a $\Integral$--submodule of $H_B$,
	the closure of $\mathrm{Ker}(\phi_{B*})$ in $\widehat{H_B}$ 
	is invariant under scalar multiplication by a unit.
	Therefore, we obtain
	$\Psi_*(\mathrm{clos}(\mathrm{Ker}(\psi_{A*})))=\mathrm{clos}(\mu\mathrm{Ker}(\phi_{B*}))
	=\mu\,\mathrm{clos}(\mathrm{Ker}(\phi_{B*}))=\mathrm{clos}(\mathrm{Ker}(\phi_{B*}))$
	in $\widehat{H_B}$, as desired.	
\end{proof}

\begin{corollary}\label{TN_correspondence}
	Under the assumptions of Theorem \ref{profinite_isomorphism_hyperbolic},
	$\Psi^*_{1/\mu}$ %$\Psi^*_{1/\mu}\colon H^1(M_B;\Real)\to H^1(M_A;\Real)$ 
	witnesses a linear isomorphism between the Thurston-norm unit ball for $M_B$ and that for $M_A$.
	In other words, $\Psi^*_{1/\mu}$ is Thurston-norm preserving.
\end{corollary}

\begin{proof}
	We first show that $\Psi^*_{1/\mu}$ is 
	Thurston-norm preserving on fibered cones, 
	using Corollary \ref{fiber_surface_correspondence}. 
	Note that $\Psi^*_{1/\mu}$ maps fibered cones for $M_B$ isomorphically onto fibered cones for $M_A$,
	(Corollary \ref{fibered_cone_correspondence}).
	It suffices to check the result for primitive integral classes, 
	since the Thurston norm is linear on its cones.
	To this end,
	we use the fact that
	the Thurston norm of any primitive integral fibered class is equal to $-1$ times
	the Euler characteristic of the connected surface fiber,
	with respect to the dual fibration.
	Note that hyperbolic $3$--manifolds of finite volume 
	can only have fiber surfaces of negative Euler characteristic.
	Corollary \ref{fiber_surface_correspondence} shows that $\Psi^*_{1/\mu}$ 
	preserves the isomorphism type of profinite completion of the fiber subgroup, 
	which determines the Euler characteristic, 
	because surface groups are all cohomologically good.
	Therefore, $\Psi^*_{1/\mu}$ is Thurston-norm preserving on fibered cones.
	
	To prove that $\Psi^*_{1/\mu}$ preserve the Thurston norm on all cones,
	we pass to a $\Psi$--corresponding pair of finite covers $M'_A\to M_A$ and $M'_B\to M_B$.
	Applying Theorem \ref{quasi-fibered} to $M_A$, we can require that 
	the image of every $\phi\in H^1(M_A;\Real)$	lies in the boundary of some fibered cone for $M'_A$,
	with respect to the covering-induced homomorphism $H^1(M_A;\Real)\to H^1(M'_A;\Real)$.
	(Hence, the same property holds for $M'_B$ by Corollary \ref{fibered_cone_correspondence}.)
	Applying Theorem \ref{profinite_isomorphism_hyperbolic} to the setting $(M'_A,M'_B,\Psi')$,
	we obtain some unit $\mu'\in\widehat{\Integral}^\times$,
	such that $\mathrm{MC}(\Psi'_*)$ equals $\mu'\Integral$.
	%We can organize the constructed objects into a commutative diagram (see Definition \ref{MC_def}):
	%$$\xymatrix{
		%H'_A \ar[r]^-{(\Psi'_*)^{\mathrm{MC}}} \ar[d]_{\mathrm{cov}_*} 
		%& H'_B\otimes_\Integral \mu'\Integral \ar[r]^-{\mathrm{incl}} 
		%& H'_B\otimes_\Integral\widehat{\Integral} \ar[r]^-{=} \ar[d]_{\mathrm{cov}_*\otimes1} & \widehat{H'_B} \ar[d]^{\mathrm{cov}_*}\\
		%H_A \ar[r]^-{\Psi_*^{\mathrm{MC}}} & H_B\otimes_\Integral\mu\Integral \ar[r]^-{\mathrm{incl}} 
		%& H_B\otimes_\Integral\widehat{\Integral} \ar[r]^-{=} & \widehat{H_B}
	%}$$
	%where $(\Psi'_*)^{\mathrm{MC}}$ and $\Psi_*^{\mathrm{MC}}$ are both isomorphisms.
	%It follows that $\mu'$ is a $\Integral$--multiple of $\mu$, so $\mu'$ must be $\pm\mu$ for invertibility.
	%%%
	By Lemma \ref{MC_cover}, $\mu\Integral$ 
	is commensurable with a $\Integral$--submodule of $\mu'\Integral$ in $\widehat{\Integral}$,
	so we must have $\mu'=\pm\mu$, since $\mu$ and $\mu'$ are both units.
	Possibly switching to the other generator $-\mu'$ of $\mathrm{MC}(\Psi'_*)$, 
	we assume $\mu'=\mu$ without loss of generality.
	Then we obtain the commutative diagram 
	$$\xymatrix{
	 H^1(M_B;\Real) \ar[r]^-{\Psi^*_{1/\mu}} \ar[d]_{\mathrm{cov}^*} & H^1(M_A;\Real) \ar[d]^{\mathrm{cov}^*} \\
	 H^1(M'_B;\Real) \ar[r]^-{(\Psi')^*_{1/\mu}} & H^1(M'_A;\Real)
	}$$
	
	For any cohomology class $\phi$ in $H^1(M_B;\Real)$, 
	we approach the image of $\phi$ using fibered classes in $H^1(M'_B;\Real)$,
	and apply the Thurston-norm preserving property for fibered classes as we have proved.
	Using the Thurston-norm formula for pullback to finite covers (Theorem \ref{TN_cover}),
	we see that the Thurston norm of $\phi$ with respect to $M_B$ 
	is equal to the Thurston norm of $\Psi^*_{1/\mu}(\phi)$ with respect to $M_A$,
	as desired.
\end{proof}

%
%\begin{corollary}\label{TN_correspondence}
%Hence, the scalar multiplication on $\widehat{\Integral}$ by $\mu^{-1}$ 
%defines a $\Integral$--valued homomorphism $1/\mu\in\mathrm{Hom}_\Integral(\mathrm{MC}(\Psi_*),\Real)$.
%
%Moreover,
%$\Psi^*_{1/\mu}\colon H^1(M_B;\Real)\to H^1(M_A;\Real)$ 
%witnesses a linear isomorphism between the Thurston-norm unit ball for $M_B$ and that for $M_A$.
%In other words, $\Psi^*_{1/\mu}$ is Thurston-norm preserving.
%\end{corollary}

The rest of this section is devoted to the proof of Theorem \ref{profinite_isomorphism_hyperbolic}.
Our goal is to show that the matrix coefficient module for $\Psi_*$
with respect to $H_1(M_A;\Integral)_{\mathtt{free}}$ and $H_1(M_B;\Integral)_{\mathtt{free}}$
has rank $1$.
To this end, we give a criterion for checking this property,
in terms of projective dual polytopes (Lemma \ref{new_point_criterion});
it is accompanied with another simple reduction,
which allows passage to finite covers (Lemma \ref{profinite_isomorphism_hyperbolic_reduction}). 
We prove the closed case by applying these criteria
together with techniques from \cite{Liu_vhsr} (Lemma \ref{profinite_isomorphism_hyperbolic_closed}).
We derive the cusped case from the closed case
by techniques of \cite{WZ_decomposition} (Lemma \ref{profinite_isomorphism_hyperbolic_robust}).
We summarize our proof of Theorem \ref{profinite_isomorphism_hyperbolic} in the end of this section.

\subsection{Criteria for rank-one}\label{Subsec-criteria_rank_one}
For any orientable compact $3$--manifold $M$,
we denote by $\mathbf{P}(H_1(M;\Real))$ the projectivization of $H_1(M;\Real)$.
The points of $\mathbf{P}(H_1(M;\Real))$
are considered to be the real $1$--dimensional linear subspaces of $H_1(M;\Real)$.
For any codimension--$0$ cone $\mathcal{C}$ in $H^1(M;\Real)$
with respect to the Thurston norm of $M$,
we introduce a subset of $\mathbf{P}(H_1(M;\Real))$ as follows:
\begin{equation}\label{cone_projective_dual}
\mathcal{D}(M,\mathcal{C})=\left\{ l\in \mathbf{P}(H_1(M;\Real))\colon 
l\cap\mathrm{Ker}(\phi)=\{0\}\mbox{ for all }\phi\in\mathrm{int}(\mathcal{C})\right\}.
\end{equation}
In general, $\mathcal{D}(M,\mathcal{C})$ is 
projectively isomorphic to a polytope, 
(see Remark \ref{polyhedron_remark}).
Its codimension in $\mathbf{P}(H_1(M;\Real))$ equals the kernel dimension of the Thurston norm.
We refer to $\mathcal{D}(M,\mathcal{C})$ as the \emph{projective dual polytope} with respect to $\mathcal{C}$. 

When we already have a corresponding pair of fibered cones,
we will be able to prove Theorem \ref{profinite_isomorphism_hyperbolic} 
if we know an extra corresponding pair of interior points
in the corresponding projective dual polytopes.
This is what the next criterion says.

\begin{lemma}\label{new_point_criterion}
	Adopt the assumptions of Theorem \ref{profinite_isomorphism_hyperbolic}.
	%Let $(M_A,M_B,\Psi)$ be a profinite morphism setting of a $3$--manifold pair.
	%Suppose that $\Psi\colon\widehat{\pi_A}\to\widehat{\pi_B}$ is an isomorphism.
	%Suppose that $M_A$, and hence $M_B$, 
	%has empty or tori boundary and 
	%admits a complete hyperbolic metric in the interior.
	Suppose 
	$\varepsilon\in \mathrm{Hom}_\Integral(\mathrm{MC}(\Psi_*);\Real)$ with nondegenerate 
	$\Psi^*_\varepsilon\colon H_1(M_A;\Real)\to H_1(M_B;\Real)$.
	Suppose that there exists a $\Psi^*_\varepsilon$--corresponding pair of fibered cones 
	$\mathcal{C}_A$ and $\mathcal{C}_B$.
	Then, $\mathrm{MC}(\Psi_*)$ is a free $\Integral$--module of rank $1$
	if the following additional condition is satisfied:
	
	%\begin{itemize}
	%\item 
	There exist a pair of interior points 
	$q_A\in \mathrm{int}(\mathcal{D}(M_A,\mathcal{C}_A))$ and
	$q_B\in \mathrm{int}(\mathcal{D}(M_B,\mathcal{C}_B))$,
	and some open neighborhood 
	$\mathcal{W}$ of $\varepsilon$ in $\mathrm{Hom}_\Integral(\mathrm{MC}(\Psi_*);\Real)$,
	such that for every $\tilde{\varepsilon}\in\mathcal{W}$,
	$\Psi_*^{\tilde{\varepsilon}}$
	is nondegenerate,
	and $\mathbf{P}(\Psi_*^{\tilde{\varepsilon}})$ maps $q_A$ to $q_B$.
	%\end{itemize}
\end{lemma}

%\begin{remark}\label{new_point_criterion_remark}
	%Instead of being closed, one may assume that 
	%$M_A$ (and hence $M_B$)	is compact,
	%and admits a complete hyperbolic metric of finite volume in the interior.
	%Then the conclusion of Lemma \ref{new_point_criterion} still holds.
	%This is because hyperbolicity is only needed in the proof
	%for ensuring that $\mathcal{D}(M_A,\mathcal{C}_A)$ 
	%has codimension--$0$ in $\mathbf{P}(H_1(M_A;\Real))$.	
%\end{remark}

\begin{proof}
	Denote by $\mathrm{PGL}(H_1(M_A;\Real))$ the projective linear transformation group acting on 
	the real projective space $\mathbf{P}(H_1(M_A;\Real))$.
	There is a comparison map 
	$$f\colon \mathcal{W}\to \mathrm{GL}(H_1(M_A;\Real)),$$
	defined as $f(\tilde{\varepsilon})=(\Psi_*^{\varepsilon})^{-1}\circ\Psi_*^{\tilde{\varepsilon}}$
	for all $\tilde{\varepsilon}\in\mathcal{W}$.
	%Note that $f$ is just a $\Real$--linear map 
	%$$\mathrm{Hom}_\Integral(\mathrm{MC}(\Psi_*),\Real)\to \mathrm{End}_\Real(H_1(M_A;\Real))$$
	%to $\mathcal{W}$, which is defined with the same expression.
	We consider the projective comparison map
	$$\mathbf{P}(f)\colon \mathcal{W}\to \mathrm{PGL}(H_1(M_A;\Real))$$
	defined as the projectivization of $f$,
	namely,
	%$$\mathbf{P}(f)(\tilde{\varepsilon})
	%=\mathbf{P}\left(\Psi_*^{\varepsilon}\right)^{-1}\circ\mathbf{P}\left(\Psi_*^{\tilde{\varepsilon}}\right),$$
	$\mathbf{P}(f)(\tilde{\varepsilon})
	=\mathbf{P}(\Psi_*^{\varepsilon})^{-1}\circ\mathbf{P}(\Psi_*^{\tilde{\varepsilon}})$
	for all $\tilde{\varepsilon}\in\mathcal{W}$. 
	In particular, $\mathbf{P}(f)(\varepsilon)$ 
	is the identity transformation on $\mathbf{P}(H_1(M_A;\Real))$.
	
	We first show that 
	the rank of $\mathrm{MC}(\Psi_*)$ equals $1$ 
	if the comparison map $\mathbf{P}(f)$ is constant on 
	some open neighborhood of $\varepsilon$ in $\mathcal{W}$.
	To this end, 
	we recall that $\mathrm{MC}(\Psi_*)$ is a free $\Integral$--module of finite rank (Proposition \ref{MC_property}).
	Fix a basis $X_1,\cdots,X_r$ of $\mathrm{MC}(\Psi_*)$ where $r=\mathrm{rank}_\Integral(\mathrm{MC}(\Psi_*))$.
	Then there is a unique decomposition
	$$\Psi_*^{\mathrm{MC}}=\Phi_1\otimes X_1+\cdots+\Phi_r\otimes X_r$$
	in $\mathrm{Hom}_\Integral(H_A,H_B)\otimes_{\Integral}\mathrm{MC}(\Psi_*)$,
	where $\Phi_1,\cdots,\Phi_r$ are homomorphisms in $\mathrm{Hom}_\Integral(H_A,H_B)$,
	(see Definition \ref{MC_def}).
	The definition of $\mathrm{MC}(\Psi_*)$ implies that $\Phi_1,\cdots,\Phi_r$ are $\Integral$--linearly independent.
	(Otherwise, by looking at some matrix entry of $\Psi_*$,
	one would find a nontrivial $\Integral$--linear combination of the values of $X_i$ in $\widehat{\Integral}$
	which equals zero, violating the rank--$r$ assumption.)
	%The tangent space $T_{\tilde{\epsilon}}\mathcal{W}$ of $\mathcal{W}$ at any $\tilde{\varepsilon}$
	%can be identified canonically with $\mathrm{Hom}_\Integral(\mathrm{MC},\Real)$,
	%and the tangent space $T_g\mathrm{GL}(H_1(M_A;\Real))$ 
	%can be canonically with $\mathrm{End}_\Real(H_1(M_A;\Real))$
	%(namely, the space of $\Real$--linear endomorphisms).
	%So the tangent map of $f$ at any point $\tilde{\varepsilon}$
	%$$D(f)\colon \mathrm{Hom}_\Integral(\mathrm{MC}(\Psi_*),\Real)\to \mathrm{End}_\Real(H_1(M_A;\Real))$$
	Since $\Psi^\varepsilon_*$ is nondegenerate, it follows that the endomorphisms 
	$(\Psi^\varepsilon_*)^{-1}\circ\Phi_1,\cdots,(\Psi^\varepsilon_*)^{-1}\circ\Phi_r$
	in $\mathrm{End}_\Real(H_1(M_A;\Real))$ are $\Real$--linearly independent.
	We observe
	$$f(\tilde{\varepsilon})=
	\tilde{\varepsilon}(X_1)\cdot F_1+\cdots+	\tilde{\varepsilon}(X_r)\cdot F_r,$$
	where $F_i=(\Psi^\varepsilon_*)^{-1}\circ\Phi_i$ for $i=1,\cdots,r$.
	This means that $f$ is actually the restriction to $\mathcal{W}$ of a linear embedding
	$\mathrm{Hom}_\Integral(\mathrm{MC}(\Psi_*),\Real)\to \mathrm{End}_\Real(H_1(M_A;\Real))$.
	Passing to the projectivization,
	we see that 
	the image of $\mathbf{P}(f)$ has dimension $r-1$ everywhere.
	In particular, $r=1$ holds if (and only if) $\mathbf{P}(f)$ is constant near $\varepsilon$.
	
	It remains to show that the additional condition offered in Lemma \ref{new_point_criterion}
	implies that $\mathbf{P}(f)$ is constant near $\varepsilon$.
	Observe that 
	for any $\tilde{\varepsilon}$ sufficiently close to $\varepsilon$
	in $\mathrm{Hom}_\Integral(\mathrm{MC}(\Psi_*),\Real)$,
	$\Psi^*_{\tilde{\varepsilon}}$ 
	witnesses the same bijective correspondence between Thurston-norm cones
	as $\Psi^*_\varepsilon$ does,
	because 
	$\Psi^*_{\tilde{\varepsilon}}$ depends continuously in $\tilde{\varepsilon}$
	while the number of Thurston-norm cones is finite.
	Therefore, 
	possibly replacing $\mathcal{W}$ with a smaller open neighborhood of $\varepsilon$,
	we assume without loss of generality that 
	the above same-correspondence property holds for $\mathcal{W}$.
	With this simplifying assumption,
	the offered condition in Lemma \ref{new_point_criterion} implies that
	$\mathbf{P}(f)(\tilde{\varepsilon})$ fixes every vertex of 
	the projective polytope $\mathcal{D}(M_A,\mathcal{C}_A)$ 
	and also fixes the interior point $q_A$,
	for every $\tilde{\varepsilon}\in\mathcal{W}$.
	Remember that $\mathcal{D}(M_A,\mathcal{C}_A)$ 
	has codimension--$0$ in $\mathbf{P}(H_1(M_A;\Real))$,
	and is projective-linearly isomorphic to an affine polytope.
	
	Therefore, 
	the rest of the proof is reduced to the following claim in projective geometry:
	Given a codimension--$0$ polytope $D$ in $\mathbf{A}^n$ and a point $q\in\mathrm{int}(D)$,
	if a projective linear isomorphism $\tau\colon \mathbf{P}^n\to \mathbf{P}^n$
	fixes all the vertices of $D$ and $q$, then $\tau$ is the identity.
	Here, $\mathbf{P}^n$ denotes the real projective $n$--space 
	with the standard homogeneous coordinates $(x_0:x_1:\cdots:x_n)$,
	(namely, $x_0,\cdots,x_n\in\Real$ are not all zero and
	$(x_0:x_1:\cdots:x_n)=(\lambda x_0:\lambda x_1:\cdots:\lambda x_n)$
	holds	for all nonzero $\lambda\in\Real$,)
	and	$\mathbf{A}^n$ denotes the real affine $n$--subspace complementary to 
	the projective hyperplane defined by $x_0=0$.
	
	The projective-geometric claim can be proved
	by induction on the number of vertices in $D$,
	as follows.
	Note that $D$ has at least $n+1$ vertices, as it has codimension $0$ in $\mathbf{A}^n$.
	Note also that $\tau$ preserves $D$, as it fixes all the vertices of $D$ and an interior point $q$.
	
	When $D$ has exactly $n+1$ vertices,
	it is an $n$--simplex. 
	Its vertices together with the interior point $q$ are in general position
	(namely, any subset of $n+1$ points span a $n$--simplex in $\mathbf{A}^n$).
	This is the same as saying that 
	these $n+2$ points form a frame of $\mathbf{P}^n$, 
	and the projective linear isomorphism $\tau$ is uniquely determined	by their images.
	In particular, $\tau$ is the identity if it fixes all the $n+2$ points.
	
	By induction, 
	suppose that the claim holds when $D$ has exactly $k$ vertices,
	for some $k$ at least $n+1$.
	When $D$ has exactly $k+1$ vertices $v_1,\cdots,v_{k+1}$,
	we may assume that $v_1,\cdots,v_{n+1}$ span a $n$--simplex in $\mathbf{A}^n$.
	Take $D'$ to be the convex hull of $v_1,\cdots,v_k$ in $\mathbf{A}^n$,
	so $D'$ has codimension zero in $\mathbf{A}^n$.
	%The codimension--$1$ closed faces of $D$ that contains $v_{k+1}$ are 
	%the codimension--$1$ closed faces of the convex hull of $F\cup\{v_{k+1}\}$ except $F$.
	Take $L$ to be the projective line in $\mathbf{P}^n$ passing through $q$ and $v_{k+1}$,
	so $L$ intersects $\partial D$ at a unique point $w$ other than $v_{k+1}$.
	Since $q$ lies in the interior of $D$, the point $w$ must lie on 
	some (not necessarily unique) codimension--$1$ closed face $F$ of $D$,
	whose vertices all come from $\{v_1,\cdots,v_k\}$.
	This means that $F$ is a codimension--$1$ face of $D'$,
	and moreover, $q$ and $v_{k+1}$ must lie 
	on the same side of the hyperplane in $\mathbf{A}^n$ that contains $F$.
	Therefore, there exists some point $q'$ on $L$ (close to $w$),
	such that $q'$ lies in the interior of $D'$.
	Since $\tau$ fixes $q$ and all vertices of $D$,	
	it preserves the face $F$, and preserves the projective $L$,
	so it fixes their intersection point $w$.
	It follows that $\tau$ fixes the projective line $L$, 
	since $\tau$ fixes three distinct points $v_{k+1}$, $w$, and $q$ on $L$.
	Therefore, $\tau$ fixes $q'$, which lies in the interior of $D'$, 
	but $\tau$ also fixes all the $k$ vertices of $D'$.
	Then the induction hypothesis implies that $\tau$ is trivial, as desired.
	
	This completes the proof of the projective-geometric claim.
	Applying to the projective comparison map $\mathbf{P}(f)$,
	we conclude that the offered condition in Lemma \ref{new_point_criterion}
	implies that $\mathbf{P}(f)$ is constantly the identity near $\varepsilon\in\mathcal{W}$,
	and therefore, $\mathrm{MC}(\Psi_*)$ is a free $\Integral$--module of rank $1$.
\end{proof}

\begin{lemma}\label{profinite_isomorphism_hyperbolic_reduction}
	Let $(M_A,M_B,\Psi)$ be a profinite morphism setting of a $3$--manifold pair (Convention \ref{profinite_morphism_setting}).
	Suppose that $\Psi$ is an isomorphism.
	%Under the assumptions of Theorem \ref{profinite_isomorphism_hyperbolic},
	If there exists some $\Psi$--corresponding pair of finite covers
	$M'_A\to M_A$ and $M'_B\to M_B$,
	and if $\mathrm{MC}(\Psi'_*)$ is a free $\Integral$--module of rank $1$,
	then 
	there exists some unit $\mu\in\widehat{\Integral}^\times$, such that 
	$\mathrm{MC}(\Psi_*)$ is the $\Integral$--submodule $\mu\Integral$ of $\widehat{\Integral}$.
	%the conclusion of Theorem \ref{profinite_isomorphism_hyperbolic} holds true.
	(See Definition \ref{corresponding_cover_def}.)
\end{lemma}

\begin{proof}
	%Denote by $L_A$ is the image of the covering induced homomorphism $H'_A\to H_A$
	%and $L_B$ similarly.
	%Fix a basis of $H'_A$ by lifting a fixed basis of $L_A$ and extending,
	%and similarly, fix a basis of $H'_B$.
	%
	%As explained in the proof of Lemma \ref{U_construction},
	%the matrix of $\Psi'_*\colon H'_A\to H'_B$ with respect to the fixed bases 
	%has entries in $\widehat{\Integral}$, and it can be arranged
	%into an upper-triangular block form, 
	%such that the matrix of $\Psi_*|_{L_A}\colon L_A\to L_B$ 
	%with respect to the fixed bases appears as a diagonal block.
	%Then it follows from Proposition \ref{MC_property} (1) and (3) that 
	%$\mathrm{MC}(\Psi_*)$ is commensurable with	a $\Integral$--submodule of $\mathrm{MC}(\Psi'_*)$,
	%(namely, $\mathrm{MC}(\Psi_*;L_A,L_B)$).
	%In particular, the rank of $\mathrm{MC}(\Psi_*)$ is at most the rank of $\mathrm{MC}(\Psi'_*)$.
	By Lemma \ref{MC_cover}, $\mathrm{MC}(\Psi_*)$ is commensurable in $\widehat{\Integral}$
	with a $\Integral$--submodule of $\mathrm{MC}(\Psi'_*)$.
	If $\mathrm{MC}(\Psi'_*)$ has rank $1$,
	we infer that $\mathrm{MC}(\Psi_*)$ must also have rank $1$.
	Let $\mu\in\widehat{\Integral}$ be a generator of $\mathrm{MC}(\Psi'_*)$.
	By Proposition \ref{MC_property} (2),
	we obtain a homomorphism $F_\mu\colon H_A\to H_B$,
	such that $\Psi_*$ equals the $\mu$--scalar multiple of the completion of $F_\mu$.
	Moreover, 
	$\mu$ must be unit, (and $F'_\mu$ must be an isomorphism,)
	because $\Psi'\colon\widehat{\pi'_A}\to \widehat{\pi'_B}$ is an isomorphism.
	Therefore, $\mathrm{MC}(\Psi_*)$ is a $\Integral$--submodule $\mu\Integral$ of $\widehat{\Integral}$
	generated by the unit element $\mu\in\widehat{\Integral}^\times$, as asserted.
	Moreover, one may check by definition that $\Psi^*_{1/\mu}\colon H^1(M_B;\Real)\to H^1(M_A;\Real)$
	is the $\Real$--linear extension of 
	$F_\mu^*\colon H^1(M_B;\Integral)\to H^1(M_A;\Integral)$,
	which is the $\Integral$--dual of $F_\mu\colon H_A\to H_B$,
	(see Definition \ref{specialization_def}).
\end{proof}

\subsection{The closed case}\label{Subsec-closed_hyp}

The next two lemmas allow us to construct some 
trackable interior point in a projective dual polytope.
In fact, essentially the same construction has been sketched in 
\cite[Problem 1.5 and Outline of the Solution]{Liu_vhsr},
as an illustration of the techniques thereof.
What we actually need for applying Lemma \ref{new_point_criterion}
is Lemma \ref{new_vertex_D_cover}.
Lemma \ref{facial_clusters} 
is separated out from its proof for an expository purpose.

When $M$ is closed hyperbolic with a primitive fibered class $\phi$,
we can identify $M$ with the mapping-torus $M_f$ of a pseudo-Anosov automorphism $f\colon S\to S$,
and $\phi$ with the distinguished cohomology class $\phi_f$, (see Remark \ref{pA_remark}).
Then, 
Fried's dual characterization of the fibered cone $\mathcal{C}_f=\mathcal{C}_{\mathtt{Th}}(M_f,\phi_f)$
is equivalent to saying that
the $1$--subspaces of $H_1(M_f;\Real)$ spanned by the periodic trajectories of the suspension flow
form a dense subset of $\mathcal{D}(M_f,\mathcal{C}_f)$, (see Section \ref{Subsec-polytopes_and_cones}).
We also note that the fundamental group of the mapping torus $\pi_1(M_f)$ is word-hyperbolic.
%Moreover,
%since the Thurston norm of closed hyperbolic $3$--manifolds are nondegenerate, 
%$\mathcal{D}(M_f,\mathcal{C}_f)$ has codimension $0$ in $\mathbf{P}(H_1(M_f;\Real))$,
%so its dimension equals $b_1(M_f)-1$.

\begin{lemma}\label{facial_clusters}
	Let $f\colon S\to S$ be a pseudo-Anosov automorphism of an orientable connected closed surface.
	Then, 
	there exist 
	a finite collection $\mathcal{H}$ 
	of infinite-index quasi-convex subgroups of $\pi_1(M_f)$,
	such that the following property is satisfied:
	
	If $\gamma$ is a periodic trajectory of the suspension flow,
	such that the $1$--subspace of $H_1(M_f;\Real)$ spanned by its homology class $[\gamma]$
	is a boundary point of the polytope $\mathcal{D}(M_f,\mathcal{C}_f)$,
	then some finite cyclic cover of $\gamma$, as a free-homotopy loop,
	represents a conjugacy class in $\pi_1(M_f)$
	which has nonempty intersection with some $H\in\mathcal{H}$.
\end{lemma}

%If one realizes the subgroups $H_i$ as the $\pi_1$--image of 
%some compact $\pi_1$--injectively immersed aspherical submanifolds $Q_i\to M_f$,
%the property in Lemma \ref{facial_clusters} means that any periodic trajectory
%carried by $\partial\mathcal{D}(M_f,\mathcal{C}_f)$ has a finite-cyclic cover
%which is freely homotopic to a loop in some $Q_i$.

\begin{proof}
	The asserted subgroups $H_i$ can be constructed 
	using the techniques in \cite[Sections 5 and 6]{Liu_vhsr}.
	Those techniques are developed based on Fried's original ideas in \cite{Fried-sections,Fried-flowEqv},
	see the notes in \cite[Section 5.4]{Liu_vhsr}.
	To simply put, 
	the subgroups $H_i$ all correspond to some subgraphs of the transition graph
	with respect to any fixed Markov partition, 
	and the dynamical cycles in those subgraphs 
	all represent periodic trajectories in $M_f$
	whose homology classes are projectively carried by $\partial\mathcal{D}(M_f,\mathcal{C}_f)$.	
	Below we present the construction in more detail,
	following \cite[Sections 5 and 6]{Liu_vhsr}.
		
	Take a Markov partition $\mathcal{R}=\{R_1,\cdots,R_k\}$ of $S$ with respect to 
	the pseudo-Anosov automorphism $f$.
	To be precise, each $R_i$ is a \emph{birectangle}, namely,
	it is a subset of $S$, parametrized as $[0,1]\times[0,1]$,
	such that any vertical segment $\{x\}\times[0,1]$ is contained in the union of
	finitely many leaves or singular points of the stable invariant foliation $\mathscr{F}^{\mathtt{s}}$,
	and such that any horizontal segment $[0,1]\times\{y\}$ 
	is compatible with the unstable invariant foliation $\mathscr{F}^{\mathtt{u}}$ similarly.
	As one may check, 
	singular points only occur on boundary of the birectangles, 
	and the stable or unstable leaves
	containing any interior point of a birectangle must be unique.
	Moreover,
	the birectangles $R_i$ cover $S$ altogether without mutual intersection in their interiors,
	and every birectangle $R_i$ is the vertical juxtaposition of finitely many horizontal blocks,
	which are mapped under $f$ onto vertical blocks in mutually distinct birectangles in $\mathcal{R}$.
	(See \cite[Section 5.1]{Liu_vhsr}.)
	
	The \emph{transition graph} $T_{f,\mathcal{R}}$
	is a directed graph, modeled as a CW $1$--complex with oriented $1$--cells,
	such that the vertice $v_i$ of $T_{f,\mathcal{R}}$ 
	corresponds bijectively to the birectangle $R_i$,
	and the edges $e_{ij}$ correspond bijectively 
	to the horizontal blocks	$R_{ij}=R_i\cap f^{-1}(R_j)$,
	such that $\mathrm{int}(R_i)\cap f^{-1}(\mathrm{int}(R_j))\neq\emptyset$.
	The transition graph $T_{f,\mathcal{R}}$ is homotopy equivalent
	to an \emph{abstract flow-box complex} $X_{f,\mathcal{R}}$,
	which is obtained from the disjoint union of the birectangles $R_i$
	by attaching boxes $R_{ij}\times[0,1]$, 
	such that $R_{ij}\times\{0\}$ is attached to $R_i$ by the inclusion $R_{ij}\to R_i$,
	and $R_{ij}\times\{1\}$ is attached to $R_j$ by $f\colon R_{ij}\to R_j$.
	The homotopy equivalence is realized by the natural collapsing map
	$X_{f,\mathcal{R}}\to T_{f,\mathcal{R}}$.
	Meanwhile, there is a natural map $q_{f,\mathcal{R}}\colon X_{f,\mathcal{R}}\to M_f$,
	which puts the building parts of $X_{f,\mathcal{R}}$
	into $M_f$ in the obvious way. (See \cite[Sections 5.2 and 5.3]{Liu_vhsr}.)
	
	For every primitive periodic trajectory $\gamma$ in $M_f$,
	the preimage $q_{f,\mathcal{R}}^{-1}(\gamma)$ in $X_{f,\mathcal{R}}$
	is a union of mutually disjoint	primitive abstract periodic trajectories,
	(meaning that its components are all cyclical concatenations of 
	mutually distinct, directly consecutive segments of the form $\{p\}\times R_{ij}$),
	\cite[Lemma 5.2]{Liu_vhsr}. Moreover,
	the collapsing map witnesses a bijective correspondence between
	the primitive abstract periodic trajectories in $X_{f,\mathcal{R}}$
	and	the primitive dynamical cycles in $T_{f,\mathcal{R}}$,
	\cite[Lemma 5.7]{Liu_vhsr}.
	Here, a \emph{dynamical cycle} in a directed graph 
	refers to a directly immersed finite cyclic directed graph,
	and it is \emph{primitive} if it does not factor through any shorter dynamical cycle.
	In particular, we see that any primitive periodic trajectory in $M_f$
	admits a finite-cyclic cover which can be represented as
	a dynamical cycle in $T_{f,\mathcal{R}}$ (among finitely many choices).
	
	For every closed face $E$ of $\mathcal{D}(M_f,\mathcal{C}_f)$,
	there exists a directed subgraph of $T_{f,\mathcal{R}}$,
	called the \emph{support graph} of $E$ and denoted as $T_{f,\mathcal{R}}[E]$.
	The support graph is characterized by the following property:
	A dynamical cycle of $T_{f,\mathcal{R}}$	
	is contained in $T_{f,\mathcal{R}}[E]$ if and only if
	it represents a periodic trajectory in $M_f$ 
	whose projective homology class in $\mathcal{D}(M_f,\mathcal{C}_f)$	lies in $E$.
	
	In fact,
	there exists a canonical polytope $\mathcal{P}_{f,\mathcal{R}}$,
	whose points are the probability measures on the edge set of $T_{f,\mathcal{R}}$
	balanced at the vertices, in the sense that the total incoming measure
	at any vertex equals the total outgoing measure;
	there exists a canonical, surjective, linear map
	$h_{f,\mathcal{R}}\colon \mathcal{P}_{f,\mathcal{R}}\to\mathcal{D}(M_f,\mathcal{C}_f)$,
	which takes any measure $\mu\in\mathcal{P}_{f,\mathcal{R}}$
	to the projective homology class of the $\Real$--cellular cycle 
	$\sum_{ij}\mu(\{e_{ij}\})\,e_{ij}$,
	(see \cite[Definition 5.9 and (5.4)]{Liu_vhsr}).
	The support of any measure $\mu\in\mathcal{P}_{f,\mathcal{R}}$
	gives rise to a nonwandering subgraph of $T_{f,\mathcal{R}}$,
	and this construction sets up a canonical bijective correspondence between
	the faces of $\mathcal{P}_{f,\mathcal{R}}$ 
	and the nonwandering subgraphs of $T_{f,\mathcal{R}}$,
	\cite[Lemma 5.10]{Liu_vhsr}.
	The support graph $T_{f,\mathcal{R}}[E]$ 
	is precisely the nonwandering subgraph of $T_{f,\mathcal{R}}$
	that corresponds to the closed face $h_{f,\mathcal{R}}^{-1}(E)$,
	(see \cite[Definition 6.6]{Liu_vhsr}).
	The characterizing property of $T_{f,\mathcal{R}}$
	follows from the above description,
	since any dynamical cycle of $T_{f,\mathcal{R}}$ 
	determines a measure in $\mathcal{P}_{f,\mathcal{R}}$,
	which is the normalization of the edge-counting measure.
	
	For every codimension--$1$ closed face $E$ of $\mathcal{D}(M_f,\mathcal{C}_f)$
	and for every component $V$ of the support graph $T_{f,\mathcal{R}}[E]$,
	we obtain a corresponding subcomplex $X_{f,\mathcal{R}}(V)$ of $X_{f,\mathcal{R}}$,
	namely, the union of all $R_i$ if $v_i$ is a vertex of $V$,
	and all $R_{ij}\times[0,1]$ if $e_{ij}$ is an edge of $V$.
	We obtain a subgroup $H_V$ of $\pi_1(M_f)$ 
	as the $\pi_1$--image of the restricted map $q_{f,\mathcal{R}}\colon X_{f,\mathcal{R}}(V)\to M_f$,
	(choosing a representative among all conjugates).
	The construction gives rise to 
	only finitely many subgroups $H_V$,
	and we collect them as $\mathcal{H}$.
	The above characterizing property of support subgraphs
	ensures that $\mathcal{H}$ satisfies the asserted property in Lemma \ref{facial_clusters}.
	Note that $\partial\mathcal{D}(M_f,\mathcal{C}_f)$ is empty if
	$\mathcal{D}(M_f,\mathcal{C}_f)$ has dimension zero.
	
	It remains to argue that any $H\in\mathcal{H}$ is quasi-convex of infinite-index.
	Suppose $V$ is a component of the support graph $T_{f,\mathcal{R}}[E]$ 
	for some codimension--$1$ closed face $E$ of $\mathcal{D}(M_f,\mathcal{C}_f)$,
	and we argue with $H_V$.
	Denote by $L_E$ the codimension--$1$ linear subspace of $H_1(M_f;\Real)$,
	such that 
	the projectivization $\mathbf{P}(L_E)$ 
	is the support hyperplane of $E$ in $\mathbf{P}(H_1(M_f;\Real))$.
	Since the dynamical cycles of the nonwandering subgraph $V$ 
	generate the homology	$H_1(X_{f,\mathcal{R}}(V);\Real)\cong H_1(V;\Real)$,
	(see \cite[Remark 5.12 (1)]{Liu_vhsr}),
	the image of the homomorphism $H_1(X_{f,\mathcal{R}}(V);\Real)\to H_1(M_f;\Real)$
	induced by $q_{f,\mathcal{R}}$ is contained in $L_E$.
	In particular, the image of $H_1(X_{f,\mathcal{R}}(V);\Real)$
	has codimension at least $1$ in $H_1(M_f;\Real)$,
	so the subgroup $H_V$ has infinite index in $\pi_1(M_f)$.
	Moreover, since $E$ is a codimension--$1$ face of $\mathcal{D}(M_f,\mathcal{C}_f)$,
	there is some nonzero cohomology class $\psi\in H^1(M_f;\Integral)$
	which vanishes on $L_E$,
	and $\psi$ lies on the boundary of the fibered cone $\mathcal{C}_f$,
	(see (\ref{cone_projective_dual})).
	In particular, $\psi$ is not a fibered class,
	so the kernel of the homomorphism $\psi\colon \pi_1(M_f)\to \Integral$ is infinitely generated
	(see \cite{Stallings-fibering}, or \cite[{Expos\'e 14, Theorem 14.2}]{FLP_book}).
	However, $H_V$ is finitely generated and is contained in the kernel of $\psi$,
	so $H_V$ must be geometrically finite in the Kleinian group $\pi_1(M_f)$,
	or equivalently, 
	quasi-convex in the word-hyperbolic group $\pi_1(M_f)$,
	(see \cite[Theorem 4.1.2 and Proposition 4.4.2]{AFW_book_group}).
\end{proof}

\begin{lemma}\label{new_vertex_D_cover}
	Let $M$ be an orientable connected closed $3$--manifold
	which admits a hyperbolic metric.
	If $M$ has a fibered cone $\mathcal{C}$,
	then there exists a finite cover $M'$ of $M$,
	such that $\mathcal{D}(M',\mathcal{C}')$ has a vertex 
	which projects into the interior of $\mathcal{D}(M,\mathcal{C})$.
\end{lemma}

Here, $\mathcal{C}'$ denotes the unique fibered cone for $M'$
which contains the image of $\mathcal{C}$ under the induced homomorphism
$H^1(M;\Real)\to H^1(M';\Real)$, and $\mathcal{D}(M',\mathcal{C}')$
denotes the projective dual polytope as defined in (\ref{cone_projective_dual}).
Note that the induced homomorphism $H_1(M';\Real)\to H_1(M;\Real)$ determines
a linear projection of the polytope $\mathcal{D}(M',\mathcal{C}')$ 
onto the polytope $\mathcal{D}(M_f,\mathcal{C}_f)$.

\begin{proof}
	Fix a primitive cohomology class $\phi\in H^1(M;\Integral)$ in $\mathcal{C}$,
	we identify $(M,\phi)$ with the mapping torus of a pseudo-Anosov automorphism $f\colon S\to S$
	and its distinguished cohomology class, namely, $(M_f,\phi_f)$.
	We obtain a finite collection $\mathcal{H}$ of infinite-index quasi-convex subgroups of $\pi_1(M_f)$
	as in Lemma \ref{facial_clusters}.
	
	Recall that a finitely generated group is said to be 
	\emph{virtually compact special} if some finite-index subgroup of the
	group is isomorphic to the fundamental group of a compact special cube complex
	(see \cite[Chapter 4, Definition 4.2]{Wise_book}; 
	compare \cite[Definition 3.2]{Haglund--Wise} and \cite[Definition 2.1]{AGM-MSQT}).
	Since $M_f$ is closed and hyperbolic,
	$\pi_1(M_f)$ is nonelementary word-hyperbolic and virtually compact special \cite[Theorem 9.3]{Agol_VHC}.
	By Wise's special quotient theorem \cite[Theorem 12.7]{Wise_book},	
	there exists a group quotient $\pi_1(M_f)\to G$,
	such that $G$ is again nonelementary word-hyperbolic and virtually compact special,
	and moreover, the image of every $H\in\mathcal{H}$ has finite image in $G$.
	(See also \cite[Proof of Lemma 9.4]{Liu_vhsr} 
	for an alternative approach through \cite{AGM-MSQT,GM-filling},
	and for more explanation about the claim of being nonelementary.)
	In particular, $G$ has some finite-index subgroup $G'$
	with $b_1(G')>0$, \cite[Corollary 1.2]{Agol_VHC}.
		
	We take $M'$ to be the finite cover of $M_f$,
	such that $\pi_1(M')$ is the preimage of $G'$ with respect to $\pi_1(M_f)\to G$.
	The pullback of $\phi_f$ is a (not necessarily primitive) fibered class $\phi'\in H^1(M';\Integral)$,
	and it lies in a unique fibered cone $\mathcal{C}'$,
	which contains the image of $\mathcal{C}=\mathcal{C}_f$.
	The suspension flow on $M_f$ lifts to a flow on $M'$.
	If one identifies $M'$ with the mapping torus of a pseudo-Anosov automorphism
	$f'\colon S'\to S'$, where $S'$ is a preimage component of $S$,
	then the lifted flow is just the suspension flow
	with velocity scaled down by the divisibility of $\phi'$.
	
	For any vertex $v'$ of $\mathcal{D}(M',\mathcal{C}')$,
	there exists a periodic trajectory $\gamma'$ of the lifted flow,
	such that $[\gamma']\in H_1(M';\Real)$ lies in the linear $1$--subspace $v'$,
	by Fried's dual characterization of $\mathcal{C}'$, 
	(see Section \ref{Subsec-polytopes_and_cones}).
	The periodic trajectory $\gamma'$ covers a unique periodic trajectory $\gamma$ in $M_f$.
	The free-homotopy class of $\gamma$ can be represented as
	the conjugacy class of some element $g$ in $\pi_1(M_f)$.
	If the vertex $v'$ projects into 
	a codimension--$1$ closed face of $\mathcal{D}(M_f,\mathcal{C}_f)$,
	the image of $g$ in $G$ has finite order,
	by the above construction and Lemma \ref{facial_clusters}.
	Then, in view of the commutative diagram
	$$\xymatrix{
	H_1(M';\Real) \ar[r] \ar[d] & H_1(G';\Real) \ar[d] \\
	H_1(M_f;\Real) \ar[r] & H_1(G;\Real)	
	}$$
	the $1$--subspace $v'$ of $H_1(M';\Real)$ must have trivial image in $H_1(G;\Real)$.
	Note that the homomorphisms in the above diagram are all surjective.
	
	By Fried's characterization,
	$\mathcal{D}(M',\mathcal{C}')$ has codimension $0$ in $\mathbf{P}(H_1(M';\Real))$,
	and 
	$\mathcal{D}(M_f,\mathcal{C}_f)$ has codimension $0$ in $\mathbf{P}(H_1(M_f;\Real))$.
	Since $b_1(G)>0$,
	there must be some vertex $w'$ of $\mathcal{D}(M',\mathcal{C}')$,
	and as a $1$--subspace of $H_1(M';\Real)$, 
	$w'$ has nonvanishing image in $H_1(G;\Real)$.
	The argument in last paragraph
	shows that $w'$ must project into the interior of $\mathcal{D}(M_f,\mathcal{C}_f)$.
	Therefore, $M'$ is a finite cover as asserted.
\end{proof}

\begin{lemma}\label{profinite_isomorphism_hyperbolic_closed}
	Adopt the assumptions of Theorem \ref{profinite_isomorphism_hyperbolic}.
	If $M_A$ is closed,
	then the conclusion of Theorem \ref{profinite_isomorphism_hyperbolic} holds true.	
\end{lemma}

\begin{proof}
%Let $(M_A,M_B,\Psi)$ be a profinite morphism setting of a $3$--manifold pair.
%Suppose that $\Psi\colon\widehat{\pi_A}\to\widehat{\pi_B}$ is an isomorphism.
%Suppose that $M_A$ and $M_B$
%are both closed and admit hyperbolic metrics.
Without loss of generality, we assume that $M_A$, and hence $M_B$, fibers over a circle,
and aim at proving $\mathrm{rank}_\Integral(\mathrm{MC}(\Psi_*))=1$.
This is because of Lemma \ref{profinite_isomorphism_hyperbolic_reduction}.
To be precise, one may otherwise pass to a pair of $\Psi$--corresponding finite cover 
$M^\dagger_A\to M_A$ and $M^\dagger_B\to M_B$, such that $M^\dagger_A$ fibers over a circle
(see Theorem \ref{quasi-fibered} and Definition \ref{corresponding_cover_def}).
Hence, $M^\dagger_B$ also fibers over a circle (Lemma \ref{both_fibered}).
Then it will suffice to prove $\mathrm{rank}_\Integral(\mathrm{MC}(\Psi^\dagger_*))=1$,
(Lemma \ref{profinite_isomorphism_hyperbolic_reduction}).

We make the following constructions:
Fix a fibered cone $\mathcal{C}_A$ in $H^1(M_A;\Real)$ for $M_A$.
Obtain a finite cover $M'_A$ of $M_A$ by Lemma \ref{new_vertex_D_cover}.
Thus, $\mathcal{D}(M'_A,\mathcal{C}'_A)$ has some vertex $v'_A$ that projects 
into the interior of $\mathcal{D}(M_A,\mathcal{C}_A)$,
where $\mathcal{C}'_A$ refers to 
the unique fibered cone for $M'_A$ that contains the image of $\mathcal{C}_A$,
with respect to the covering-induced homomorphism $H^1(M_A;\Real)\to H^1(M'_A;\Real)$.
Denote by $q_A$ the image of $v'_A$ in the interior of $\mathcal{D}(M_A,\mathcal{C}_A)$.
Obtain the finite cover $M'_B\to M_B$
which makes a $\Psi$--corresponding pair with $M'_A\to M_A$.
Fix an open dense subset $\mathcal{U}'$ of $\mathrm{Hom}_\Integral(\mathrm{MC}(\Psi_*),\Real)$
as provided in Lemma \ref{U_construction}.
Fix a pair of homomorphisms $\varepsilon\in\mathcal{U}'$ and 
$\varepsilon'\in\mathrm{Hom}_\Integral(\mathrm{MC}(\Psi'_*),\Real)$
which satisfy the conclusion of Lemma \ref{U_construction}.
In particular, $\Psi_*^\varepsilon\colon H^1(M_B;\Real)\to H^1(M_A;\Real)$
and $(\Psi')^{\varepsilon'}_*\colon H^1(M'_B;\Real)\to H^1(M'_A;\Real)$
are both nondegenerate.
Denote by $\mathcal{C}_B$ the image of $\mathcal{C}_A$ under $\Psi_*^\varepsilon$,
which is a fibered Thurston-norm cone for $M_B$ (Corollary \ref{fibered_cone_correspondence}).
Similarly, denote by $\mathcal{C}'_B$ the fibered Thurston-norm cone for $M'_B$
given as the image of $\mathcal{C}'_A$ under $(\Psi')_*^{\varepsilon'}$.
Note that $\mathcal{C}_B$ embeds into $\mathcal{C}'_B$
under $H^1(M_B;\Real)\to H^1(M'_B;\Real)$,
as shown by dualizing the commutative diagram in Lemma \ref{U_construction},
(see also (\ref{Psi_dual_cov})).

We observe the commutative diagram
\begin{equation}\label{Psi_D_cov}
\xymatrix{
\mathcal{D}(M'_A,\mathcal{C}'_A) \ar[rr]^-{\mathbf{P}\left((\Psi')^{\varepsilon'}_*\right)} \ar[d]_{\mathbf{P}(\mathrm{cov}_*)} & &
\mathcal{D}(M'_B,\mathcal{C}'_B) \ar[d]^{\mathbf{P}(\mathrm{cov}_*)} \\
\mathcal{D}(M_A,\mathcal{C}_A) \ar[rr]^-{\mathbf{P}\left(\Psi^{\varepsilon}_*\right)} & &
\mathcal{D}(M_B,\mathcal{C}_B)  
}
\end{equation}
Every codimension--$1$ open face of $\mathcal{C}_A$ is a codimension--$1$ open Thurston--norm cone for $M_A$,
and the linear hyperplane in $H^1(M_A;\Real)$ determines a projective hyperplane in $\mathbf{P}(H^1(M_A;\Real))$,
which is dual to a vertex of $\mathcal{D}(M_A,\mathcal{C}_A)$ in $\mathbf{P}(H_1(M_A;\Real))$.
The same relation holding for $M_B$, 
it follows that the projective linear isomorphism $\mathbf{P}(\Psi^{\varepsilon}_*)$ maps
$\mathcal{D}(M_A,\mathcal{C}_A)$ onto $\mathcal{D}(M_B,\mathcal{C}_B)$.
Similarly, the projective linear isomorphism $\mathbf{P}((\Psi')^{\varepsilon'}_*)$ maps
$\mathcal{D}(M'_A,\mathcal{C}'_A)$ onto $\mathcal{D}(M'_B,\mathcal{C}'_B)$.
Therefore, the image of $v'_A$ under $\mathbf{P}((\Psi')^{\varepsilon'}_*)$
is a vertex $v'_B$ of $\mathcal{D}(M'_B,\mathcal{C}'_B)$,
and the image of $q_A$ under $\mathbf{P}(\Psi^{\varepsilon}_*)$ 
is an interior point $q_B$ of $\mathcal{D}(M_B,\mathcal{C}_B)$.
Moreover, the commutative diagram (\ref{Psi_D_cov}) shows that $v'_B$ projects onto $q_B$
under the covering-induced projection.

For any $\tilde{\varepsilon}$ in $\mathcal{U}'$ sufficiently close to $\varepsilon$,
we require $\tilde{\varepsilon}'$ in $\mathrm{Hom}_\Integral(\mathrm{MC}(\Psi_*),\Real)$
sufficiently close to $\varepsilon'$, as in Lemma \ref{U_construction}.
Then $\Psi^*_{\tilde{\varepsilon}}$ induces the same Thurston-norm cone correspondence as 
$\Psi^*_{\varepsilon}$ does, and similarly
with $(\Psi')^*_{\tilde{\varepsilon}'}$ and $(\Psi')^*_{\varepsilon'}$.
It follows that $\mathbf{P}(\Psi^*_{\tilde{\varepsilon}})$ must also map
$q_A$ to $q_B$, (and also map the vertices of $\mathcal{D}(M_A,\mathcal{C}_A)$
to the vertices of $\mathcal{D}(M_B,\mathcal{C}_B)$ 
the same way as $\Psi^*_{\varepsilon}$ does).
Applying Lemma \ref{new_point_criterion},
we conclude that $\mathrm{MC}(\Psi_*)$ is a free $\Integral$--module of rank $1$,
which implies Theorem \ref{profinite_isomorphism_hyperbolic} 
by Lemma \ref{profinite_isomorphism_hyperbolic_reduction}.
Therefore, the conclusion of Theorem \ref{profinite_isomorphism_hyperbolic}
holds true if both $M_A$ and $M_B$ are closed.
\end{proof}

\subsection{The cusped case}\label{Subsec-cusped_hyp}
We derive the cusped case of Theorem \ref{profinite_isomorphism_hyperbolic}
from the closed case (Lemma \ref{profinite_isomorphism_hyperbolic_closed})
as follows.

\begin{lemma}\label{profinite_isomorphism_hyperbolic_robust}
	Adopt the assumptions of Theorem \ref{profinite_isomorphism_hyperbolic}.
	If the homomorphism $H_1(\partial M_A;\Rational)\to H_1(M_A;\Rational)$ 
	induced by the boundary inclusion has nontrivial cokernel,
	then the conclusion of Theorem \ref{profinite_isomorphism_hyperbolic} holds true.	
\end{lemma}

\begin{proof}
	We enumerate the boundary tori of $M_A$ as $T_{A,i}$, indexing with a finite set $I$,
	and for each $T_{A,i}$, fix a conjugate $P_{A,i}\cong\pi_1(T_{A,i})$ 
	of the associated peripheral subgroup in $\pi_1(M_A)$.
	Under the assumption of Lemma \ref{profinite_isomorphism_hyperbolic_robust},
	$\Psi$ determines a bijective correspondence
	between the boundary tori of $M_A$ and those of $M_B$.
	In fact, this follows from Wilton and Zalesski \cite[Proposition 3.1]{WZ_decomposition},
	and in particular, $M_A$ and $M_B$ have the same number of boundary tori.
	To be precise, we can choose peripheral conjugates
	$P_{B,i}\cong\pi_1(T_{B,i})$ in $\pi_1(M_B)$, similarly as above,
	and the correspondence of the boundary tori says that upon suitable re-indexing,
	$\Psi\colon\widehat{\pi_A}\to\widehat{\pi_B}$ 
	maps the closure of each $P_{A,i}$
	isomorphically onto a conjugate of the closure of $P_{B,i}$,
	for all $i\in I$.
	Note that any finitely generated subgroup
	is separable in a hyperbolic $3$--manifold group,
	so its closure in the profinite completion 
	is naturally isomorphic to its own profinite completion,
	(see \cite[Corollary 4.2.3]{AFW_book_group}).
	
	We enumerate the $1$--dimensional cone carriers for $M_A$
	as $l_{A,j}$, indexing with a finite set $J$,
	and there are corresponding cone carriers $l_{B,j}$ for $M_B$,
	(Corollary \ref{cone_carrier_correspondence}).
	Denote by $K_{A,j}=K(l_{A,j})$ the lattice kernels in $H_A$,
	and similarly $K_{B,j}$ in $H_B$, (see (\ref{lattice_kernel_def})).
	By Corollary \ref{cone_carrier_correspondence}, 
	the closure of each $K_{A,j}$ projects isomorphically onto 
	the closure of $K_{B,j}$
	under $\Psi_*\colon\widehat{H_A}\to \widehat{H_B}$.
	Note that
	$H^1(M_A;\Real)=\sum_{j\in J} l_{A,j}$ and $H^1(M_B;\Real)=\sum_{j\in J}^h l_{B,j}$,
	since the Thurston-norm unit balls for $M_A$ and $M_B$ are both polytopes,
	and hence,
	$\cap_{j\in J} K_{A,j}=\{0\}$ and $\cap_{j\in J} K_{B,j}=\{0\}$.
		
	We observe some pairs of corresponding boundary slopes as follows.
	By a \emph{slope} on a torus, we mean an essential simple closed curve up to isotopy.
	For each $i\in I$,
	the image of $P_{A,i}$ under the free abelianization $\pi_A\to H_A$	has rank either $1$ or $2$.
	For example, this follows from the well-known fact that
	for any orientable compact $3$--manifold $N$,
	the kernel and the image of the natural homomorphism 
	$H_1(\partial N;\Rational)\to H_1(N;\Rational)$
	both have dimension $b_1(\partial N)/2$,
	applying to a Dehn filling of $M_A$	that leaves only $T_{A,i}$ unfilled.
	Since $H^1(M_A;\Real)=\sum_{j\in J} l_{A,j}$,
	there exists some $j\in J$, such that $K_{A,j}$ does not contain the image of $P_{A,i}$.
	For any such $j$, 
	the preimage of $K_{A,j}$ in $P_{A,i}$ is an infinite cyclic subgroup $C_{A,i}^j$,
	such that $P_{A,i}/C_{A,i}^j$ is infinite cyclic,
	and $C_{A,i}^j$ represents a unique slope $c^j_{A,i}$ on the torus $T_{A,i}$.
	With the same $(i,j)\in I\times J$, the same construction with $P_{B,i}$ and $K_{B,j}$
	yields a slope $c^j_{B,i}$ on $T_{B,i}$,
	which is represented by an infinite cyclic subgroup $C_{B,i}^j$ of $P_{B,i}$.
	The above profinite correspondences imply that 
	$\Psi\colon \widehat{\pi_A}\to\widehat{\pi_B}$ maps the closure of $C_{A,i}^j$
	isomorphically onto a conjugate of the closure of $C_{B,i}^j$,
	and in this sense, 
	we say that the slopes $c_{A,i}^j$ and $c_{B,i}^j$ correspond to each other.
	
	We can design corresponding Dehn fillings	of $M_A$ and $M_B$,
	by choosing corresponding pairs of slopes on all the corresponding pairs of boundary tori.
	To be precise, we say that a map $s\colon I\to J$ is an \emph{admissible filling pattern},
	if $K_{A,s(i)}$ does not contain the image of $P_{A,i}$ for any $i\in I$.
	The above observation shows the existence of admissible filling patterns.
	For any admissible filling pattern $s$,
	we obtain a Dehn filling $M_{A,s}$ of $M_A$ along the slopes $c^{s(i)}_{A,i}$ for all $i\in I$.
	Namely, $M_{A,s}$ is the orientable connected closed $3$--manifold
	obtained from $M_A$, by attaching solid tori to all the boundary tori $T_{A,i}$,
	such that the slopes $c^{s(i)}_{A,i}$ bounds meridian disks in the attached solid tori.
	By the van Kampen theorem, the fundamental group $\pi_{A,s}=\pi_1(M_{A,s})$
	is the quotient of $\pi_A$	by the smallest normal subgroup that contains all $C_{A,i}^{s(i)}$,
	so the profinite completion $\widehat{\pi_{A,s}}$ is the quotient of $\widehat{\pi_A}$
	by the smallest normal closed subgroup that contains all $\mathrm{clos}(C_{A,i}^{s(i)})$.
	Similarly, we obtain the Dehn filling $M_{B,s}$ of $M_B$.
	The above correspondence between slopes shows that $\Psi$ descends
	to a profinite isomorphism $\Psi_s\colon \widehat{\pi_{A,s}}\to \widehat{\pi_{B,s}}$.
	
	Lemma \ref{profinite_isomorphism_hyperbolic_closed} would apply to
	the setting $(M_{A,s},M_{B,s},\Psi_s)$, if 
	the Dehn filling $M_{A,s}$ (and hence $M_{B,s}$) admits a hyperbolic metric.
	In general, the Dehn filling could be exceptional,
	but we claim that 
	the following rank-one property still holds for any admissible filling pattern $s$:
	\begin{equation}\label{filling_rank_one}
	\mathrm{rank}_\Integral (\mathrm{MC}(\Psi_{s\,*};H_{A,s},H_{B,s}))=1
	\end{equation}	
	where $H_{A,s}=H_1(M_{A,s};\Integral)_{\mathtt{free}}$ 
	and $H_{B,s}=H_1(M_{B,s};\Integral)_{\mathtt{free}}$,
	(see Definition \ref{MC_def}).
	To avoid distraction,
	let us assume this claim for the moment,
	and explain how Lemma \ref{profinite_isomorphism_hyperbolic_robust}
	follows from it.
	
	To this end, we denote $\bar{H}_A=H_1(M_A,\partial M_A;\Integral)_{\mathtt{free}}$,
	which can be naturally identified with the quotient of $H_A$ by
	the smallest direct summand that contains the images of all $P_{A,i}$ under $\pi_A\to H_A$.
	Note that $\bar{H}_A$ has positive rank 
	since $H_1(\partial M_A;\Rational)\to H_1(M_A;\Rational)$ has nontrivial cokernel.
	Similarly, we denote $\bar{H}_B=H_1(M_B,\partial M_B;\Integral)_{\mathtt{free}}$.
	Therefore, $\Psi_*$ descends to an isomorphism $\bar{\Psi}_*$	
	between the profinite completions of $\bar{H}_A$ and $\bar{H}_B$.
	We observe that $\bar{H}_A$ and $H_B$ are common quotients
	of $H_{A,s}$ and $H_{B,s}$ for all admissible filling patterns $s$,
	respectively.
	Moreover, the induced quotient homomorphisms of the profinite completions
	naturally commutes with $\bar{\Psi}_*$ and $\Psi_{s\,*}$.
	Then the matrix coefficient module $\mathrm{MC}(\bar{\Psi}_*;\bar{H}_A,\bar{H}_B)$ 
	is contained in all $\mathrm{MC}(\Psi_{s\,*};H_{A,s},H_{B,s})$,
	by Proposition \ref{MC_property} (1), 
	(choosing bases of $H_{A,s}$ by lifting bases of $\bar{H}_A$ and completing, 
	and choosing bases of $H_{B,s}$ similarly).
	Using the claimed (\ref{filling_rank_one}),
	$\mathrm{MC}(\bar{\Psi}_*;\bar{H}_A,\bar{H}_B)$ is commensurable with 
	$\mathrm{MC}(\Psi_{s\,*};H_{A,s},H_{B,s})$ for all admissible filling patterns $s$,
	so these matrix coefficient modules are mutually commensurable 
	as $\Integral$--submodules in $\widehat{\Integral}$.
	In particular, $\mathrm{MC}(\oplus_s\Psi_{s\,*};\,\oplus_s H_{A,s},\,\oplus_s H_{B,s})$
	also has rank $1$, where $s$ ranges over all the admissible filling patterns.
	
	On the other hand, we observe that the quotient homomorphisms $H_A\to H_{A,s}$
	induced by the natural quotient homomorphisms $\pi_A\to \pi_{A,s}$
	give rise to an embedding $H_A\to \oplus_s H_{A,s}$.
	In fact, for any nonzero $u\in H_A$, there is some $K_{A,j}$ which does not contain $u$,
	so there is an admissible filling pattern $t\colon I\to J$ 
	with $t(i)=j$, unless $K_{A,j}$ contains $P_{A,i}$.
	Then the kernel of $H_A\to H_{A,t}$ is contained in $K_{A,j}$,
	so $u$ survives under $H_A\to H_{A,t}$.
	Similarly, $H_B\to \oplus_s H_{B,s}$ is also an embedding.
	By Proposition \ref{MC_property} (3),
	the rank of $\mathrm{MC}(\Psi_*;H_A,H_B)$ is 
	at most the rank of $\mathrm{MC}(\oplus_s\Psi_{s\,*};\,\oplus_s H_{A,s},\,\oplus_s H_{B,s})$,
	which equals $1$.
	Therefore, $\mathrm{MC}(\Psi_*;H_A,H_B)$ has rank $1$,
	and the conclusion of Theorem \ref{profinite_isomorphism_hyperbolic}
	follows from Proposition \ref{MC_property} (2).
	In other words, 
	we have proved Lemma \ref{profinite_isomorphism_hyperbolic_robust},
	assuming the claimed rank-one property (\ref{filling_rank_one}).
	
	It remains to prove the claim (\ref{filling_rank_one}).
	We appeal to the orbifold hyperbolic Dehn filling theorem \cite[Chapter 8, Theorem 8.4]{BMP_book_orbifold}.
	The procedure is conceptually simple but notationally a bit involved.
	The hyperbolic $3$--orbifold showing up below only plays an auxiliary role.
	One may alternatively employ other more general constructions, 
	such as fillings of relatively hyperbolic groups \cite{GM-filling}.
	
	Let $s$ be any admissible filling pattern.
	For simplicity, we fix $s$ below and drop the superscript $s(i)$ in slope related terms,
	writing $c_{A,i}$, $C_{A,i}$, $c_{B,i}$, $C_{B,i}$. 
	For sufficiently large positive integers $m_i$, $i\in I$,
	take the orbifold Dehn filling of $M_A$ along the $m_i$--cyclic covers of the slopes $c_{A,i}$.
	The result is an orientable closed hyperbolic $3$--orbifold $\mathcal{O}_{A,s}$,
	whose underlying $3$--manifold is $M_{A,s}$, and whose singular loci
	are the core curves in the solid tori attached to $T_{A,i}$, 
	with cone singularity of order $m_i$.
	Take a regular finite manifold cover $M'_{A,s}\to \mathcal{O}_{A,s}$.
	so $M'_{A,s}$ is an orientable closed hyperbolic $3$--manifold.
	The composite map $M'_{A,s}\to\mathcal{O}_{A,s}\to M_{A,s}$ is a branched covering map,
	ramifying on the filling core curves of orders $m_i$.
	Equivalently, 
	$M'_{A,s}$ can be obtained as the Dehn filling of a regular finite cover $M'_A$ of $M_A$.
	To be precise, $M'_A\to M_A$ is the pullback of $M'_{A,s}\to \mathcal{O}_{A,s}$ 
	with respect to the orbifold Dehn filling map $M_A\to \mathcal{O}_{A,s}$.
	By construction, every preimage component $c'$ of $c_{A,i}$ in $M'_A$ is a boundary slope,
	which covers $c_{A,i}$ cyclically of degree $m_i$.
	Therefore, $M'_{A,s}$ is the Dehn filling of $M'_A$,	
	such that for any boundary torus $T'_A$ that covers $T_{A,i}$, 
	$T'_A$ is filled along the slope $c'_A$ that covers $c_{A,i}$.
		
	Take $M'_B\to M_B$ to be the regular finite cover corresponding to $M'_A\to M_A$,
	with respect to $\Psi$, (Definition \ref{corresponding_cover_def}). 
	%By \cite[3.1]{WZ_decomposition} again (or by direct check),
	%every boundary torus $T'_A$ corresponds to a unique boundary $T'_B$ of $M'_B$.
	Note that any boundary torus $T'_A$ of $M'_A$ which covers $T_{A,i}$
	is represented by the intersection of $g P_{A,i}g^{-1}$ with $\pi'_A$, for some $g\in \pi_A$,
	and the slope $c'_A$ on $T'_A$ which covers $c_{A,i}$	is represented 
	by the intersection of $g C_{A,i} g^{-1}$ with $\pi'_A$.
	Under $\Psi$, the closure of $g P_{A,i}g^{-1}$ in $\widehat{\pi_A}$
	projects isomorphically onto the closure of $h P_{B,i} h^{-1}$ in $\widehat{\pi_B}$, for some $h\in\widehat{\pi_B}$,
	and the closure of $g C_{A,i} g^{-1}$ projects isomorphically onto the closure of $h C_{B,i} h^{-1}$.
	It follows that the corresponding slope $c'_B$ on the corresponding torus $T'_B$ of $M'_B$
	covers the slope $c_{B,i}$ on $T_{B,i}$ with degree $m_i$.
	Therefore, we can construct the Dehn filling $M'_{B,s}$ of $M'_B$ along the slopes $c'_B$
	on all the boundary tori, and a branched covering map $M'_{B,s}\to M_{B,s}$,
	ramifying on the filling core curves of orders $m_i$.
	It follows that the restricted isomorphism $\Psi'$ between the profinite completions of $\pi'_A$ and $\pi'_B$
	descends to an isomorphism $\Psi'_s$ 
	between the profinite completions of $\pi'_{A,s}=\pi_1(M'_{A,s})$ and $\pi'_{B,s}=\pi_1(M'_{B,s})$.
	In particular, $M'_{B,s}$ is an orientable connected closed hyperbolic $3$--manifold,
	by \cite[Proposition 3.1]{WZ_decomposition}.
	Moreover, we obtain a commutative diagram of profinite completions of fundamental groups,
	and another of profinite completions of the free abelianizations:
	$$\xymatrix{
	\widehat{\pi'_{A,s}} \ar[r]^-{\Psi'_s} \ar[d] & \widehat{\pi'_{B,s}} \ar[d] 
		& & \widehat{H'_{A,s}} \ar[r]^-{\Psi'_{s\,*}} \ar[d] & \widehat{H'_{B,s}} \ar[d] \\
	\widehat{\pi_{A,s}} \ar[r]^-{\Psi_s} & \widehat{\pi_{B,s}}  
		& & \widehat{H_{A,s}} \ar[r]^-{\Psi_{s\,*}} & \widehat{H_{B,s}}
	}$$
	
	As Lemma \ref{profinite_isomorphism_hyperbolic_closed} applies to the setting
	$(M'_{A,s},M'_{B,s},\Psi'_s)$,
	the rank of $\mathrm{MC}(\Psi'_{s\,*};H'_{A,s},H'_{B,s})$ equals $1$.
	By the above construction, $\pi'_{A,s}$ projects onto a finite-index normal subgroup of $\pi_{A,s}$,
	so $H'_{A,s}$ projects onto a finite-index $\Integral$--submodule of $H_{A,s}$,
	and $H'_{B,s}\to H_{B,s}$ has the same property.
	It follows that $\mathrm{MC}(\Psi_{s\,*};H_{A,s},H_{B,s})$
	is commensurable with a $\Integral$--submodule of $\mathrm{MC}(\Psi'_{s\,*};H'_{A,s},H'_{B,s})$
	in $\widehat{\Integral}$, by Proposition \ref{MC_property} (1) and (3).
	Therefore, $\mathrm{MC}(\Psi_{s\,*};H_{A,s},H_{B,s})$ has rank $1$,
	as claimed in (\ref{filling_rank_one}).
	This completes the proof of Lemma \ref{profinite_isomorphism_hyperbolic_robust}.
\end{proof}

We conclude the proof of Theorem \ref{profinite_isomorphism_hyperbolic_robust} as follows.
Let $(M_A,M_B,\Psi)$ be a profinite morphism setting of a $3$--manifold pair.
Suppose that $\Psi\colon\widehat{\pi_A}\to\widehat{\pi_B}$ is an isomorphism.
Suppose that $M_A$ and $M_B$ both admit complete hyperbolic metrics of finite volume in the interior.

Take a finite cover of $M'_A\to M_A$, 
such that $M'_A$ contains a nonseparating, embedded, orientable closed subsurface,
for example, as constructed by Cooper, Long, and Reid \cite[Theorem 2.3]{CLR_closed_surface}.
This property guarantees that $H_1(\partial M'_A;\Rational)\to H_1(M'_A;\Rational)$
has nontrivial cokernel,
because
some loop in $M'_A$ intersects the surface exactly once,
whereas any loop on $\partial M'_A$ has no intersection with the surface.
Obtain the corresponding finite cover $M'_B\to M_B$ with respect to $\Psi$, (Definition \ref{corresponding_cover_def}).

By Lemma \ref{profinite_isomorphism_hyperbolic_robust},
the matrix coefficient module $\mathrm{MC}(\Psi'_*;H'_A,H'_B)$ has rank $1$.
By Lemma \ref{profinite_isomorphism_hyperbolic_reduction},
$\mathrm{MC}(\Psi_*;H_A,H_B)$ also has rank $1$, as asserted.

This completes the proof of Theorem \ref{profinite_isomorphism_hyperbolic}.

\section{Profinite invariance of twisted Reidemeister torsions}\label{Sec-TRT}
In this section, we show that for profinitely isomorphic pairs
of orientable finite-volume hyperbolic $3$--manifolds,
corresponding fibered classes give rise to the same twisted Reidemeister torsion,
with respect to corresponding $\Rational$--representations
with finite image (Theorem \ref{profinite_invariance_TRT}).

\begin{definition}\label{corresponding_quotient_def}
Let $(M_A,M_B,\Psi)$ be a profinite morphism setting of a $3$--manifold pair
where $\Psi$ is an isomorphism (Convention \ref{profinite_morphism_setting}).
Let $\Gamma$ be a finite group.
A pair of homomorphisms $\gamma_A\colon \pi_A\to\Gamma$ and 
$\gamma_B\colon \pi_B\to \Gamma$
are said to be a \emph{$\Psi$--corresponding pair} if
$\gamma_A$ is the composite homomorphism
$$\xymatrix{\pi_A \ar[r]^-{\mathrm{incl}} & \widehat{\pi_A} \ar[r]^-{\Psi} & \widehat{\pi_B} \ar[r]^-{\widehat{\gamma_B}} 
& \Gamma.}$$
\end{definition}

\begin{theorem}\label{profinite_invariance_TRT}
Let $(M_A,M_B,\Psi)$ be a profinite morphism setting of a $3$--manifold pair (Convention \ref{profinite_morphism_setting}).
Suppose that $\Psi\colon\widehat{\pi_A}\to\widehat{\pi_B}$ is an isomorphism.
Suppose that $M_A$ and hence $M_B$
both have positive first Betti number, 
and admit complete hyperbolic metrics of finite volume in the interior.
Let $\mu\in\widehat{\Integral}^\times$ be a unit as provided in Theorem \ref{profinite_isomorphism_hyperbolic}.
%Suppose $\mathrm{MC}(\Psi_*)$ is the $\Integral$--submodule $\mu\Integral$ of $\widehat{\Integral}$
%for some invertible $\mu\in\widehat{\Integral}^\times$.
Suppose that $\gamma_A\colon\pi_A\to\Gamma$ and $\gamma_B\colon\pi_B\to \Gamma$
be a $\Psi$--corresponding pair of finite quotients.

Suppose that $\phi_B\in H^1(M_B;\Integral)$ is a fibered cohomology class.
Let $\psi_A\in H^1(H_A;\Integral)$ be the fibered cohomology class $\Psi^*_{1/\mu}(\phi_B)$.
For any representation $\rho\colon\Gamma\to \mathrm{GL}(k,\Rational)$,
let $\rho_A\colon \pi_A\to \mathrm{GL}(k,\Rational)$ 
be the pullback $\gamma^*_A(\rho)$, and $\rho_B$ be $\gamma^*_B(\rho)$.
Then, the following equality holds in $\Rational[t^{\pm1}]$ up to monomial factors 
with coefficient in $\Rational^\times$:
$$\tau^{\rho_A,\psi_{A}}(t)\doteq\tau^{\rho_B,\phi_B}(t).$$
%
%
%$$\tau^{\rho_A,\psi_{A}}\left(M_A;\Rational[t^{\pm1}]^k\right)
%\doteq
%\tau^{\rho_B,\phi_B}\left(M_B;\Rational[t^{\pm1}]^k\right).
%$$
\end{theorem}

The rest of this section is devoted to the proof of Theorem \ref{profinite_invariance_TRT}.
Our proof invokes the following criterion for identifying reciprocal polynomials over $\Integral$
through principal ideals of the completion group algebra 
$\llbracket\widehat{\Integral}t^{\widehat{\Integral}}\rrbracket$,
due to J.~Ueki \cite{Ueki}:

\begin{theorem}\label{Ueki_lemma}
	Let $a(t),b(t)\in\Integral[t]$ be a pair of reciprocal polynomials,
	and $\mu\in\widehat{\Integral}^\times$ be a unit.
	If the principal ideals $(a(t^\mu))$ and $(b(t))$
	of $\llbracket\widehat{\Integral}t^{\widehat{\Integral}}\rrbracket$ are equal to each other,
	then $a(t)\doteq b(t)$ holds in $\Integral[t^{\pm1}]$ 
	up to monomial factors with coefficients $\pm1$.
\end{theorem}

\begin{remark}\label{Ueki_lemma_remark}
	Recall that a polynomial $f(t)\in\Integral[t]$ is said to be \emph{reciprocal}
	if it takes the form $c_0+c_1t+\cdots+c_rt^r$, such that $c_i=c_{r-i}$ for all $i=0,\cdots,r$.
	Theorem \ref{Ueki_lemma} is quoted from \cite[Lemma 3.6]{Ueki}.
	In fact, Ueki shows that for any $\mu\in\widehat{\Integral}^\times$,
	the fraction $\Phi_m(t^\mu)/\Phi_m(t)$ is defined and is a unit
	in $\llbracket\widehat{\Integral}t^{\widehat{\Integral}}\rrbracket$,
	where $\Phi_m(t)\in\Integral[t]$ denotes the $m$--th cyclotomic polynomial,
	(see \cite[Lemma 3.4]{Ueki}).
	This allows one to identify the cyclotomic factors
	of the compared reciprocal polynomials.
	The other factors are identified using a criterion of D.~Fried \cite{Fried_resultant},
	which determines reciprocal polynomials through their cyclic resultants,
	(see also \cite[Proposition 3.1]{Ueki}). 
\end{remark}

In certain situations we are able to obtain polynomial principal ideals of 
$\llbracket\widehat{\Integral}t^{\widehat{\Integral}}\rrbracket$
and compare them easily.
For any profinite group $G$, 
it makes sense to speak of profinite integral powers of elements.
Namely, for any $\nu\in\widehat{\Integral}$ and $g\in G$, 
the \emph{$\nu$--power} of $g$ refers to the unique element in $G$ as follows, 
\begin{equation}\label{g_to_nu}
	g^\nu=\varprojlim_{N}\,g^n\bmod N, 
\end{equation}
where $N$ ranges over the inverse system of the open normal subgroups of $G$, and 
$n\in\Integral$ is congruent to $\nu$ modulo the order of the quotient group $|G/N|$.
Note that $hg^\nu h^{-1}=(hgh^{-1})^\nu$ holds for all $h\in G$.

\begin{lemma}\label{Ueki_lemma_matrix}
	Suppose $A,B\in\mathrm{GL}(r,\Integral)$ and $\mu\in\widehat{\Integral}^\times$.
	Denote $a(t)=\mathrm{det}_{\Integral[t]}(1-tA)$ and $b(t)=\mathrm{det}_{\Integral[t]}(1-tB)$.
	If $A$ is conjugate to the $\mu$--power matrix $B^\mu$ in $\mathrm{GL}(r,\widehat{\Integral})$,
	then $(a(t^\mu))=(b(t))$ holds as principal ideals in $\llbracket\widehat{\Integral}t^{\widehat{\Integral}}\rrbracket$.
	%Hence, $a(t)a(t^{-1})=b(t)b(t^{-1})$ holds in $\Integral[t^{\pm1}]$.
\end{lemma}

\begin{proof}
	We rewrite $\llbracket\widehat{\Integral}t^{\widehat{\Integral}}\rrbracket$ as $\widehat{R}$.
	For any $l,d\in\Natural$,
	denote by $R_{l,d}$ the group algebra
	$[(\Integral/l\Integral)t^{\Integral/m\Integral})]$,
	which is a finite ring quotient of $\widehat{R}$.
	Denote by $a_{l,d}(t^\mu)$ and $b_{l,d}(t)$ the image of 
	$a(t^\mu)$ and $b(t)$ in $R_{l,d}$.
	So we have 
	$b_{m,l}(t)=\mathrm{det}_{l,d}(1-t B)$,
	where $\mathrm{det}_{l,d}$ denotes the determinant on square matrices over $R_{l,d}$.
	We have
	$a_{l,d}(t^\mu)=\mathrm{det}_{l,d}(1-t^\mu A)=\mathrm{det}_{l,d}(1-t^\mu B^\mu)$,
	since $B^\mu$ is conjugate to $A$ in $\mathrm{GL}(r,\widehat{\Integral})$. 
	Let $m$ be some positive integer 
	which is congruent to $\mu$ modulo $d\cdot|\mathrm{GL}(r,\Integral/l\Integral)|$.
	Since $1-t^\mu B^\mu=1-t^m B^m=(1-tB)(1+tB+\cdots+t^{m-1}B^{m-1})$ holds
	as square matrices over $R_{m,l}$,
	the principal ideal $(a_{m,l}(t^\mu))$ of $R_{m,l}$ 
	is contained in $(b_{m,l}(t))$.
	Passing to the inverse limit, 
	it follows that the principal ideal $(a(t^{\mu}))$ of $\widehat{R}$
	is contained in $(b(t))$.
	Since $\mu$ is invertible, 
	the $\mu^{-1}$--power matrix $A^{\mu^{-1}}$ is conjugate to $B$ in $\mathrm{GL}(r,\widehat{\Integral})$,
	so the same argument shows that $(b(t^{\mu^{-1}}))$ is contained in $(a(t))$ in $\widehat{R}$.
	Using the profinite ring isomorphism $\widehat{R}\to\widehat{R}$
	determined by $t\mapsto t^\mu$,
	we see that $(b(t))$ is contained in $(a(t^\mu))$.
	Therefore, $(a(t^\mu))=(b(t))$ holds in $\widehat{R}$.	
\end{proof}

%Ueki also shows that any nonzero $f(t)\in[\Integral t^\Integral]$ is not a zero-divisor
%in $\llbracket\widehat{\Integral}t^{\widehat{\Integral}}\rrbracket$.
%In particular, 
%$\Rational\otimes_\Integral\llbracket\widehat{\Integral}t^{\widehat{\Integral}}\rrbracket$
%can be identified as the ring of fractions 

%\begin{corollary}\label{Ueki_lemma_corollary}
	%Let $f(t),g(t)\in\Rational[t]$ be a pair of reciprocal polynomials,
	%and $\mu\in\widehat{\Integral}^\times$ be a unit.
	%If the principal ideals $(f(t))$ and $(g(t^\mu))$
	%of $\Rational\otimes_\Integral\llbracket\widehat{\Integral}t^{\widehat{\Integral}}\rrbracket$ are equal to each other,
	%then $f(t)=r\cdot g(t)$ in $\Rational[t]$ for some $r\in\Rational^\times$.
%\end{corollary}
%
%\begin{proof}
	%Let $A=\llbracket\widehat{\Integral}t^{\widehat{\Integral}}\rrbracket$ and $S=\Integral\setminus \{0\}$.
	%Since $\Rational\otimes_\Integral\llbracket\widehat{\Integral}t^{\widehat{\Integral}}\rrbracket$
	%is the ring of fractions $S^{-1}A$,
	%elements of $S^{-1}A$ can be written as $\alpha/n$ for some $\alpha\in A$ and $n\in S$.
	%Note that any $s\in S$ is not a zero-divisor in $A$,
	%and indeed,
	%any nonzero element of $\Integral[t]$ is not a zero-divisor in $A$,
	%(see \cite[Lemma 3.3]{Ueki}).
	%If $(f(t))=(g(t^\mu))$ in $S^{-1}A$, we have 
	%$f(t)=g(t^{\mu})\alpha/n$ and $g(t^{\mu})=f(t)\beta/m$ for some $\alpha/n,\beta/m\in S^{-1}A$,
	%it follows that $mn=\alpha\beta$ holds in $A$.
	%We may also assume $n\cdot f(t), m\cdot g(t)\in \Integral[t]$ by raising $m$ and $n$
	%to some positive integral multiple.
	%Then it follows 
	%
%\end{proof}

In order to prove Theorem \ref{profinite_invariance_TRT},
one may essentially work with representations
$\Gamma\to\mathrm{GL}(k,\Integral)$.
We start by clarifying this point,
explaining a relation between twisted invariants over $\Integral$ and over $\Rational$.
%The following Lemmas \ref{Q_to_Z} and \ref{twisted_homology_Z_to_Q}
%are probably well-known to experts.

\begin{lemma}\label{Q_to_Z}
		Let $\Gamma$ be a finite group.
		Then, any representation $\rho\colon\Gamma\to\mathrm{GL}(k,\Rational)$
		is isomorphically realizable over $\Integral$.
		Namely, $\rho$ is conjugate to the induced representation $\sigma_\Rational$ over $\Rational$
		of some representation $\sigma\colon\Gamma\to\mathrm{GL}(k,\Integral)$.
\end{lemma}

\begin{proof}
	Let $H$ be the intersection in $\Rational^k$ of 
	the finitely many lattices $\rho(g)(\Integral^k)$,
	where $g$ ranges over $\Gamma$.
	Then $H$ is a lattice of $\Rational^k$
	and is invariant under  $\Gamma$.
	Let $T\in\mathrm{GL}(k,\Rational)$ 
	be a square matrix whose column vectors form a basis of $H$.
	Then $T^{-1}\rho(g)T$ is the matrix of $\rho(g)$ with respect to that basis of $H$.
	It follows that $\sigma(g)=T^{-1}\rho(g)T$
	defines a representation $\sigma\colon\Gamma \to \mathrm{GL}(k,\Integral)$ as asserted.
\end{proof}

Since twisted Reidemeister torsions are invariant under conjugation of the representation
(see \cite[Section 3.3.2, Lemma 1]{Friedl--Vidussi_survey}), 
we henceforth assume without loss of generality that 
the representation $\rho$ in Theorem \ref{profinite_invariance_TRT}
satisfies
\begin{equation}\label{sigma_def}
\rho=\sigma_\Rational
\end{equation}
for some representation $\sigma\colon\Gamma\to \mathrm{GL}(k,\Integral)$,
(Lemma \ref{Q_to_Z}).
We denote $\sigma_A\colon\pi_A\to \mathrm{GL}(k,\Integral)$ the pull-back representation $\gamma^*_A(\sigma)$,
and likewise $\sigma_B$.

To understand the situation of Theorem \ref{profinite_invariance_TRT},
we identify $(M_A,\psi_A)$ with a mapping torus of 
an orientation-preserving self-homeomorphism $f_A\colon S_A\to S_A$
and its distinguished cohomology class (see Remark \ref{pA_remark}).
Similarly we identify $(M_B,\phi_B)$ for some homeomorphism $f_B\colon S_B\to S_B$.
The normal subgroup $\Sigma_A=\pi_1(S_A)$ of $\pi_A$ corresponds to 
the normal subgroup $\Sigma_B=\pi_1(S_B)$ of $\pi_B$,
in the sense of Corollary \ref{fiber_surface_correspondence}.
In other words,
the closure of $\Sigma_A$ in $\widehat{\pi}_A$ projects isomorphically onto
the closure of $\Sigma_B$ in $\widehat{\pi}_B$,
under the profinite group isomorphism $\Psi\colon\widehat{\pi_A}\to\widehat{\pi_B}$.
These closures can be identified as $\widehat{\Sigma}_A$ and $\widehat{\Sigma}_B$.
Denote by $\Psi_{\Sigma}\colon\widehat{\Sigma}_A\to \widehat{\Sigma}_B$
the restricted profinite group isomorphism.
Therefore, we obtain the following commutative diagram of groups,
where the rows are all short exact sequences:
\begin{equation}\label{diagram_fibration}
\xymatrix{
1 \ar[r] & \Sigma_A \ar[r] \ar[rd]^-{\mathrm{incl}} 
	& \pi_A \ar[r]^{\psi_A} \ar[rd]^-{\mathrm{incl}} & \Integral \ar[r] \ar[rd]^-{\mathrm{incl}} & 0\\
& 1 \ar[r] & \widehat{\Sigma}_A \ar[r] \ar[d]_{\Psi_{\Sigma}}^-{\cong}
	& \widehat{\pi}_A \ar[r]^{\widehat{\psi}_A} \ar[d]_{\Psi}^-{\cong} & \widehat{\Integral} \ar[r] \ar[d]^{\mu} & 0\\
& 1 \ar[r] & \widehat{\Sigma}_B \ar[r]  
	& \widehat{\pi}_B \ar[r]^{\widehat{\phi}_B}  & \widehat{\Integral} \ar[r] & 0\\
1 \ar[r] & \Sigma_B \ar[r] \ar[ru]_-{\mathrm{incl}} 
	& \pi_B \ar[r]_{\phi_B} \ar[ru]_-{\mathrm{incl}} & \Integral \ar[r] \ar[ru]_-{\mathrm{incl}} & 0\\
}
\end{equation}

Strictly speaking, the self-homeomorphism $f_A$ of $S_A$ only determines
a group automorphism of $\Sigma_A$ up to conjugacy.
This is because we need to modify $f_A$ with some isotopy,
so that it fixes the (implicitly assumed) basepoint of $S_A$.
For convenience,
we fix a choice of the modification, and denote by $f_A\colon \Sigma_A\to\Sigma_A$
the induced group automorphism.
This determines a unique element $t_A\in \pi_A$, such that $f_A(g)=t_A^{-1}gt_A$ holds for all $g\in\Sigma_A$.
We observe that $\psi_A(t_A)=1$ and that $\pi_A=\Sigma_A\rtimes\langle t_A\rangle$ as a semi-direct product.
Similarly, we fix $f_B\colon \Sigma_B\to \Sigma_B$,
and $\pi_B=\Sigma_B\rtimes \langle t_B\rangle$, where $\phi_B(t_B)=1$.
With these notations, the diagram (\ref{diagram_fibration})
is equivalent to saying that $\Psi(t_A)$ is conjugate to the $\mu$--power
$t_B^{\mu}$ of $t_B$ in $\widehat{\pi}_B$, (see (\ref{g_to_nu})).
One may also say that $\Psi_{\Sigma*}(\widehat{f}_A)$
and $\widehat{f}_B^\mu$ are conjugate in the profinite automorphism group $\mathrm{Aut}(\widehat{\pi}_B)$,
(compare \cite[Section 3]{Liu_procongruent_conjugacy}).

The $n$--th twisted homology 
$H^{\sigma_A}_n(\Sigma_A;\Integral^k)$ is a finitely generated $\Integral$--module,
where $\sigma_A$ is considered as the restricted representation $\Sigma_A\to\mathrm{GL}(k,\Integral)$.
Moreover, $f_A$ naturally induces a $\Integral$--linear isomorphism
$f_{A,n}\colon H^{\sigma_A}_n(\Sigma_A;\Integral^k)\to H^{\sigma_A}_n(\Sigma_A;\Integral^k)$,
which actually depends only on the conjugacy class of the group automorphism $f_A\colon \Sigma_A\to\Sigma_A$,
regardless of our chosen modification,
(see \cite[Chapter III, Corollary 8.2]{Brown_book}).
In other words, $f_{A,n}$ depends only on the representation $\sigma_A\colon\pi_A\to\mathrm{GL}(k,\Integral)$
and the primitive fibered class $\psi_A\in H^1(\pi_A;\Integral)$.
We obtain the following commutative diagram of $\Integral$--modules:
\begin{equation}\label{diagram_monodromy}
\xymatrix{
0 \ar[r] & H^{\,\sigma_A}_n\left(\Sigma_A;\Integral^k\right)_{\mathtt{tors}} \ar[r] \ar[d]_-{f_{A,n}^{\mathtt{tors}}}
	& H^{\,\sigma_A}_n\left(\Sigma_A;\Integral^k\right) \ar[r] \ar[d]_-{f_{A,n}} 
	& H^{\,\sigma_A}_n\left(\Sigma_A;\Integral^k\right)_{\mathtt{free}} \ar[r] \ar[d]^-{f_{A,n}^{\mathtt{free}}} & 0\\
0 \ar[r]  & H^{\,\sigma_A}_n\left(\Sigma_A;\Integral^k\right)_{\mathtt{tors}} \ar[r] 
	& H^{\,\sigma_A}_n\left(\Sigma_A;\Integral^k\right) \ar[r]  
	& H^{\,\sigma_A}_n\left(\Sigma_A;\Integral^k\right)_{\mathtt{free}} \ar[r] & 0
}	
\end{equation}
where the rows are canonical short exact sequences.
Note that $f_{A,n}^{\mathtt{free}}$ can be represented as 
an invertible square matrix over $\Integral$,
if we fix a basis of the free $\Integral$--module 
$H^{\sigma_A}_n(\Sigma_A;\Integral^k)_{\mathtt{free}}$.
We introduce the following polynomial in $\Integral[t]$ with leading coefficient $\pm1$:
\begin{equation}\label{P_A_def}
P_{A,n}(t)=\mathrm{det}_{\Integral[t]}\left(\mathbf{1}-t\cdot f_{A,n}^{\mathtt{free}}\right).
\end{equation}
Similarly, we introduce
\begin{equation}\label{P_B_def}
P_{B,n}(t)=\mathrm{det}_{\Integral[t]}\left(\mathbf{1}-t\cdot f_{B,n}^{\mathtt{free}}\right).
\end{equation}

\begin{lemma}\label{Delta_to_P}
	Adopt the notations in Theorem \ref{profinite_invariance_TRT}
	and (\ref{sigma_def}), (\ref{P_A_def}), (\ref{P_B_def}).
	For all $n$,
	the following equalities hold in $\Rational[t^{\pm}]$ up to monomials with coefficients in $\Rational^\times$:	
	\begin{eqnarray*}
	\Delta^{\rho_A,\psi_A}_n(t)&\doteq& P_{A,n}(t);\\
	\Delta^{\rho_B,\phi_B}_n(t)&\doteq& P_{B,n}(t).
	\end{eqnarray*}
	%$$\Delta^{\rho_A,\psi_A}_n(t)\doteq P_{A,n}(t),$$
	%and 
	%$$\Delta^{\rho_B,\phi_B}_n(t)\doteq P_{B,n}(t).$$
\end{lemma}

\begin{proof}
	We prove the equality with $\Delta^{\rho_A,\psi_A}_n(t)$,
	and the other one with $\Delta^{\rho_B,\phi_B}_n(t)$ can be proved similarly.
	By the universal coefficient theorem,
	there are natrual isomorphisms of $\Rational$--modules:
	$$H^{\,\sigma_A}_n\left(\Sigma_A;\Rational^k\right)
	\cong H^{\,\sigma_A}_n\left(\Sigma_A;\Integral^k\right)\otimes_\Integral\Rational
	\cong H^{\,\sigma_A}_n\left(\Sigma_A;\Integral^k\right)_{\mathtt{free}}\otimes_\Integral\Rational,$$
	(see \cite[Chapter V, Theorem 2.5]{Hilton--Stammbach_book}).
	Since $\rho_A=(\sigma_A)_\Rational$,
	we obtain natural isomorphisms of $\Rational[t^{\pm1}]$--modules:
	\begin{equation}\label{TH_Q}
	H^{\rho_A,\psi_A}_n\left(\Sigma_A;\Rational[t^{\pm1}]^k\right)
	\cong H^{\rho_A}_n\left(\Sigma_A;\Rational^k\right)\otimes_\Rational\Rational[t^{\pm1}]
	\cong H^{\,\sigma_A}_n\left(\Sigma_A;\Integral^k\right)_{\mathtt{free}}\otimes_\Integral\Rational[t^{\pm1}].
	\end{equation}
	Here, the left-hand side refers to the twisted homology 
	with respect to the restricted representation $t^{\psi_A}\rho_A\colon \Sigma\to\mathrm{GL}(k,\Rational[t^{\pm1}])$,
	such that $g\mapsto t^{\psi(g)}\cdot\rho_A(g)$.
	So the the first isomorphism in (\ref{TH_Q}) follows 
	since $\Sigma_A$ is the kernel of $\psi_A\colon\pi_A\to \Integral$.
	We obtain the Wang exact sequence of $\Rational[t^{\pm1}]$--modules:
	\begin{equation}\label{Wang_LES}
	\xymatrix{
	& & \cdots \ar[r] & H^{\rho_A,\psi_A}_{n+1}\left(\pi_A;\Rational[t^{\pm1}]^k\right)   \ar[llld]\\
	H^{\rho_A,\psi_A}_n\left(\Sigma_A;\Rational[t^{\pm1}]^k\right) \ar[rr]_-{1-t f^{\Rational}_{A,n}} & &
	H^{\rho_A,\psi_A}_n\left(\Sigma_A;\Rational[t^{\pm1}]^k\right) \ar[r] & 
	H^{\rho_A,\psi_A}_n\left(\pi_A;\Rational[t^{\pm1}]^k\right) \ar[llld]  \\
	H^{\rho_A,\psi_A}_{n-1}\left(\Sigma_A;\Rational[t^{\pm1}]^k\right) \ar[rr]_-{1-t f^{\Rational}_{A,n-1}} & &
	H^{\rho_A,\psi_A}_{n-1}\left(\Sigma_A;\Rational[t^{\pm1}]^k\right) \ar[r] & \cdots
	}
	\end{equation}
	where $f^{\Rational}_{A,n}$ denotes the scalar extension of $f^{\mathtt{free}}_{A,n}$ over $\Rational[t^{\pm1}]$,
	in view of (\ref{TH_Q}).
	In fact, this long exact sequence is a special case of 
	\cite[Chapter VII, Section 9, (9.4)]{Brown_book}.
	Note that 
	$H^{\rho_A,\psi_A}_n\left(\Sigma_A;\Rational[t^{\pm1}]^k\right)$ 
	is a free $\Rational[t^{\pm1}]$--module of finite rank,
	and the $\Rational[t^{\pm1}]$--endomorphism $1-t f^{\Rational}_{A,n}$ is injective,
	since it has determinant $P_{A,n}(t)\neq0$ in $\Rational[t^{\pm1}]$.
	It follows that the rows in (\ref{Wang_LES}) are presentations
	of the $\Rational[t^{\pm1}]$--modules
	$H^{\rho_A,\psi_A}_n(\pi_A;\Rational[t^{\pm1}]^k)$, for all $n$.
	Since the presentation matrix $1-t f^{\Rational}_{A,n}$ is a square matrix,
	its determinant is by definition the order of
	$H^{\rho_A,\psi_A}_n(\pi_A;\Rational[t^{\pm1}]^k)$, (see Remark \ref{order_remark}).
	Therefore, 
	the $n$--th twisted Alexander polynomial $\Delta^{\rho_A,\phi_A}_n(t)$ is by definition
	equal to $P_{A,n}(t)$ in $\Rational[t^{\pm1}]$
	up to monomial factors with coefficients in $\Rational^\times$,
	as asserted.
\end{proof}

\begin{lemma}\label{P_ideal_equal}
	Adopt the notations in Theorem \ref{profinite_invariance_TRT}
	and (\ref{sigma_def}), (\ref{P_A_def}), (\ref{P_B_def}).
	For all $n$,
	the following equality of principal ideals
	holds in $\llbracket\widehat{\Integral}t^{\widehat{\Integral}}\rrbracket$:
	\begin{eqnarray*}
	\left(P_{A,n}(t^\mu)\right)&=& \left(P_{B,n}(t)\right).
	\end{eqnarray*}	
\end{lemma}

\begin{proof}
	For any $N\in\Natural$ divisible by 
	the order of $H^{\sigma_A}_{n-1}(\Sigma_A;\Integral^k)_{\mathtt{tors}}$, 
	the universal coefficient theorem implies
	natural isomorphisms of finite $\Integral$--modules:
	\begin{equation*}\label{UCT_mod_N}
	\xymatrix{
	0 \ar[r] &
	H^{\,\sigma_A}_n\left(\Sigma_A;\Integral^k\right)\otimes_\Integral(\Integral/N\Integral)
	\ar[r] &
	H^{\,\sigma_{A,N}}_n\left(\Sigma_A;(\Integral/N\Integral)^k\right)
	\ar[r] &
	H^{\,\sigma_A}_{n-1}\left(\Sigma_A;\Integral^k\right)_{\mathtt{tors}}
	\ar[r] & 0,
	}
	\end{equation*}
	where $\sigma_{A,N}\colon \Sigma_A\to\mathrm{GL}(k,\Integral/N\Integral)$
	denotes the reduction modulo $N$ of $\sigma_A$ restricted to $\Sigma_A$,
	and the short exact sequence splits naturally with respect to 
	the variable $N$, 
	(see \cite[Chapter V, Theorem 2.5]{Hilton--Stammbach_book}).
	So we obtain
	\begin{equation}\label{UCT_mod_N_splitting}
	H^{\,\sigma_{A,N}}_n\left(\Sigma_A;(\Integral/N\Integral)^k\right)
	\cong
	H^{\,\sigma_{A}}_n\left(\Sigma_A;\Integral^k\right)\otimes_\Integral(\Integral/N\Integral)
	\,\oplus\,
	H^{\,\sigma_A}_{n-1}\left(\Sigma_A;\Integral^k\right)_{\mathtt{tors}},
	\end{equation}
	and hence 
	\begin{equation}\label{UCT_profinite_splitting}
	H^{\,\widehat{\sigma}_{A}}_n\left(\widehat{\Sigma}_A;\widehat{\Integral}^k\right)
	\cong
	H^{\,\sigma_{A}}_n\left(\Sigma_A;\Integral^k\right)\otimes_\Integral\widehat{\Integral}
	\,\oplus\,
	H^{\,\sigma_A}_{n-1}\left(\Sigma_A;\Integral^k\right)_{\mathtt{tors}},	
	\end{equation}
	as an isomorphism of profinite $\widehat{\Integral}$--modules.
	Note the left-hand side of (\ref{UCT_profinite_splitting}) can be identified with
	the inverse limit of the left-hand side of (\ref{UCT_mod_N_splitting}),
	since $\Sigma_A$ is cohomologically good,
	(see also \cite[Chapter 6, Corollary 6.1.10(a) and Section 6.3]{Ribes--Zalesskii_book}).
	Therefore,
	we obtain a canonical $\widehat{\Integral}$--submodule
	$H^{\widehat{\sigma}_A}_n(\widehat{\Sigma}_A;\widehat{\Integral}^k)$,
	consisting of all the $\Integral$--torsion elements,
	and a (non-canonical) isomorphism
	$$H^{\,\widehat{\sigma}_A}_n\left(\widehat{\Sigma}_A;\widehat{\Integral}^k\right)_{\mathtt{tors}}
	\cong
	H^{\,\sigma_A}_n\left(\Sigma_A;\Integral^k\right)_{\mathtt{tors}}\oplus
	H^{\,\sigma_A}_{n-1}\left(\Sigma_A;\Integral^k\right)_{\mathtt{tors}}.
	$$
	Moreover, factoring this out, 
	we obtain a canonical quotient free $\widehat{\Integral}$--module
	and a canonical isomorphism
	$$H^{\,\widehat{\sigma}_A}_n\left(\widehat{\Sigma}_A;\widehat{\Integral}^k\right)_{\mathtt{free}}
	\cong
	H^{\,\sigma_A}_n\left(\Sigma_A;\Integral^k\right)_{\mathtt{free}}\otimes_\Integral\widehat{\Integral}.
	$$
	The scalar extension of the $\Integral$--module automorphism $f^{\mathtt{free}}_{A,n}$ 
	over $\widehat{\Integral}$ thereby yields a $\widehat{\Integral}$--module automorphism
	of $H^{\widehat{\sigma}_A}_n(\widehat{\Sigma}_A;\widehat{\Integral}^k)_{\mathtt{free}}$,
	which we denote as $F_{A,n}$.
	Similarly, 
	we obtain $H^{\widehat{\sigma}_B}_n(\widehat{\Sigma}_B;\widehat{\Integral}^k)_{\mathtt{free}}$
	and its automorphism $F_{B,n}$.
	
	By construction, $F_{A,n}$ coincides with the automorphism 
	induced from the completion automorphism $\widehat{f}_{A,n}$
	on $H^{\widehat{\sigma}_A}_n(\widehat{\Sigma}_A;\widehat{\Integral}^k)$,
	which depends only on the conjugacy class of $\widehat{f}_A$ on $\widehat{\Sigma}_A$,
	(see \cite[Chapter 6, Proposition 6.7.12]{Ribes--Zalesskii_book}).
	Then the diagram \ref{diagram_fibration} implies the commutative diagram of free $\widehat{\Integral}$--modules:	
	\begin{equation}\label{FFmu}
	\xymatrix{
	H^{\,\widehat{\sigma}_A}_n\left(\widehat{\Sigma}_A;\widehat{\Integral}^k\right)_{\mathtt{free}}
	\ar[r]^-{F_{A,n}} \ar[d]_{\Psi_{\Sigma*}} &
	H^{\,\widehat{\sigma}_A}_n\left(\widehat{\Sigma}_A;\widehat{\Integral}^k\right)_{\mathtt{free}}
	\ar[d]^{\Psi_{\Sigma*}}\\
	H^{\,\widehat{\sigma}_B}_n\left(\widehat{\Sigma}_B;\widehat{\Integral}^k\right)_{\mathtt{free}}
	\ar[r]^-{F_{B,n}^\mu} &
	H^{\,\widehat{\sigma}_B}_n\left(\widehat{\Sigma}_B;\widehat{\Integral}^k\right)_{\mathtt{free}}
	}
	\end{equation}
	where $F_{B,n}^\mu$ is the $\mu$--power of $F_{B,n}$, 
	(not to be confused with our notation of $\varepsilon$--specializations;
	see (\ref{g_to_nu})).
	
	If one fixes a basis of $H^{\sigma_A}_n(\Sigma_A;\Integral^k)_{\mathtt{free}}$,
	$F_{A,n}$ becomes an invertible square matrix over $\widehat{\Integral}$,
	which is the same as the matrix of $f^{\mathtt{free}}_{A,n}$.	
	Similarly, $F_{B,n}$ becomes a square matrix the same as that of $f^{\mathtt{free}}_{B,n}$.
	The diagram (\ref{FFmu}) implies that 
	$\Psi_{\Sigma*}\circ F_{A,n}\circ \Psi_{\Sigma*}^{-1}$ 
	is conjugate to $F_{B,n}^\mu$ as square matrices over $\widehat{\Integral}$.
	Then Lemma \ref{Ueki_lemma_matrix} applies,
	implying $(P_A(t^\mu))=(P_B(t))$ in $\llbracket\widehat{\Integral}t^{\widehat{\Integral}}\rrbracket$,
	as asserted.	
\end{proof}

By Lemmas \ref{Delta_to_P} and \ref{P_ideal_equal} and Theorem \ref{Ueki_lemma},
we obtain the equality in $\Rational[t^{\pm1}]$ up to monomial factors with coefficients in $\Rational^\times$:
\begin{equation}\label{Delta_reciprocal_equal}
	\Delta^{\rho_A,\psi_A}_n(t)\cdot\Delta^{\rho_A,\psi_A}_n(t^{-1})\doteq
	\Delta^{\rho_B,\phi_B}_n(t)\cdot\Delta^{\rho_B,\phi_B}_n(t^{-1}),
\end{equation}
for all $n$.
By (\ref{TRT_def}), in $\Rational(t)$ up to monomial factors with coefficients in $\Rational^\times$:
\begin{equation}\label{tau_reciprocal_equal}
	\tau^{\rho_A,\psi_A}(t)\cdot\tau^{\rho_A,\psi_A}(t^{-1})\doteq
	\tau^{\rho_B,\phi_B}(t)\cdot\tau^{\rho_B,\phi_B}(t^{-1}).
\end{equation}

By Theorem \ref{TAP_properties} and (\ref{TRT_def}),
we obtain 
$\tau^{\rho_A,\psi_A}(t^{-1})\doteq\tau^{\bar{\rho}_A,\psi_A}(t)$,
where $\bar{\rho}_A$ is the $\gamma_A$--pullback of $\bar\rho\colon \Gamma\to\mathrm{GL}(k,\Rational)$,
such that $\bar{\rho}(g)$ is the transpose of $\rho(g)^{-1}$.
However, since $\Gamma$ is finite,
$\rho$ is conjugate to a representation $\Gamma\to \mathrm{O}(k)$ over $\Real$,
which is fixed under the bar involution.
It follows that $\rho$ is conjugate to $\bar{\rho}$ over $\Rational$.
This implies
$\tau^{\bar{\rho}_A,\psi_A}(t)\doteq\tau^{\rho_A,\psi_A}(t)$,
since twisted Reidemeister torsions are invariant under conjugation of the representation,
(see \cite[Section 3.3.2, Lemma 1]{Friedl--Vidussi_survey}).
Similar manipulation works for $\tau^{\rho_B,\phi_B}$.
Therefore, (\ref{tau_reciprocal_equal}) implies
that the square of $\tau^{\rho_A,\psi_A}(t)$ equals the square of $\tau^{\rho_B,\phi_B}(t)$,
up to monomial factors with coefficient in $\Rational^\times$.
By unique factorization in $\Rational(t)$, we obtain
$$\tau^{\rho_A,\psi_A}(t)\doteq\tau^{\rho_B,\phi_B}(t).$$

This completes the proof of Theorem \ref{profinite_invariance_TRT}.

\section{Profinite correspondence of Nielsen numbers}\label{Sec-Nielsen}
%In this section, we show that indexed orbit numbers of mapping classes
%are procongruent conjugacy invariants.

Suppose that $M$ is an orientable connected closed $3$--manifold that admits a hyperbolic metric.
Suppose that $\phi\in H^1(M;\Integral)$ is a primitive fibered class.
We identify $(M,\phi)$ with 
the mapping torus of some pseudo-Anosov automorphism
$f\colon S\to S$ and its distinguished cohomology class,
namely, $(M_f,\phi_f)$, (see Remark \ref{pA_remark}).
For all $m\in\Natural$ and $i\in\Integral\setminus\{0\}$,
we denote the number of index--$i$ $m$--periodic orbits of $f$ as
\begin{equation}\label{nu_m_i_def}
\nu_m(M,\phi;i)=\nu_m(f;i)=\#\left\{\mathbf{O}\in\mathrm{Orb}_m(f)\colon \mathrm{ind}_m(f;\mathbf{O})=i\right\},
\end{equation}
and denote the number of (essential) $m$--periodic orbits of $f$ as
\begin{equation}\label{N_m_def}
N_m(M,\phi)=N_m(f)=\sum_{i\in\Integral\setminus\{0\}} \nu_m(f;i)=\#\,\mathrm{Orb}_m(f),
\end{equation}
(see (\ref{ind_m})).
The number $N_m(f)$ is called the \emph{$m$--th orbit Nielsen number} of $f$.
It is known that the orbit Nielsen numbers growth exponentially fast as $m$ tends to $\infty$,
and they determine the stretch factor $\lambda(f)$ as
\begin{equation}\label{N_m_and_lambda}
\lambda(f)=\limsup_{m\to\infty} N_m(f)^{1/m},
\end{equation}
(see \cite[Section 2, Example 2]{Jiang_periodic}).

\begin{theorem}\label{profinite_invariance_nu}
	Let $(M_A,M_B,\Psi)$ be a profinite morphism setting of a $3$--manifold pair (Convention \ref{profinite_morphism_setting}).
	Suppose that $\Psi\colon\widehat{\pi_A}\to\widehat{\pi_B}$ is an isomorphism.
	Suppose that $M_A$ and $M_B$ are both closed, fibered, and hyperbolic.
	Let $\mu\in\widehat{\Integral}^\times$ be a unit as provided in Theorem \ref{profinite_isomorphism_hyperbolic}.
	
	Let $\phi_B\in H^1(M_B;\Integral)$ be a primitive fibered class.
	Let $\psi_{A}\in H^1(M_A;\Integral)$
	be the primitive fibered class $\Psi^*_{1/\mu}(\phi_B)$, as guaranteed by Theorem \ref{profinite_isomorphism_npc}.
	Then, the following equality holds for all $m\in\Natural$ and $i\in\Integral\setminus\{0\}$:
	$$\nu_m\left(M_A,\psi_A;i\right)=\nu_m\left(M_B,\phi_B;i\right).$$
	Hence, for all $m\in\Natural$,
	$$N_m\left(M_A,\psi_A\right)=N_m\left(M_B,\phi_B\right).$$
\end{theorem}

%\begin{corollary}\label{profinite_correspondence_N}
	%Under the assumptions of Theorem \ref{profinite_invariance_nu},
	%$$N_m\left(M_A,\phi_{A,1/\mu}\right)=N_m\left(M_B,\phi_B\right).$$
%\end{corollary}
%
%\begin{proof}
%\end{proof}

\begin{corollary}\label{pA_features_correspondence}
	Under the assumptions of Theorem \ref{profinite_invariance_nu},
	denote by $f_A\colon S_A\to S_A$ and $f_B\colon S_B\to S_B$
	the pseudo-Anosov automorphisms of the fiber surfaces
	as determined by $\psi_A$ and $\phi_B$, respectively.
	Then the following characteristic features are all identical for $f_A$ and $f_B$:
	\begin{itemize}
	\item The stretch factor
	\item The number of index--$i$ fixed points, $i\neq0$.
	\item The number of $k$--prong singular points in the invariant stable/unstable foliation, $k\geq3$
	\item The transverse orientability of the invariant stable/unstable foliation
	\end{itemize}
\end{corollary}

\begin{remark}\label{pA_features_correspondence_remark}
	The similar conclusion is proved in \cite[Section 11]{Liu_procongruent_conjugacy},
	when $f_A,f_B\colon S\to S$ 
	are pseudo-Anosov automorphisms whose induced profinite outer automorphisms
	$[f_A],[f_B]\in\mathrm{Out}(\widehat{\pi_1(S)})$ are conjugate.
	With Theorem \ref{profinite_invariance_nu},
	the former proof of \cite[Theorem 11.1]{Liu_procongruent_conjugacy}
	works for Corollary \ref{pA_features_correspondence_remark} without change.
	In fact, the listed features are all determined by
	the collection of all the indexed periodic Nielsen numbers.
\end{remark}

The rest of this section is devoted to the proof of Theorem \ref{profinite_invariance_nu}.

We need some representation theory of finite groups over the rational field.
The reader may consult Serre's textbook \cite{Serre_book_representation} for general reference.
Let $\Gamma$ be a finite group of order $|\Gamma|$. 
The power operation gives rise to
a set-theoretic permutation action of $\widehat{\Integral}^\times$ on $\Gamma$,
namely, $\widehat{\Integral}^\times\times \Gamma\to \Gamma\colon
(\nu,g)\mapsto g^\nu$, 
and $g^\nu$ stands for the element $g^n\in \Gamma$ 
where $n\in\Integral$ is congruent to $\nu$ modulo $|\Gamma|$,
(compare (\ref{g_to_nu})).
With this action in mind,
we say that $g,g'\in\Gamma$ are \emph{$\widehat{\Integral}^\times$--conjugate},
if there exists some $\nu\in\widehat{\Integral}^\times$ such that $g^\nu$ and $g'$ are conjugate in $\Gamma$.
Being $\widehat{\Integral}^\times$--conjugate is clearly an equivalence relation on $\Gamma$,
and every equivalence class is the union of finitely many conjugacy classes of $\Gamma$.
Therefore, we obtain an equivalence relation on the set of conjugacy classes $\mathrm{Orb}(\Gamma)$,
which we call \emph{$\widehat{\Integral}^\times$--equivalence},
such that $\mathbf{c},\mathbf{c}'\in\mathrm{Orb}(\Gamma)$ are $\widehat{\Integral}^\times$--equivalent 
if and only if their representatives are $\widehat{\Integral}^\times$--conjugate in $\Gamma$.
We introduce the notation
\begin{equation}\label{Omega_def}
\Omega(\Gamma)=\left\{\widehat{\Integral}^\times\mbox{--equivalence classes of }\mathrm{Orb}(\Gamma)\right\},
\end{equation}
so $\mathrm{Orb}(\Gamma)$ is the disjoint union of all $\omega\in\Omega(\Gamma)$.
For any $\omega\in\Omega(\Gamma)$, denote by $\chi_\omega\colon \mathrm{Orb}(\Gamma)\to\Rational$
the characteristic function of $\omega$, namely,
\begin{equation}\label{chi_omega}
\chi_\omega(\mathbf{c})=\begin{cases} 1 & \mathbf{c}\in\omega \\ 0 & \mbox{otherwise}\end{cases}
\end{equation}

\begin{lemma}\label{Q_rep_finite_group}\
	Let $\Gamma$ be a finite group.
	A function $\mathrm{Orb}(\Gamma)\to \Rational$
	is constant on every $\widehat{\Integral}^\times$--equivalence class
	$\omega\in\Omega(\Gamma)$
	if and only if
	it is a	$\Rational$--linear combinations of the $\Rational$--irreducible characters of $\Gamma$.
\end{lemma}

\begin{proof}
	This follows from \cite[Section 12.4, Corollary 1]{Serre_book_representation}.
	To adapt with notations thereof, 
	we rewrite the finite group as $G$, and its order as $m$.
	Let $K$ be the rational field $\Rational$,	
	and $L$ be the cyclotomic extension of $K$ joining a primitive $m$--th root of unity.
	Since the Galois group $\Gamma_K=\mathrm{Gal}(L/K)$
	is isomorphic to $(\Integral/m\Integral)^\times$,
	the `$\Gamma_K$--conjugacy' relation on $G$ in Serre's notations
	agrees with our $\widehat{\Integral}^\times$--conjugacy relation.
	Then \cite[Section 12.4, Corollary 1]{Serre_book_representation} says
	that a conjugation-invariant function $G\to K$ is a $K$--linear combination
	of $K$--irreducible characters if and only if it is constant on 
	every $\Gamma_K$--conjugacy class of $G$. 
	Lemma \ref{Q_rep_finite_group} says exactly the same thing 
	on the level of $\mathrm{Orb}(G)\to K$.
	(See also \cite[Section 13.1]{Serre_book_representation} for further discussions about
	linear representations of finite groups over $\Rational$.)
\end{proof}

Let $M_f$ be the mapping torus of a pseudo-Anosov automorphism
$f\colon S\to S$. % (see Remark \ref{pA_remark}).
For any function $\xi\colon \mathrm{Orb}(\pi_1(M_f))\to\Rational$,
we introduce an analogous twisted Lefschetz number as the following value in $\Rational$:
\begin{equation}\label{L_m_omega}
L_m(f;\xi)=
\sum_{\mathbf{O}\in\mathrm{Orb}(\pi_1(M_f))} \xi(\ell_m(f;\mathbf{O}))\cdot\mathrm{ind}_m(f;\mathbf{O}),
\end{equation}
(compare (\ref{L_m_rho})).

\begin{lemma}\label{N_m_inequality}
	Let $M_f$ be the mapping torus of a pseudo-Anosov automorphism
	$f\colon S\to S$, (see Remark \ref{pA_remark}).
	Then, for any finite quotient $\gamma\colon \pi_1(M_f)\to \Gamma$,
	the following estimate holds for all $m\in\Natural$:
	$$N_m(f)\geq\#\left\{\omega\in\Omega(\Gamma)\colon L_m(f;\gamma^*\chi_\omega)\neq0\right\},$$
	where $\gamma^*\chi_\omega\colon \Omega(\pi_1(M_f))\to \Rational$ denotes the pullback of 
	the characteristic function (\ref{chi_omega}).
	Moreover, assume that the equality is achieved for some given $\gamma$ and $m$, then the following equalities
	hold for all $i\in\Integral$:
	$$\nu_m(f,i)=\#\left\{\omega\in\Omega(\Gamma)\colon L_m(f;\gamma^*\chi_\omega)=i\right\}.$$
\end{lemma}

\begin{proof}
	The idea is similar as with \cite[Lemma 9.3]{Liu_procongruent_conjugacy}.
	We say that an $m$--periodic orbit $\mathbf{O}\in\mathrm{Orb}_m(f)$ 
	\emph{hits} a $\widehat{\Integral}^\times$--equivalence class $\omega\in\Omega(\Gamma)$,
	if $\gamma_*\colon\mathrm{Orb}(\pi_1(M_f))\to \mathrm{Orb}(\Gamma)$ 
	maps $\ell_m(f;\mathbf{O})$ into the subset $\omega$.
	
	The asserted inequality about $N_m(f)$ follows from 
	the observation that $L_m(f;\gamma^*\chi_{\omega})\neq0$ occurs
	only if at least one $m$--periodic orbit hits $\omega$.
	Indeed, the right-hand side is bounded by
	the total number of $m$--periodic orbits, namely, $N_m(f)$.
	
	Moreover, when the equality is achieved, 
	every $\mathbf{O}\in\mathrm{Orb}_m(f)$
	must hit a distinct $\omega\in\Omega(\Gamma)$.
	So $L_m(f;\gamma^*\chi_{\omega})=\mathrm{ind}_m(f;\mathbf{O})$ holds
	for each hitting pair $(\mathbf{O},\omega)$.
	This yields the asserted equality about $\nu_m(f;i)$.
\end{proof}

\begin{lemma}\label{N_m_equality}
	Let $M_f$ be the mapping torus of a pseudo-Anosov automorphism
	$f\colon S\to S$, (see Remark \ref{pA_remark}).
	Then, for any $m\in\Natural$, 
	there exists a finite quotient $\gamma\colon \pi_1(M_f)\to \Gamma_m$,
	such that the following equality holds 
	$$N_m(f)=\#\left\{\omega\in\Omega(\Gamma_m)\colon L_m(f;\gamma^*\chi_\omega)\neq0\right\}.$$
\end{lemma}

\begin{proof}
	We invoke the fact
	that the fundamental group of any orientable connected compact $3$--manifold
	is conjugacy separable \cite[Theorem 1.3]{HWZ_conjugacy_separability},
	(see also \cite{Minasyan_RAAG}).
	In particular, this means that there exists some finite quotient
	$\pi_1(M_f)\to \tilde{\Gamma}_m$, such that for all $m$--periodic orbits $\mathbf{O}\in\mathrm{Orb}_m(f)$,
	the conjugacy classes $\ell_m(f;\mathbf{O})\in\mathrm{Orb}(\pi_1(M_f))$
	project onto mutually distinct conjugacy classes in $\mathrm{Orb}(\tilde{\Gamma}_m)$.
	Note that if $\Gamma_m$ is any finer finite quotient,
	through which $\pi_1(M_f)\to \tilde{\Gamma}_m$ factors,
	the same property also holds for $\pi_1(M_f)\to \Gamma_m$.
	Below we construct $\Gamma_m$ more carefully,
	so that all the $\ell_m(f;\mathbf{O})$ actually hit
	distinct $\widehat{\Integral}^\times$--equivalence classes in $\Omega(\Gamma_m)$.
	Then the asserted equality holds in this case,
	because of (\ref{N_m_def}) and (\ref{L_m_omega}).
	
	We construct the asserted finite quotient $\pi_1(M_f)\to \Gamma_m$ as follows.
	Let $\tilde{\pi}$ be the kernel of $\pi_1(M_f)\to \tilde{\Gamma}_m$. 
	Let $K_m$ be a finite-index characteristic subgroup
	of $\pi_1(S)$ which is contained in $\pi_1(S)\cap \tilde{\pi}$.
	For example, one may take $K_m$ as the intersection of all the subgroups of $\pi_1(S)$
	of index $[\pi_1(S):\pi_1(S)\cap \tilde{\pi}]$. 
	Then $K_m$ is a normal subgroup of $\pi_1(M_f)$.
	Fixing some $t\in\pi_1(M_f)$ with $\phi_f(t)=1$,
	we write $\pi_1(M_f)$ as the semi-direct product $\pi_1(S)\rtimes\langle t\rangle$.	
	Denote by $a=|\mathrm{Aut}(\pi_1(S)/K_m)|$ the order of the automorphism group.
	We observe that $t^a$ projects into the center of the quotient group
	$\pi_1(M_f)/K_m=(\pi_1(S)/K_m)\rtimes\langle t\rangle$.
	Denote by $c=|\mathrm{Z}(\pi_1(S)/K_m)|$ be the order of the center
	of $\pi_1(S)/K_m$.
	Let $d$ be some positive integral multiple of $ac$,
	such that $t^d$ lies in $\tilde{\pi}$.
	Therefore,
	we construct $\Gamma_m$	as the quotient of $\pi_1(M_f)$ 
	by the finite-index normal subgroup $K_m\rtimes \langle t^{dm}\rangle$.
	In particular, $\pi_1(M_f)\to\tilde{\Gamma}_m$ factors through $\Gamma_m$.	
	Moreover, 
	we obtain the following commutative diagram of groups
	\begin{equation}\label{Gamma_m_diagram}
	\xymatrix{
	1 \ar[r] & \pi_1(S) \ar[r] \ar[d]
		& \pi_1(M_f) \ar[r]^-{\phi_f} \ar[d] & \Integral \ar[r] \ar[d] & 1 \\
	1 \ar[r] & \pi_1(S)/K_m \ar[r] & \Gamma_m \ar[r] & \Integral/dm\Integral \ar[r] & 1
	}
	\end{equation}
	where the rows are short exact sequences and 
	the vertical arrows are quotient homomorphisms.

	For any $g\in\pi_1(M_f)$, we denote by $\bar{g}$ its image in $\Gamma_m$.
	We observe that if $\phi_f(g)=m$, then $\bar{g}$ has order $d$ in $\Gamma_m$.
	In fact, 
	the order of $\bar{g}$ is at least $d$ since it projects onto
	$m \bmod dm$ in $\Integral/dm\Integral$.
	On the other hand,	
	$\bar{g}^a$ can be written as $\bar{t}^{\,am}\bar{z}$ for some $\bar{z}$
	in the center of $\pi_1(S)/K_m$;
	and since $\bar{t}^a$ is central in $\Gamma_m$, 
	we obtain $\bar{g}^{d}=\bar{t}^{dm}\bar{z}^{d/a}$,
	which is trivial in $\Gamma_m$. 
	So the order of $\bar{g}$ is exactly $d$.
	
	Let $\mathbf{O},\mathbf{O}'\in\mathrm{Orb}_m(f)$ be a pair of $m$--periodic orbits.
	Let $g\in\pi_1(M_f)$ be a representative of the conjugacy class $\ell_m(f;\mathbf{O})$,
	and similarly a representative $g'$ of $\ell_m(f;\mathbf{O}')$.
	Suppose that $\bar{g}^\mu$ is conjugate to $\bar{g}'$ in $\Gamma_m$,
	for some $\mu\in\widehat{\Integral}^\times$.
	This is equivalent to saying that the cyclic subgroups $\langle \bar{g}\rangle$
	and $\langle \bar{g}'\rangle$ of $\Gamma_m$ are conjugate to each other.
	However, since $\phi_f(g)=m$, the cyclic subgroup $\langle \bar{g}\rangle$ has order $d$,
	and it intersects the coset $(\pi_1(S)/K_m)\bar{t}^m$ of $\Gamma_m$ at the unique element $\bar{g}$.
	Similarly, $\langle \bar{g}'\rangle$ intersects $(\pi_1(S)/K_m)\bar{t}^m$ at $\bar{g}'$.
	Since $(\pi_1(S)/K_m)\bar{t}^m$ is invariant under conjugations of $\Gamma_m$,
	it follows that $\bar{g}$ is conjugate to $\bar{g}'$ in $\Gamma_m$.
	This implies $\mathbf{O}=\mathbf{O}'$,
	since $\Gamma_m$ is a finer quotient than $\tilde{\Gamma}_m$.
\end{proof}

We prove Theorem \ref{profinite_invariance_nu} as follows.
Let $(M_A,M_B,\Psi)$ be a profinite morphism setting (Convention \ref{profinite_morphism_setting})
where $\Psi\colon \widehat{\pi}_A\to \widehat{\pi}_B$ is an isomorphism
and where $M_A$ and $M_B$ are both closed and hyperbolic.
Let $\mu\in\widehat{\Integral}^\times$ be a unit as provided in Theorem \ref{profinite_isomorphism_hyperbolic}.

For any primitive fibered class $\phi_B$ in $H^1(M_B;\Integral)$,
the cohomology class $\psi_A=\Psi^*_{1/\mu}(\phi_B)$ in $H^1(M_A;\Integral)$
is also primitive and fibered (Theorem \ref{profinite_isomorphism_npc}).
Moreover, we can identify $(M_A,\psi_A)$ with the mapping torus and its distinguished cohomology class
for some pseudo-Anosov automorphism $f_A\colon S_A\to S_A$,
and identify $(M_B,\phi_B)$ similarly for some $f_B\colon S_B\to S_B$.
Then $\widehat{\pi_1(S_A)}$ projects isomorphically onto $\widehat{S_B}$ under $\Psi$, 
(Corollary \ref{fiber_surface_correspondence}).
For any $\Psi$--corresponding pair of finite quotients
$\gamma_A\colon \pi_A\to \Gamma$ and $\gamma_B\colon \pi_B\to \Gamma$,
and for any finite-dimensional representation $\rho\colon \Gamma\to\mathrm{GL}(k,\Rational)$,
we obtain the following equality for all $m\in\Natural$:
$$L_m(f_A;\gamma_A^*\chi_\rho)=L_m(f_B;\gamma_B^*\chi_\rho),$$
by Theorems \ref{tau_zeta}, \ref{profinite_invariance_TRT} and (\ref{L_m_rho}).
Suitable $\Rational$--linear combinations of the above equalties over all the $\Rational$--irreducible $\rho$
yield the equality
$$L_m(f_A;\gamma_A^*\chi_\omega)=L_m(f_B;\gamma_B^*\chi_\omega),$$
for all $m\in\Natural$ and all $\omega\in\Omega(\Gamma)$,
by (\ref{chi_omega}), (\ref{L_m_omega}) and Lemma \ref{Q_rep_finite_group}.
Applying Lemmas \ref{N_m_equality} and \ref{N_m_inequality}, taking $\Gamma_m$ with respect to $f_A$,
we obtain a comparison
$N_m(f_A)=\#\left\{\omega\in\Omega(\Gamma)\colon L_m(f_A;\gamma_A^*\chi_\omega)\neq0\right\}
= \#\left\{\omega\in\Omega(\Gamma)\colon L_m(f_B;\gamma_B^*\chi_\omega)\neq0\right\}
\leq N_m(f_B)$.
Similarly, 
we obtain $N_m(f_B)\leq N_m(f_A)$ by taking $\Gamma_m$ with respect to $f_B$.
Therefore, we obtain the equality for the orbit Nielsen numbers
$$N_m(M_A,\psi_A)=N_m(f_A)=N_m(f_B)=N_m(M_B,\phi_B)$$
for all $m\in\Natural$, as asserted.
Since this forces the equality to be achieved in the above comparisons,
we obtain the equality for the indexed orbit Nielsen numbers
$$\nu_m(M_A,\psi_A;i)=\nu_m(M_B,\phi_B;i)$$
for all $m\in\Natural$ and $i\in\Integral\setminus\{0\}$,
as asserted, by Lemma \ref{N_m_inequality}.

This completes the proof of Theorem \ref{profinite_invariance_nu}.

\section{Profinite almost rigidity}\label{Sec-profinite_almost_rigidity}

In this section, we prove Theorem \ref{main_profinite_almost_rigidity},
and indeed, we prove a slightly more topological statement as follows.

\begin{theorem}\label{profinite_almost_rigidity_hyperbolic}
	For any finite-volume hyperbolic $3$--manifold $M$,
	there exists a finite collection $\mathcal{M}$ of finite-volume hyperbolic $3$--manifolds,
	such that the following property holds:
	
	If $N$ is a connected compact $3$--manifold and if the profinite completion $\widehat{\pi_1(N)}$
	is isomorphic to $\widehat{\pi_1(M)}$, then,
	possibly after capping off all boundary spheres with $3$--balls,
	the interior of $N$ is homeomorphic to some member from $\mathcal{M}$.
\end{theorem}

We prove Theorem \ref{profinite_almost_rigidity_hyperbolic} in the rest of this section.
Our argument relies eventually on the following well-known finiteness result in mapping class group theory,
due to Arnoux--Yoccoz \cite{Arnoux--Yoccoz_pA} and Ivanov \cite{Ivanov_pA},
(see also \cite[Chapter 14, Theorem 14.9]{Farb--Margalit_book}).

\begin{theorem}\label{pA_finiteness}
	For any integers $g,p\geq0$ with $2-2g-p<0$, and for any real constant $C>1$,
	there are at most finitely many pseudo-Anosov automorphisms
	$f\colon S\to S$ of an orientable surface $S$ 
	with genus $g$ and $p$ punctures,	
	up to topological equivalence,
	such that the stretch factor $\lambda(f)$ is at most $C$.
\end{theorem}

We need the punctured case of Theorem \ref{pA_finiteness}
to prove the cusped case of Theorem \ref{profinite_almost_rigidity_hyperbolic}.
When $S$ has punctures, a \emph{pseudo-Anosov} automorphism $f\colon S\to S$
is similarly defined as in Remark \ref{pA_remark},
preserving a transverse pair of measured foliations $(\mathscr{F}^\mathtt{s},\mu^{\mathtt{s}})$
and $(\mathscr{F}^\mathtt{u},\mu^\mathtt{u})$, shrinking or stretching the measure with 
reciprocal factors, and we allow $k$--prong singularities in the interior with $k\geq3$,
or at the punctures with $k\geq1$.
A pair of pseudo-Anosov automorphisms $f_A\colon S_A\to S_A$ and $f_B\colon S_B\to S_B$
are said to be \emph{topologically equivalent} if there exists a homeomorphism $h\colon S_A\to S_B$,
such that $f_B\circ h= h\circ f_A$.

The following Lemma \ref{p_a_r_cover} is stated in the finite-volume hyperbolic case,
which suffices for our reduction.
As kindly pointed out by the anonymous referee,
a similar proof also works for compact $3$--manifolds with empty or tori boundary,
by invoking B.~Zimmermann's result on finite group actions
on Haken $3$--manifolds, (see \cite[Corollary 4.2]{Zimmermann}).

\begin{lemma}\label{p_a_r_cover}
	Let $M$ be a finite-volume hyperbolic $3$--manifold.
	The conclusion of Theorem \ref{profinite_almost_rigidity_hyperbolic} holds true for $M$
	if it holds true for some regular finite cover of $M$.
\end{lemma}

\begin{proof}
	Suppose that we have obtained a finite collection $\mathcal{M}'$ as asserted
	for some finite cover $M'$ of $M$. If $N$ is any connected compact $3$--manifold
	with $\pi_1(N)$ profinitely isomorphic to $\pi_1(M)$,
	then there is some regular finite cover $N'$ of $N$,
	such that $\pi_1(N')$ is profinitely isomorphic $\pi_1(M')$.
	By assumption, there is some $X'\in\mathcal{M}'$,
	such that the interior of $N'$ is homeomorphic to $X'$ with finitely many punctures.
	Note that the boundary spheres of $N'$ compactifying those punctures
	must project homeomorphically onto boundary spheres of $N$.
	In fact, otherwise $\pi_1(N)$ would contain torsion,
	but $\widehat{\pi_1(N)}\cong\widehat{\pi_1(M)}$ is torsion-free,
	since $\pi_1(M)$ is a residually finite, good group of finite cohomological dimension
	\cite[Corollary 7.6]{Reid_discrete}.
	So, possibly after capping of the boundary spheres of $N$ and $N'$,
	we may assume that the interior of $N'$ is homeomorphic to $X'$.
	%By the Mostow--Prasad rigidity theorem, 
	%the deck transformation group action on $N'$
	%is topologically equivalent to an isometric group action on $X'$,
	%(see \cite[Theorem 1.7.1]{AFW_book_group}). 
	%This means $N$ is homeomorphic to an isometric quotient of $X'$.
	%Therefore, we can form $\mathcal{M}$ by taking 
	%all the finite-volume hyperbolic $3$--manifolds $X$,
	%such that $X$ is some isometric quotient of some $X'\in\mathcal{M}'$.
	By the Perelman--Thurston geometrization theorem,
	we infer that $N$ is finite-volume hyperbolic,
	then it follows from the Mostow--Prasad rigidity theorem
	that the finite cover $N'$ is isometric to $X'$,
	(see \cite[Theorems 1.7.1 and 1.7.5]{AFW_book_group});
	one may also infer the hyperbolicity of $N$ 
	as a particular case of \cite{WZ_geometry}.
	This means that $N$ is an isometric quotient of $N'\cong X'$.
	Therefore, we can form $\mathcal{M}$ by taking 
	all the finite-volume hyperbolic $3$--manifolds $X$,
	such that $X$ is some isometric quotient of some $X'\in\mathcal{M}'$.
	Since the isometry group of any finite-volume hyperbolic $3$--manifold is finite,
	$\mathcal{M}$ is a finite collection as desired.
\end{proof}

\begin{lemma}\label{p_a_r_closed}
	The conclusion of Theorem \ref{profinite_almost_rigidity_hyperbolic} holds if $M$ is closed.
\end{lemma}

\begin{proof}
		With Lemma \ref{p_a_r_cover}, we may assume without loss of generality that $M$ is 
		orientable, and fibers over a circle (Theorem \ref{quasi-fibered}).
		Identify $M$ with the mapping torus of some pseudo-Anosov automorphism $f_M\colon S_M\to S_M$.
		Possibly after cap-off,	we may assume $N$ has no boundary spheres.
		Then the assumption $\widehat{\pi_1(N)}\cong\widehat{\pi_1(M)}$ implies that
		$N$ is homeomorphic to an orientable closed hyperbolic $3$--manifold \cite{WZ_decomposition}.
		By Corollaries \ref{fiber_surface_correspondence} and \ref{pA_features_correspondence},
		$N$ is homeomorphic to the mapping torus of some pseudo-Anosov automorphism $f_N\colon S_N\to S_N$,
		with $\chi(S_N)=\chi(S_M)$ and $\lambda(f_N)=\lambda(f_M)$.
		Therefore, Theorem \ref{pA_finiteness} implies that 
		$(S_M,f_M)$ determines the homeomorphism type of
		$N$ up to finitely many possibilities after cap-off.
\end{proof}

\begin{lemma}\label{p_a_r_cusped}
	The conclusion of Theorem \ref{profinite_almost_rigidity_hyperbolic} holds if $M$ is cusped.
\end{lemma}

\begin{proof}
	The same proof as with Lemma \ref{p_a_r_closed} will work,
	once we show $\chi(S_N)=\chi(S_M)$ and $\lambda(f_N)=\lambda(f_M)$.
	To this end,
	it suffices to assume that $M$ and $N$ are the mapping tori of pseudo-Anosov automorphisms
	$f_M\colon S_M\to S_M$ and $f_N\colon S_N\to S_N$ 
	of orientable finite-type punctured surfaces,
	such that a profinite group isomorphism $\Psi\colon\widehat{\pi_1(M)}\to \widehat{\pi_1(N)}$
	maps the closed normal subgroup $\widehat{\pi_1(S_M)}$ isomorphically onto $\widehat{\pi_1(S_N)}$.
	
	We construct a pseudo-Anosov automorphism $f^\dagger_M\colon S^\dagger_M\to S^\dagger_M$
	as follows. The surface $S^\dagger_M$ is a connected finite cover of $S_M$
	associated to a characteristic finite-index subgroup $K^\dagger_M$ of $\pi_1(S_M)$,
	such that no simple closed curves of $S_M$ lift to $S^\dagger_M$.
	For example, we take an intermediate cover $S'_M$ 
	associated to the kernel of the mod 2 abelianization of $\pi_1(S_M)$,
	and take $S^\dagger_M$ associated to 
	the kernel of the mod 2 abelianization of $\pi_1(S'_M)$.
	The automorphism $f^\dagger_M\colon S^\dagger_M\to S^\dagger_M$
	is an iterate of a lift of $f_M$,
	such that $f^\dagger_M$ fixes every puncture of $S^\dagger_M$.
	For example, 
	%if we take the iteration power to be divisible by $[\pi_1(S_M):K^\dagger_M]$,
	%then $f^\dagger_M$ will not depend on the choice of the lift;
	%if we further require the iteration power to be 
	we take the iteration power to be
	$D=[\pi_1(S_M):K^\dagger_M]\times \prod_{g} (2-2g-\chi(S_M))!$,
	where $g$ ranges over $0,1,\cdots,\lfloor(2-\chi(S_M))/2\rfloor$,
	then $f^\dagger_M$ does not depend on the choice of the lift,
	and $f^\dagger_M$ fixes every puncture of $S^\dagger_M$,
	regardless of the particular genus or puncture number of $S_M$.
	Moreover, 
	we obtain a pseudo-Anosov automorphism $\bar{f}^\dagger_M\colon\bar{S}^\dagger_M\to \bar{S}^\dagger_M$
	of an orientable closed surface $\bar{S}^\dagger_M$	by filling the punctures of $S^\dagger_M$ with points.
	Since the stable/unstable measured foliations on $S_M$ pull back to $S^\dagger_M$,
	our construction ensures that there are no $1$--prong singularities
	at any punctures of $S^\dagger_M$,
	so $\lambda(\bar{f}^\dagger_M)=\lambda(f^\dagger_M)=\lambda(f_M)^D$.
	%which equals the $D$--th power of $\lambda(f_M)$.
	
	We construct similarly $f^\dagger_N\colon S^\dagger_N\to S^\dagger_N$
	with respect to $N$, using the $\Psi$--corresponding characteristic finite-index subgroup $K^\dagger_N$ of $\pi_1(N)$
	and the same iteration power $D$.
	So $\lambda(\bar{f}^\dagger_N)=\lambda(f^\dagger_N)=\lambda(f_N)^D$.
	The isomorphism $\Psi$ determines an isomorphism $\Psi^\dagger$ between 
	the profinite completions of 
	the cusped mapping-torus groups
	$\pi_1(M^\dagger)$ and $\pi_1(N^\dagger)$, 
	associated to $f^\dagger_M$ and $f^\dagger_N$ respectively.
	Moreover, $\Psi^\dagger$ descends to an isomorphism $\bar{\Psi}^\dagger$	
	between the profinite completions
	of the closed mapping-torus groups
	$\pi_1(\bar{M}^\dagger)$ and $\pi_1(\bar{N}^\dagger)$,
	associated to $\bar{f}^\dagger_M$ and $\bar{f}^\dagger_N$ respectively.
	This follows from the fiber surface correspondence (Corollary \ref{fiber_surface_correspondence})
	and the profinite cusp correspondence \cite[Proposition 3.1]{WZ_decomposition},
	(as aforementioned in the proof of Lemma \ref{profinite_isomorphism_hyperbolic_robust}).
	
	In fact, $S_M^\dagger$ and $S_N^\dagger$ has bijectively corresponding punctures, 
	since $M^\dagger$ and $N^\dagger$ has bijectively corresponding peripheral tori,
	as witnessed by $\Psi^\dagger$ \cite[Proposition 3.1]{WZ_decomposition}.
	More precisely, for every peripheral torus $T^\dagger_M$ of $M^\dagger$
	there is a unique peripheral torus $T^\dagger_N$ of $N^\dagger$,
	such that $\Psi^\dagger$ maps the closure of $\pi_1(T^\dagger_M)$ onto
	a conjugate of the closure of $\pi_1(T^\dagger_N)$.
	The intersection of $\mathrm{clos}(\pi_1(T^\dagger_M))$ with 
	the normal subgroup $\mathrm{clos}(\pi_1(S^\dagger_M))$ 
	in the profinite completion of $\pi_1(M^\dagger)$
	is exactly $\mathrm{clos}(\pi_1(c^\dagger_M))$, 
	where $c^\dagger_M$ is the unique slope on $T^\dagger$ 
	which is freely homotopic in $M^\dagger$ to a peripheral simple closed curve on $S^\dagger_M$.
	The profinite completion of $\pi_1(\bar{M}^\dagger)$
	is the quotient of the profinite completion of $\pi_1(M^\dagger)$
	by the smallest closed normal subgroup 
	which contains all the peripheral subgroups
	$\mathrm{clos}(\pi_1(c^\dagger_M))$ as above.
	Since $\Psi^\dagger$ carries the similar description to 
	the profinite completion of $\pi_1(\bar{N}^\dagger)$,
	it desends to an isomorphism $\bar{\Psi}^\dagger$
	between the profinite completions of $\pi_1(\bar{M}^\dagger)$ and $\pi_1(\bar{N}^\dagger)$.
	
	Since $\pi_1(S_M)$ and $\pi_1(S_N)$ are profinitely isomorphic,
	we obtain $\chi(S_M)=\chi(S_N)$.	
	Applying Corollary \ref{pA_features_correspondence} to $\bar{\Psi}^\dagger$,
	we obtain $\lambda(\bar{f}^\dagger_M)=\lambda(\bar{f}^\dagger_N)$,
	and hence, $\lambda(f_M)=\lambda(f_N)$.	
	Supplied with these equalities, 
	the argument with the closed case (Lemma \ref{p_a_r_closed})
	works for the cusped case (Lemma \ref{p_a_r_cusped}), as desired.
\end{proof}

Theorem \ref{profinite_almost_rigidity_hyperbolic} follows from Lemmas \ref{p_a_r_closed} and \ref{p_a_r_cusped}.

\bibliographystyle{amsalpha}

%\bibliography{../refs}

\end{document}